\documentclass[a4paper,12pt]{amsart}
\usepackage{fullpage}
\usepackage{latexsym,amsfonts,amssymb,amsmath,amsthm, color}
\usepackage{mathtools}
\usepackage{mathrsfs}

\usepackage[top=1in, bottom=1in, left=1in, right=1in]{geometry}
\usepackage{graphicx}
\usepackage{epstopdf}
\usepackage{verbatim}
\usepackage{enumerate}

\newcommand{\Mod}[1]{\ (\mathrm{mod}\ #1)}

\newcommand{\lt}{\left}
\newcommand{\rt}{\right}
\newcommand{\bsm}{\lt(\begin{smallmatrix}} 
\newcommand{\esm}{\end{smallmatrix}\rt)}

\newtheorem{thm}{Theorem}[section]
\newtheorem{corollary}[thm]{Corollary}
\newtheorem{prop}[thm]{Proposition}
\newtheorem{lem}[thm]{Lemma}
\newtheorem{lemma}[thm]{Lemma}
\theoremstyle{remark}
\newtheorem{rem}[thm]{Remark}

\newcommand{\sumstar}{\sideset{}{^*}\sum}
\newcommand{\sumb}{\sideset{}{^\flat}\sum}

\newcommand{\sumbq}{\sideset{}{^\flat}\sum_{\chi \Mod q}}

\newcommand{\sumtwo}{\operatorname*{\sum\sum}}
\newcommand{\sumtwodee}{\operatorname*{\sideset{}{^d}\sum\sideset{}{^d}\sum}}
\newcommand{\sumtwostar}{\operatorname*{\sideset{}{^*}\sum\sideset{}{^*}\sum}}
\newcommand{\sumthree}{\operatorname*{\sum\sum\sum}}
\newcommand{\sumfour}{\operatorname*{\sum\sum\sum\sum}}

\newcommand{\sumd}{\sideset{}{^d}\sum}

\newcommand{\tRe}{\textup{Re }}
\newcommand{\tIm}{\textup{Im }}
\newcommand{\bfrac}[2]{\left(\frac{#1}{#2}\right)}
\newcommand{\al}{\boldsymbol\alpha}
\newcommand{\be}{\boldsymbol \beta}

\newcommand{\chiq}{\chi \Mod{q}}

\newcommand{\cb}{\overline{\chi}}

\newcommand{\W}{\mathcal W}
\newcommand{\D}{\mathcal D (\Psi, Q; \boldsymbol \alpha, \boldsymbol \beta)}

\newcommand{\Hc}{\mathcal H}
\newcommand{\B}{\mathcal B}

\newcommand{\R}{{\mathrm{Re}}}

\newcommand{\h}{\frac{1}{2}}

\newcommand{\phib}{\phi^\flat(q)}

\newcommand{\Wt}{\widetilde{\mathcal{W}}}
\newcommand{\psit}{\widetilde{\Psi}}

\newcommand{\Z}{\mathcal{Z}}

\newcommand{\A}{\mathcal{A}}
\newcommand{\Vcal}{\mathcal{V}}
\newcommand{\Q}{\mathcal{Q}}

\newcommand{\calS}{\mathcal{S}}
\newcommand{\tQ}{\widetilde{\mathcal{Q}}}

\newcommand{\fr}{\mathfrak{r}}

\newcommand{\m}{\mathfrak{m}}
\newcommand{\n}{\mathfrak{n}}
\newcommand{\ex}{\mathrm{e}}

\begin{document}
\title{The sixth moment of Dirichlet $L$-functions at the central point}

\author{Vorrapan Chandee}
\address{Mathematics Department, 138 Cardwell Hall Manhattan, KS 66506 USA}

\email{chandee@ksu.edu}

\author{Xiannan Li}
\address{Mathematics Department, 138 Cardwell Hall Manhattan, KS 66506 USA}
\email{xiannan@ksu.edu}

\author{Kaisa Matom\"aki}
\address{Department of Mathematics and Statistics, University of Turku, 20014 Turku, Finland}
\email{ksmato@utu.fi}

\author{Maksym Radziwi\l\l}
\address{Department of Mathematics, Lunt Hall, 2033 Sheridan Road, Evanston, IL, 60208, USA}
\email{maksym.radziwill@gmail.com}


\begin{abstract}
  In 1970, Huxley obtained a sharp upper bound for the sixth moment of Dirichlet $L$-functions at the central point,  averaged over primitive characters $\chi$ modulo $q$ and all moduli $q \leq Q$. In 2007, as an application of their ``asymptotic large sieve'', Conrey, Iwaniec and Soundararajan showed that when an additional short $t$-averaging is introduced into the problem, an asymptotic can be obtained. In this paper we show that this extraneous averaging can be removed, and we thus obtain an asymptotic for the original moment problem considered by Huxley.

 The main new difficulty in our work is the appearance of certain challenging ``unbalanced'' sums that arise as soon as the $t$-aspect averaging is removed.
\end{abstract}

\maketitle

\section{Introduction} 

Moments of $L$-functions have been studied for application to arithmetic objects as well as for their own interest.  Classically, the first moments  studied were those of the Riemann zeta function, which are averages of the form
$$I_k(T) := \int_0^T |\zeta(\tfrac{1}{2} + it)|^{2k}dt,
$$ where as usual $\zeta(s)$ denotes the Riemann zeta function. We refer to $I_k(T)$ as the $2k$-th moment of the Riemann $\zeta$ function. Here, asymptotic formulas were proven for $k = 1$ by Hardy and Littlewood and for $k=2$ by Ingham (see e.g.~\cite[Chapter VII]{Ti}).  Despite extensive further work, including various refinements of the result of Ingham, no such result is available for any other values of $k$.

A well known conjecture states that  $I_k(T) \sim c_k T(\log T)^{k^2}$ for constants $c_k$ depending on $k$.  The values of $c_k$ remained mysterious for general $k$ until the work of Keating and Snaith \cite{KS} which related these moments to similar statistics of random matrices, thus providing precise conjectures for $c_k$.  Based on heuristics for shifted divisor sums, Conrey and Ghosh derived a conjecture in the case $k=3$ \cite{CGh} and Conrey and Gonek derived a conjecture in the case $k=4$ \cite{CGo}.  Further conjectures including lower order terms, and for other symmetry groups are available from the work of Conrey, Farmer, Keating, Rubinstein and Snaith \cite{CFKRS} as well as from the work of Diaconu, Goldfeld and Hoffstein \cite{DGH}.  Recently, Conrey and Keating have produced an alternative method of deriving these conjectures through more arithmetic considerations (i.e. with the circle method as basis) \cite{CK}.

Moments of other families of $L$-functions have also been studied. Again, asymptotics are only available for small values of $k$, while large values appear out of reach. However in certain families it is possible to reach higher values of $k$ than for the Riemann $\zeta$-function. For example, in 1970 as an application of the large sieve, Huxley \cite{Huxley} obtained an upper bound for the sixth and eight moment of Dirichlet $L$-functions,
$$
\sum_{q \leq Q}\;\; \sumbq |L(\tfrac 12, \chi)|^{2k} \ll Q^{2} (\log Q)^{k^{2}} \ , \ k \in \{1,2,3,4\}.
$$
where the superscript $\flat$ means that we are only summing over primitive even characters\footnote{The restriction to even characters is for technical convenience, and the analogous result may be derived for odd characters using the same method.}.
This family is also unitary so one conjectures that Huxley's upper bound is sharp.

Huxley's upper bound can be easily turned into an asymptotic when $k \in \{1,2\}$, in fact these cases do not even require the additional averaging over $q$ (see the breakthrough work of \cite{Young}).
Unfortunately, Huxley's upper bound for the sixth moment has resisted attempts at being improved into an asymptotic. The closest result so far comes from the work of Conrey, Iwaniec and Soundararajan \cite{CIS} in which an asymptotic formula is obtained provided that an additional short averaging in the $t$-aspect is included, namely,
\begin{equation} \label{eq:aver}
\sum_{q \leq Q}\;\; \sumbq \int_{\mathbb{R}} |L(\tfrac 12 + it, \chi)|^{6} \phi(t) dt
\end{equation}
and $\phi$ is a fixed smooth function with rapid decay at infinity. A similar result was recently obtained by the authors in the case of the eighth moment \cite{CLMR}.

Despite being short (essentially of length $\approx 1$) the $t$-averaging in \eqref{eq:aver} is significant. It eliminates from the problem so-called unbalanced sums, that is sums of $d_{3}(n) d_{3}(m) \chi(m) \overline{\chi}(n)$ with $m$ much larger than $n$.  In our main result we are able to successfully handle the contribution of such sums. Thus we obtain an asymptotic for the sixth moment without any $t$-averaging, turning Huxley's upper bound for the sixth moment into an asymptotic.

\begin{corollary}\label{cor:main}
	As $Q \rightarrow \infty$,
	\begin{align*}
	\sum_{q\leq Q} \;\; &\sumb_{\chiq}  \left| L\big( \tfrac{1}{2} , \chi \big)\right|^6 \sim 42 a_3\sum_{q\leq Q}  \prod_{p|q} \frac{\left( 1 - \frac{1}{p}\right)^5}{\left( 1 + \frac{4}{p} + \frac{1}{p^2}\right)} \phib \frac{(\log q)^{9}}{9!},
	\end{align*}where
    \begin{align*} 
    a_3 := \prod_p \left(1-\frac{1}{p^4}\right)\left(1+\frac{4}{p} + \frac{1}{p^2}\right),
    \end{align*}
and $\phi^{\flat}(q)$ counts the number of primitive even characters with modulus $q$.
\end{corollary}
To deal with the new unbalanced sums that arise we will need a variety of methods, notably the spectral theory of automorphic forms and bounds of Deligne for hyper-Kloosterman sums. This is in juxtaposition to \cite{CIS} which exploits the elementary complementary divisor trick using more classical complex analytic tools.

We also note that Corollary \ref{cor:main} is consistent with the conjectures in \cite{CFKRS}. Similarly to~\cite{CIS}, we in fact prove a more general and stronger result about the sixth moment with shifts, with a power saving error term, which we state in \S \ref{sec:shiftedmoments}.


\subsection{Outline of the proof}

The problem is roughly equivalent to obtaining an estimation (in all ranges of $N$ and $M$) of
$$
Q^{-2} \sum_{q \sim Q}\; \sum_{\psi \Mod{q}} \sum_{n \sim N} \sum_{m \sim M} \frac{d_{3}(n)\psi(n)d_{3}(m)\overline{\psi}(m)}{\sqrt{n m}}
$$
that is precise within $Q^{- \varepsilon}$ for some $\varepsilon > 0$. Thus, as is usual in moment problems, we want to slightly beat square-root cancellation in the individual $n,m$ sums by exploiting the averaging over the family. Furthermore, the functional equation allows us to restrict our attention to $N M \leq Q^{3}$. We restrict our discussion to the hardest range, namely when $M N \asymp Q^{3}$, and we also assume without loss of generality that $M > N$, and thus $M > Q^{3/2}$. 


First let us consider the range $M \leq Q^{2 - \varepsilon}$. This range is similar to the work of Conrey, Iwaniec, Soundararajan \cite{CIS}. 
Using orthogonality of characters we can think of the sum as essentially
\begin{equation} \label{eq:bigOB}
Q^{-1} \sum_{q \sim Q} \; \sumtwo_{{\substack{n \sim N, \; m \sim M \\ n \equiv m \Mod{q}}}} \frac{d_{3}(n)d_{3}(m)}{\sqrt{n m}}
\end{equation}
with various ``main terms'' subtracted.  Write $n - m = e q$, and notice that $e \asymp M / Q$ is smaller than $Q$ if $M < Q^{2 - \varepsilon}$. It is thus beneficial to re-write the congruence condition $n \equiv m \Mod{q}$ as $n \equiv m \Mod{e}$
and replace each occurence of $q$ by $(n - m) / e$.
This allows us to re-write \eqref{eq:bigOB} as
$$
Q^{-1} \sum_{e \sim M / Q} \; \sumtwo_{{\substack{n \sim N, \; m \sim M \\ n \equiv m \Mod{e}}}} \frac{d_{3}(n)d_{3}(m)}{\sqrt{n m}}.
$$
Relating back this sum to primitive characters we obtain another sequence of ``main terms'', most (but not all) of which
cancel out with the main terms subtracted from \eqref{eq:bigOB}. The remaining error term is controlled by
$$
\frac{M}{Q^{2}} \sum_{e \sim M / Q}\; \sum_{\psi \Mod{e}} \; \sum_{n \sim N} \sum_{m \sim M} \frac{d_{3}(n)\psi(n)d_{3}(m)\overline{\psi}(m)}{\sqrt{n m}}.
$$
This is a mirror-problem of the problem we started with, but in different ranges.
We do not anymore need more than square-root cancellation in sums over $m$ and $n$ and can just apply the large sieve. This leads to a bound that is $\ll M Q^{{\varepsilon}} / Q^{2}$ which is sufficient as long as $M$ is slightly smaller than $Q^{2}$.

Let us next dispose of the extreme range in which $M > Q^{5/2 + \varepsilon}.$
In this range, the functional equation converts
$$
\sum_{m \sim M} \frac{d_{3}(m)\psi(m)}{\sqrt{m}}
$$
into
$$
\varepsilon_{{\psi}}^{3} \sum_{m \sim Q^{3} / M} \frac{d_{3}(m) \overline{\psi}(m)}{\sqrt{m}},
$$
where $\varepsilon_{\psi}$ is the root-number of $L(s, \psi)$. Furthermore, upon averaging over $\psi \Mod{q}$,
we note that
\begin{equation} \label{eq:asd}
\sum_{{\psi \Mod{q}}} \varepsilon_{\psi}^{3} \cdot \psi(n)
\end{equation}
is a hyper-Kloosterman sum. Thus, for generic $q$ and $n$, \eqref{eq:asd} is bounded by $\ll Q^{1/2 + \varepsilon}$.
The remaining sums over $n \sim N, q \sim Q$ and $m \sim Q^{3} / M$ are bounded trivially, and we are left with a final
bound of (up to factors of $Q^{\varepsilon}$)
$$
\ll Q^{-2} \cdot Q \sqrt{N \cdot \frac{Q^{3}}{M}} \cdot \sqrt{Q} \ll \frac{Q^{5/2}}{M}
$$
which is sufficient if $M$ is slightly larger than $Q^{5/2}$.

Thus it remains to handle the range $Q^{2 - \varepsilon} < M < Q^{5/2 + \varepsilon}$. In this range we open up the definition of $d_{3}(m)$ and thus we aim to estimate
$$
Q^{-2} \sum_{q \sim Q} \; \sum_{\psi \Mod{q}} \; \sumfour_{\substack{n \sim N \\ e \sim E, f \sim F, g \sim G}} \frac{d_{3}(n)\overline{\psi}(n) \psi(e f g)}{\sqrt{n e f g}}
$$
with $E F G \asymp M$ and $E > F > G$.
We now apply Poisson summation on the two longest variables $e$ and $f$ to get
$$
\varepsilon_{\psi}^{2} \sumfour_{\substack{n \sim N \\ e \sim Q / E, f \sim Q / F, g \sim G}} \frac{d_{3}(n) \overline{\psi}(n e f) \psi(g)}{\sqrt{n e f g}}.
$$
Executing the sum over $\psi \Mod{q}$ converts
$$
\sum_{\psi \Mod{q}} \varepsilon_{\psi}^{2} \cdot \overline{\psi}(n e f) \psi(g)
$$
into a Kloosterman sum $S(n e f, \overline{g}; q)$. We now use Kuznetsov in $q$ to get an average over forms of level $g \asymp G$ on the spectral side. Note that
$$
\frac{\sqrt{n e f}}{q \sqrt{g}} \asymp \frac{\sqrt{N (Q / E) (Q / F)}}{Q \sqrt{G}} = \frac{\sqrt{N}}{\sqrt{E F G}} =: \frac{1}{X}
$$
and therefore the dual sum over the spectrum is morally of length $\ll 1 + 1/X \ll 1$. For simplicity, we neglect the contribution of the continuous and holomorphic spectrum and get that
$$
\sum_{q \sim Q} S(n e f, \overline{g}; q) \approx \frac{Q}{\sqrt{g}} \sum_{\substack{\phi_{j} \text{ level } g \\ \text{eigenvalue} \ll 1}} X^{2 i \kappa_{j}}\lambda_{j}(n e f),
$$
where $\tfrac 14 + t_{j}^{2} = (\tfrac 12 + i \kappa_{j}) \cdot (\tfrac 12 - i \kappa_{j})$ is the eigenvalue of the form $\phi_{j}$, and where we choose $\kappa_{j}$ so that if it is imaginary then $i \kappa_{j} > 0$. Thus we are left with estimating
\begin{equation} \label{eq:largesieve}
\frac{1}{Q} \sum_{g \sim G} \frac{1}{g} \sum_{\substack{\phi_{j} \text{ level } g \\ \text{eigenvalue} \ll 1}} X^{2 i \kappa_{j}} \sum_{n \sim N} \frac{\lambda_{j}(n) d_{3}(n)}{\sqrt{n}} \sum_{e \sim Q / E} \frac{\lambda_{j}(e)}{\sqrt{e}} \sum_{f \sim Q / F} \frac{\lambda_{j}(f)}{\sqrt{f}}.
\end{equation}
If we assumed the Ramanujan conjecture and the Lindel\"of hypothesis, the above would be $\ll G Q^{\varepsilon}/ Q$ which would be sufficient since $G \leq Q^{5/6}$ when $M \leq Q^{5/2}$.

In order to bound \eqref{eq:largesieve} unconditionally we will use a refinement of the spectral large sieve of Deshouillers-Iwaniec. Ultimately the success of our argument will crucially use that the best known bound towards the Ramanujan conjecture, due to Kim-Sarnak \cite{KimS}, gives $\kappa_j \leq \tfrac {7}{64} = \tfrac{1}{7} - \tfrac{15}{448}$ and thus strictly less than $\tfrac 17$.

Using the Cauchy-Schwarz inequality, we bound \eqref{eq:largesieve} by
\begin{align} \label{eq:tob}
  \frac{1}{Q G} \Big ( \sum_{\substack{g \sim G \\ \phi_{j} \text{ eigenvalue } \ll 1}} |Y^{2 i \kappa_{j}}| \Big | \sum_{k \asymp Q^{2} / E F} \frac{\alpha(k)}{\sqrt{k}} \lambda_{j}(k) \Big |^{2} \Big )^{1/2} \cdot
  \Big ( \sum_{\substack{g \sim G \\ \phi_{j} \text{ eigenvalue } \ll 1}} |Z^{2i\kappa_{j}}| \Big | \sum_{n \sim N} \frac{d_{3}(n)}{\sqrt{n}} \lambda_{j}(n) \Big |^{2} \Big )^{1/2}
  \end{align}
  for any choice of $Y, Z \geq 1$ such that $Y Z = X^{2} = E F G / N$ and where $\alpha(k)$ is a coefficent obtained from
  grouping the variable $e$ and $f$ together. We show that
  \eqref{eq:tob} is
  \begin{equation} \label{eq:tog}
  \ll \frac{1}{Q} \cdot \Big ( (1 + Y_{1}^{2 \theta}) \cdot (1 + Z_{1}^{2 \theta}) \cdot \Big ( G + \frac{Q^{2}}{E F} \Big ) \cdot (G + N) \Big )^{1/2},
\end{equation}
where $\theta$ is the best current bound towards the Ramanujan conjecture, and
  $$
  Y_{1} = \frac{Y Q^{2}}{Q^{2} + E F G} \ll \frac{Y Q^{2}}{E F G}
\ , \ Z_{1} = \frac{Z N}{N + G} \ , \ Y_{1}Z_{1} \ll \frac{Q^{2}}{N + G}.
$$
We can pick $Y$ and $Z$ appropriately so that $Y_{1}, Z_{1} \geq 1$. Notice that in \eqref{eq:tog} we use a large sieve bound of the form $ G (G + N) \| \boldsymbol{\alpha} \|_{2}^{2}$ instead of the conjecturally optimal $(G^{2} + N) \| \boldsymbol{\alpha} \|_{2}^{2}$. 

We now further bound \eqref{eq:tog} by
\begin{equation}
  \label{eq:intupp}
\frac{\sqrt{G} \cdot (G + N)^{1/2}}{Q} \cdot \Big ( \frac{Q^{2}}{N + G} \Big )^{\theta}.
\end{equation}

Assuming the Ramanujan conjecture, one may put $\theta = 0$, and then this bound achieves the same maximum of $Q^{-1/6}$ at $N = Q^{1/2}, E=F=G=Q^{5/6}$ and $N = Q, E=F=G=Q^{2/3}$.  This is purely coincidental; in particular note that for $G \ge N$, the bound $G(G+N) \asymp G^2$ is essentially optimal, but this is not the case when $G$ is smaller.  For $\theta>0$, the bound~\eqref{eq:intupp} is largest for $N = Q^{1/2}, G = Q^{5/6}$.  Here, we obtain a final bound that is
$$
Q^{-1/6 + (2 - 5/6) \cdot \theta}.
$$
This is $\ll Q^{-\varepsilon}$ for some $\varepsilon > 0$ provided that $\theta < \tfrac 17$. Luckily
the Kim-Sarnak \cite{KimS} bound gives $\theta \leq \tfrac {7}{64} < \tfrac 17$ and this suffices to conclude the proof.

There is one additional difficulty that we did not mention in this outline. In the
case $Q^{2} < M$ we also have to show that certain main terms, similar to the
main terms that we mentioned in the case $M < Q^{2}$, do not contribute. In the range
$Q^{2} < M < Q^{5/2}$ this requires an intricate calculation followed by an application of the large sieve inequality which also appears to be new, in this context.

\section{Shifted moments}\label{sec:shiftedmoments}

We start by recalling the basic setup: we let $\chiq$ be a primitive even Dirichlet character, and let (for $\tRe s > 1$),
$$ L(s, \chi) = \sum_{n=1}^\infty \frac{\chi(n)}{n^s} = \prod_p \left( 1 - \frac{\chi(p)}{p^s}\right)^{-1} $$
be the Dirichlet $L$-function associated to it.  Then the completed $L$-function $\Lambda(s, \chi)$ defined by
$$ \Lambda\big( \tfrac{1}{2} + s, \chi \big)  := \left(\frac{q}{\pi}\right)^{s/2} \Gamma \left(\tfrac{1}{4} + \tfrac{s}{2}\right) L\big( \tfrac{1}{2} + s, \chi\big)
$$
satisfies the functional equation
\[
\Lambda\big( \tfrac{1}{2} + s, \chi \big) = \epsilon_\chi \Lambda\big( \tfrac{1}{2} - s, \overline{\chi}\big),
\]
where $|\epsilon_\chi| = 1$.

We will mostly follow the notation in \cite{CIS}. Let $\al = (\alpha_1, \alpha_2, \alpha_3)$ and $\be = (\beta_1, \beta_2, \beta_3)$.  For convenience, we also write $\alpha_{3 + j} = \beta_j$ for $j = 1, 2, 3$. Moreover let $S_{6}$ be the permutation group on six elements. For $\pi \in S_6$, define
$$\pi(\al, \be) = (\pi(\al), \pi(\be)) = (\alpha_{\pi(1)}, ..., \alpha_{\pi(6)}),$$ where we take $\pi(\al)$ as the first three coordinates of $\pi(\al, \be)$ and $\pi(\be)$ as the last three coordinates of $\pi(\al, \be)$. 

Now let
$$ \Lambda(s, \chi; \al, \be) := \prod_{j = 1}^3 \Lambda(s+\alpha_j, \chi) \Lambda(s-\beta_j, \overline{\chi})$$
and 
$$ \Lambda(\chi, \al, \be) := \Lambda\left( \frac 12, \chi; \al, \be \right).$$

Further let
\begin{equation}
\label{eq:Gdef}
G(s, \al, \be) := \prod_{i = 1}^3 \Gamma\left(\frac{s}{2} + \frac{\alpha_i}{2}\right)\Gamma\left(\frac{s}{2}-\frac{\beta_i}{2}\right),
\end{equation}
so that
$$ \Lambda\left(\tfrac{1}{2}, \chi ; \al, \be \right) = \bfrac{q}{\pi}^{\delta(\al, \be)} G\left(\frac 12, \al, \be\right) \prod_{i=1}^3 L\left( \frac{1}{2} + \alpha_i, \chi\right)L\left(\frac{1}{2} - \beta_i,  \cb \right),$$
where
\begin{align*}
\delta(\al, \be) := \frac{1}{2} \sum_{j=1}^3 (\alpha_j - \beta_j).
\end{align*}
As usual, for $ \tRe(s)$ sufficiently large, we may write
\begin{equation}\label{eqn:Dirichletseries}
\prod_{i=1}^3 L\left( s + \alpha_i, \chi\right)L\left(s - \beta_i,  \cb \right) = \sumtwo_{m, n \geq 1} \frac{\sigma (m; \al) \sigma (n; -\be)}{m^s n^s} \chi(m) \cb(n),
\end{equation}
where the coefficients are
\begin{align*}
\sigma(m; \al)  := \sum_{m=m_1m_2m_3} m_1^{-\alpha_1}m_2^{-\alpha_2}m_3^{-\alpha_3}
\end{align*}and similarly for $\sigma(n; -\be)$.

Our final result will involve certain arithmetic factors, which we define below.
As is standard, we expect an arithmetic factor resulting from the diagonal term coming from $m=n$ in \eqref{eqn:Dirichletseries}.  Let
\begin{equation}
  \label{eqn:sumoverm} \sum_{\substack{m=1\\ (m, q) = 1}}^\infty \frac{\sigma (m; \al) \sigma (m; -\be)}{m^{2s}} = \prod_{p\nmid q} \B_p(s; \al, \be),
  \end{equation}
where 
$$ \B_p(s; \al, \be) := \sum_{r = 0}^{\infty} \frac{\sigma(p^r; \al)\sigma(p^r; -\be)}{p^{2rs}}.$$

Further, for $\zeta_p(x) = (1-p^{-x})^{-1}$, we let
\begin{align*}
\Z_p(s; \al, \be) = \prod_{i, j = 1}^3 \zeta_p(2s + \alpha_i - \beta_j), \ \ \ \ \ \textup{and} \ \ \ \ \Z(s; \al, \be) = \prod_{i, j = 1}^3 \Z_p(2s + \alpha_i - \beta_j).
\end{align*}
The sum $\B_p$  behaves similarly to $\Z_p$.  To be specific, the Euler product defined by
\begin{equation}\label{def:A}
\A(s; \al, \be) := \prod_p \B_p(s; \al, \be) \Z_p(s; \al, \be)^{-1},
\end{equation} will be absolutely convergent in a wider region. In particular, $\A(s; 0, 0)$ converges for $\tRe s > 1/4$.

Now, letting
\begin{equation}
  \label{eq:Bqdef}
\B_q(s; \al, \be) := \prod_{p|q} \B_p(s; \al, \be),
\end{equation}
we define
\begin{equation} \label{def:Qalbe}
\Q(q; \al, \be) = \bfrac{q}{\pi}^{\delta(\al, \be)} G\left( \frac 12; \al, \be\right) \frac{\A\left( \frac 12; \al, \be\right)\Z\left( \frac 12; \al, \be\right)}{\B_q\left( \frac 12; \al, \be\right)} ,
\end{equation} which corresponds to the diagonal contribution $m=n$.  We expect our final result to be symmetric under the action of $S_6$, while $\Q(q; \al, \be)$ is only symmetric under the action of $S_3 \times S_3$.  This motivates the definition of the symmetric version
\begin{equation*} 
\tQ(q; \al, \be) = \sum_{\pi \in S_6/(S_3 \times S_3)} \Q(q; \pi(\al), \pi(\be)).
\end{equation*}

The standard conjecture (see~\cite{CFKRS}) is that whenever the shifts are small, then, for any given $\varepsilon > 0$, 
\begin{align*}
\sumb_{\chi \Mod q} \Lambda(\chi; \al, \be) = \phib \tQ(q; \al, \be) (1+O_{\varepsilon}(q^{-1/2+\varepsilon})).
\end{align*}

We prove a version of this conjecture, with an additional average over $q$.  Specifically, we show the following theorem.

\begin{thm}\label{thm:main}
Let $Q\ge 3$, and let $\al, \be$ $3$-tuples satisfying $\alpha_i, \beta_i \ll \frac{1}{\log Q}$ and such that $\alpha_i \neq \beta_j$ for all $1 \leq i,  j \leq 3$. Then we have, for any smooth function $\Psi$ supported on $[1, 2]$,
\begin{align*}
\sum_q \Psi\bfrac{q}{Q} \sumbq \Lambda(\chi; \al, \be) =  \sum_q \Psi\bfrac{q}{Q} \phib \tQ(q; \al, \be) + O(Q^{2-\frac{11}{1196} + \varepsilon}).
\end{align*}
\end{thm}

\begin{rem}
The main term on the right hand side is of size $Q^2 (\log Q)^9$ so that we save a power of $11/1196-\varepsilon$ in the error term. We have not tried to optimize this saving, see for example Remark~\ref{rem:triviImpr}. 
\end{rem}

Corollary \ref{cor:main} quickly follows from Theorem \ref{thm:main} by letting the shifts $\alpha_i, \beta_i$ tend to $0$ (for details of a similar derivation, see~\cite{CFKRS}).

\section*{Notations and assumptions}

We shall throughout the paper assume the set-up of Theorem~\ref{thm:main}. In particular $Q\ge 3$, $\al, \be$ are $3$-tuples satisfying $\alpha_i, \beta_i \ll \frac{1}{\log Q}$ with $\alpha_i \neq \beta_j$ for all $1 \leq i,j \leq 3$ and $\Psi$ is a smooth function supported on $[1, 2]$. We will also denote by 
$$
\sumb_{\chi \Mod q}
$$
a sum over primitive even characters modulo $q$, and by 
$$
\sumd_{M, N}
$$
a sum over $M$ and $N$ running over positive powers of two. Finally given a smooth function $v$, we will denote by 
\begin{equation} \label{def:MellinV}
\widetilde{v}(s) := \int_{0}^{\infty} v(x) x^{s - 1} dx
\end{equation}
the Mellin transform of $v$. We denote by 
$$
\widehat{V}(x) := \int_{-\infty}^{\infty} V(\xi) \ex(- x \xi) d \xi
$$
the Fourier transform of $V$, where $\ex(x) = e^{2\pi i x}.$ We will also set $\mathbb{N} = \{1, 2, \ldots \}$.

Throughout the paper, $\varepsilon$ denotes a small positive real number. Moreover, $\delta_0$ and $\delta'$ are fixed positive constants to be chosen later. 

\section{Preliminary setup}
\subsection{Standard lemmas}
Here we state some standard results from the literature.  Let 
\begin{align*}
H(s; \al, \be) := \prod_{i, j=1}^3 \left(s^2 - \bfrac{\alpha_i-\beta_j}{2}^2\right)^3,
\end{align*}
and for $\xi, \eta, \mu >0$,
\begin{equation} \label{eqn:Walbe}
W_{\al, \be}(\xi, \eta; \mu) := \bfrac{\mu}{\pi}^{\delta(\al, \be)}\frac{1}{2\pi i} \int_{(1)} G \left(\frac 12 + s; \al, \be\right) H(s; \al, \be) \left( \frac{\xi\eta \pi^3}{\mu^3} \right)^{-s} \frac{ds}{s}.
\end{equation}
Finally, let
\begin{align*}
\Lambda_0(\chi; \al, \be) = \sumtwo_{m, n \geq 1} \frac{\sigma(m; \al) \sigma(n; -\be) \chi(m) \cb(n)}{\sqrt{mn}} W_{\al, \be}\left(m, n; q\right).
\end{align*}

The following lemma (see~\cite[Proposition 1]{CIS}) gives the approximate function equation for $\Lambda(\chi; \al, \be)$.
\begin{lem}\label{lem:approxfunc}
With notation as above,
\begin{align*}
H(0; \al, \be) \Lambda(\chi; \al, \be) = \Lambda_0(\chi; \al, \be) + \Lambda_0(\chi; \be, \al).
\end{align*}
\end{lem}

We will also find it convenient to have the following bound for $W_{\al, \be}$.
\begin{lemma} \label{lem:weightW} Let $W_{\al, \be}$ be defined in \eqref{eqn:Walbe}. For any non-negative integers $\ell_1, \ell_2, \ell_3$ and $\xi, \eta, \mu$,  
\begin{align*} \frac{d^{\ell_1}}{d\xi}\frac{d^{\ell_2}}{d\eta} \frac{d^{\ell_3}}{d\mu}W_{\al, \be}(\xi, \eta; \mu) \ll_{\ell_1, \ell_2, \ell_3} \frac{1}{\xi^{\ell_1} \eta^{\ell_2} \mu^{\ell_3}}\left( \frac{\mu}{\pi}\right)^{\R \, \delta(\al, \be)}\exp \left(-c_0 \left( \frac{\xi \eta}{\mu^3}\right)^{1/3} \right)
\end{align*}
for some constant $c_0>0$. 
\end{lemma}
\begin{proof} The proof follows closely the proof of Lemma 1 in \cite{CIS}. In particular, we take derivatives and move the contour of integration in the definition of $W_{\al, \be}$ to the line $\Re s = \left ( \frac{\xi \eta}{\mu^3}\right )^{1/3}$. We obtain the bound by the Stirling's formula for the Gamma function. 
\end{proof}

We also need the following standard orthogonality relation for primitive even characters (see e.g.~\cite[Lemma 2]{CIS}). There and later we write, for $b, c \in \mathbb{Z}$, $\sum_{a \mid (b \pm c)} = \sum_{a \mid (b+c)} + \sum_{a \mid (b-c)}$.
\begin{lemma} \label{lem:orthogonal} Let $q \in \mathbb{N}$. If $m, n$ are integers with $(mn, q) = 1$ then
$$ \sumb_{\chiq} \chi(m) \cb(n) = \frac{1}{2} \sum_{\substack{q = dr \\ r | (m \pm n)}} \mu(d) \phi(r).$$
\end{lemma}

Furthermore, the following bound will be helpful in studying the range $m, n \leq Q^{2 - \delta_0}.$

\begin{lemma}
\label{le:zeta6th}
Let $T \geq 3$. Then
$$
\int_{-T}^T |\zeta(1/2+c+it)|^6 dt \ll T^{5/4 + \varepsilon} 
$$
for any $c \geq 0$.
\end{lemma} 
\begin{proof} By H\"older's inequality,
\[
\int_{-T}^T |\zeta(1/2 + c +it)|^6 dt \ll \left(\int_{-T}^T |\zeta(1/2 + c + it)|^4 dt\right)^{3/4} \left(\int_{-T}^T |\zeta(1/2 + c + it)|^{12} dt\right)^{1/4} .\\
\]
The lemma follows from the upper bounds for the fourth and twelfth power moments of the Riemann zeta function:
$$\int_{-T}^T |\zeta(1/2 + c + it)|^4 dt \ll T^{1 + \epsilon}, \qquad \int_{-T}^T |\zeta(1/2 + c + it)|^{12} dt \ll T^{2 + \epsilon} $$
(see e.g.~\cite[formula (7.6.3)]{Ti} for the fourth moment and see~\cite{H-B} for the twelfth moment).

\end{proof}

\subsection{Dissection}
Let us now turn to the first steps in the proof of Theorem~\ref{thm:main}. We start by applying Lemma~\ref{lem:approxfunc} to
\[
\sum_q \Psi\bfrac{q}{Q} \sumbq \Lambda(\chi; \al, \be).
\]
Due to the symmetry of $\al, \be$, it is sufficient to consider the contribution from $\Lambda_0(\chi; \al, \be)$. 
Hence we would like to evaluate
\begin{align*}
\sum_{q} \Psi\bfrac{q}{Q} \sumbq \Lambda_0(\chi; \al, \be)  ,
\end{align*}
where we recall that $\Psi$ is a smooth function supported on $[1, 2]$.
We now extract diagonal terms and introduce smooth partitions of unity.  Let
\begin{equation}
  \label{def:D}
  \D := \sum_q \Psi\bfrac{q}{Q}  \phi^\flat(q) \sum_{(m, q) = 1} \frac{\sigma(m; \al) \sigma(m; -\be) }{m} W_{\al, \be}\left(m, m ; q\right),
  \end{equation}
and write
\begin{align}
\label{eq:extractDiagonal}
\sum_{q} \Psi\bfrac{q}{Q} \sumbq \Lambda_0(\chi; \al, \be) = \D + \sum_{q} \Psi\bfrac{q}{Q} \sumbq \sumd_{M, N} S(M, N),
\end{align}
where $\sumd_{M, N}$ denotes a sum over powers of two and where
\begin{equation} \label{def:SMN1}
S(M, N) := \sumtwo_{\substack{m, n\\ m\neq n}} \frac{\sigma(m; \al) \sigma(n; -\be) \chi(m) \cb(n)}{\sqrt{mn}} W_{\al, \be}\left(m, n ; q\right) V\bfrac{m}{M} V\bfrac{n}{N},
\end{equation}
with $V$ a smooth function supported on $[1/2, 5/2]$ satisfying
$$\sumd_M V\bfrac{m}{M} = 1
$$for all $m\ge 1$. Note that we can always remove and add back terms with $mn\gg Q^{3+\varepsilon}$ with negligible error by using the rapid decay of $W_{\al, \be}\left(m, n; q\right)$ (see Lemma \ref{lem:weightW}).

Let $\widetilde V(s)$ be a Mellin transform of $V(s)$, defined as in \eqref{def:MellinV}.
 Since $V(x)$ is smooth and compactly supported away from zero, the Mellin transform $\widetilde{V}$ is entire and decays rapidly along the vertical axis. 

We now split our analysis into two main cases. 
\subsubsection{Balanced sums}  The first case is the balanced sums where $M$ and $N$ are not too far apart, more precisely the case $\max(M, N) \leq Q^{2 - \delta_0},$ where $\delta_0$ is a fixed real positive number to be chosen later. Let
\begin{equation}
\label{eq:BSdef} \mathcal{BS}(\Psi, Q; \al, \be)  := \sum_{q} \Psi\bfrac{q}{Q} \sumbq \sumd_{\substack{M, N \\ \max(M, N) \le Q^{2-\delta_0}}} S(M, N),
\end{equation}
with $S(M, N)$ be defined as in \eqref{def:SMN1}.  We will prove the following proposition.

\begin{prop}\label{prop:offdiagonal} Let $\varepsilon, \delta_0 > 0$. Then
\begin{align*}
\begin{aligned}
\mathcal{BS}(\Psi, Q; \al, \be) 
&= H(0; \al, \be) \sum_q \Psi \left( \frac{q}{Q}\right) \phi^{\flat}(q) \sum_{\substack{\pi \in S_6/S_3 \times S_3 \\ \pi \ \textrm{permutes exactly } \\ \textrm{one $\alpha_i$ and $\beta_j$}} }\mathcal Q(q; \pi(\al), \pi(\be)) \\
& \hskip 2in + O\left(Q^{2+\varepsilon} \left( Q^{-\frac{\delta_0}2} + Q^{- \frac{1}{4} + \frac{3\delta_0}{4}}\right)\right).
\end{aligned}
\end{align*}
\end{prop}

Morally the balanced case includes also the diagonal terms $\D$. Concerning them we will prove the following proposition.

\begin{prop}\label{prop:diagonal} Let $\varepsilon > 0$ and let $\D$ be as in \eqref{def:D}. Then
$$\D = H(0; \al, \be) \sum_q \Psi\bfrac{q}{Q}\phi^\flat(q) \Q(q; \al, \be) + O(Q^{5/4 + \varepsilon}).
$$
\end{prop}

Proofs of Propositions~\ref{prop:offdiagonal} and~\ref{prop:diagonal} follow~\cite{CIS}. Proposition \ref{prop:diagonal} will be proven in Section~\ref{sec:diagonal} while the longer proof of Proposition \ref{prop:offdiagonal} will be given in Section~\ref{sec:balancedsum}.

\subsubsection{Unbalanced sums}
In the second case one of $M$ and $N$ is much larger than the other. This case was not encountered in~\cite{CIS}. Without loss of generality, we can concentrate on the case $M > N$. We define
\begin{equation}
\label{eq:USdef}
\mathcal{US}(\Psi, Q; \al, \be)  := \sum_{q} \Psi\bfrac{q}{Q} \sumbq \sumd_{\substack{M, N \\ M \ge Q^{2-\delta_0} \\ M > N}} S(M, N),
\end{equation}
with $S(M, N)$ be as in \eqref{def:SMN1}, and will show the following.

\begin{prop}\label{prop:unbalanced} Let $\varepsilon > 0$ and $\delta_0 \in (0, 1/8)$. For any $D \geq 1/2$ and $\delta' \in (0, 1/2)$, we have that
$$\mathcal{US}(\Psi, Q; \al, \be)  \ll   Q^{2+\varepsilon} \left(\frac{1}{D} + Q^{-\delta'/2}+ D Q^{-11/384+ \delta_0 / 2 +11 \delta'/192} \right).
$$
\end{prop}
The proof of Proposition \ref{prop:unbalanced} will be given in Sections \ref{sec:unbalancedprelim}--\ref{se:dgamsmall}.

\section{Proof of Theorem \ref{thm:main}}
We can now quickly deduce Theorem~\ref{thm:main} assuming Propositions~\ref{prop:diagonal},~\ref{prop:offdiagonal}, and~\ref{prop:unbalanced}. Recall that we would like to evaluate 
$$ \sum_q \Psi\left( \frac{q}{Q}\right) \sumb_{\chi \Mod q} \Lambda(\chi; \al, \be). $$
From Lemma \ref{lem:approxfunc}, we have
\begin{equation*}
H(0; \al, \be) \Lambda(\chi; \al, \be) = \Lambda_0(\chi; \al, \be) + \Lambda_0(\chi; \be, \al),
\end{equation*}
and by~\eqref{eq:extractDiagonal},~\eqref{eq:BSdef}, and~\eqref{eq:USdef} we have
\begin{align*}
  &H(0; \al, \be) \sum_q \Psi\left( \frac{q}{Q}\right) \sumb_{\chi \Mod q} \Lambda_0(\chi; \al, \be) \\
  &= \mathcal D(\Psi, Q ; \al, \be) + \mathcal {BS}(\Psi, Q ; \al, \be) + 2\mathcal {US}(\Psi, Q ; \al, \be) +  O(1/Q).
  \end{align*}
We shall see from Propositions~\ref{prop:diagonal}--\ref{prop:unbalanced} that
\begin{align}
  \label{eq:H0pf2.1}
  \begin{aligned}
    &H(0; \al, \be)\sum_q \Psi\bfrac{q}{Q} \sumbq \Lambda(\chi; \al, \be) = H(0; \al, \be)\sum_q \Psi\bfrac{q}{Q} \phib \tQ(q; \al, \be) \\
    &+ O\left(Q^{2+\varepsilon} \left(Q^{-\frac{\delta_0}{2}} + Q^{- \frac{1}{4} + \frac{3\delta_0}{4}}+  \frac{1}{D} + Q^{-\frac{\delta'}{2}}+ D Q^{-\frac{11}{384}+  \frac{\delta_0 }{ 2} + \frac{11 \delta'}{192}} \right) \right).
  \end{aligned}
  \end{align}
Indeed, the error terms match, and to match the main terms, let $\pi \in S_6 / S_3 \times S_3$.
\begin{enumerate}
	\item The contribution of $\pi$ being the coset $S_3 \times S_3$ corresponds to the diagonal 
	$\mathcal{D}(\Psi, Q; \al, \be)$.
	\item The contribution of $\pi$ being the coset of the element that flips all $\alpha_i$ with all $\beta_i$ corresponds to $\mathcal{D}(\Psi, Q; \be, \al)$.
	\item The contributions of $\pi$ being a coset of an element that flips exactly one $\alpha_i$ with one $\beta_i$ sums to $\mathcal{B}\mathcal{S}(\Psi, Q; \al, \be)$.
	\item The contributions of $\pi$ being a coset of an element that flips exactly two $\alpha_i$ with two $\beta_i$ sums to $\mathcal{B}\mathcal{S}(\Psi, Q; \be, \al)$. 
\end{enumerate}

To minimize the error term in~\eqref{eq:H0pf2.1}, we note that the second term is always majorized by the fifth term if the entire error term is to be $\ll Q^2$.  To balance the remaining terms, we choose
\[
Q^{\delta_0/2} = D = Q^{\delta'/2},
\]
so the error term is
\[
\ll Q^{2+\varepsilon}\left(\frac{1}{D} + Q^{-\frac{11}{384}} D^{2+\frac{11}{96} }\right).
\]
This is further balanced for $D = Q^{11/1196}$ and so, with this choice, the error term is $O(Q^{2-11/1196+\varepsilon})$.

Similarly to~\cite[End of Section 11]{CIS}, we can remove the factor $H(0; \al, \be)$ and conclude the proof of Theorem~\ref{thm:main}.

\section{The diagonal terms }\label{sec:diagonal}
In this section we will prove Proposition \ref{prop:diagonal}. The proof is similar to \cite[Proof of Lemma 3]{CIS}, with slight modifications. We include the proof details to make this paper more self-contained. 

By the definition of $W_{\al, \be}$ in \eqref{eqn:Walbe}, the sum over $m$ in $\D$ in \eqref{def:D} is
\begin{equation}\label{integralinD} \bfrac{q}{\pi}^{\delta(\al, \be)} \frac{1}{2\pi i} \int_{(1)} G \left(\frac 12 + s; \al, \be\right) H(s; \al, \be) \left( \frac q\pi\right)^{3s} \sum_{(m, q) = 1} \frac{\sigma(m; \al) \sigma(m; -\be) }{m^{1 + 2s}}  \frac{ds}{s}.
\end{equation}

Moreover, from \eqref{eqn:sumoverm},~\eqref{def:A}, and \eqref{eq:Bqdef}, we obtain that
\begin{align*} \sum_{(m, q) = 1} \frac{\sigma(m; \al) \sigma(m; -\be) }{m^{1 + 2s}}
  = \frac{\A\left( \frac 12 + s; \al, \be \right)\Z\left( \frac 12 + s; \al, \be \right)}{\B_q\left( \frac 12 + s; \al, \be \right)},
\end{align*}
and so \eqref{integralinD} equals
\begin{align*}
  \bfrac{q}{\pi}^{\delta(\al, \be)} \frac{1}{2\pi i} \int_{(1)} G \left(\frac 12 + s; \al, \be\right) H(s; \al, \be) \left( \frac q\pi\right)^{3s} \frac{\A\left( \frac 12 + s; \al, \be \right)\Z\left( \frac 12 + s; \al, \be \right)}{\B_q\left( \frac 12 + s; \al, \be \right)}  \frac{ds}{s}.
  \end{align*}
We move the contour integral to $\R (s) = -\frac 14 + \varepsilon$, picking up a simple pole at $s = 0$. Note  that the poles of $\Z(1/2 + s; \al, \be)$  at $s = -(\alpha_i - \beta_j)/2$ are cancelled by the zeros of $H(s; \al, \be)$ at these same points. Thus the expression above is 
\begin{align*}
  \bfrac{q}{\pi}^{\delta(\al, \be)}\left( G\left( \frac 12 ; \al, \be\right) H(0; \al, \be)\frac{\A\left( \frac 12; \al, \be \right)\Z\left( \frac 12; \al, \be \right)}{\B_q\left( \frac 12; \al, \be \right)} + O(q^{-3/4 + 4\varepsilon})\right),
\end{align*}
and we obtain the lemma by inserting this into~\eqref{def:D}, using~\eqref{def:Qalbe}, and adjusting $\varepsilon$.

\section{Balanced sums} \label{sec:balancedsum}

\subsection{Initial reductions}
We will follow \cite[Sections 5--10]{CIS} to calculate the balanced sum in Proposition~\ref{prop:offdiagonal}. Since many calculations will be very similar to \cite{CIS}, we will quote results from~\cite{CIS} along with necessary modification for our balanced sum. 

Using orthogonality relation for characters in Lemma \ref{lem:orthogonal}, we obtain that 

\begin{align} \label{eqn:BSinitial}
  \begin{aligned}
    \mathcal{BS}(\Psi, Q; \al, \be)  &= \frac 12\sumd_{\substack{M, N \\ \max(M, N) \le Q^{2-\delta_0}}} \sumtwo_{\substack{m, n \geq 1\\ m\neq n}} \frac{\sigma(m; \al) \sigma(n; -\be) }{\sqrt{mn}}  V\bfrac{m}{M} V\bfrac{n}{N} \\
	& \hskip 1in \cdot \sumtwo_{\substack{d, r \\ (dr, mn) = 1 \\ r | m \pm n}} \mu(d) \phi(r) \Psi \left( \frac{dr}{Q}\right)  W_{\al, \be}\left(m, n ; dr\right) \\
        &=: \sumd_{\substack{M, N \\ \max(M, N) \le Q^{2-\delta_0}}}  \mathcal {BD}(M, N) + \sumd_{\substack{M, N \\ \max(M, N) \le Q^{2-\delta_0}}} \mathcal {BG}(M, N),
  \end{aligned}
\end{align}
where for $D_0$ is a parameter to be chosen later, $\mathcal {BD}(M, N)$ is the contribution from terms with $d > D_0$ and $\mathcal {BG}(M, N)$ is the contribution of terms with $d \leq D_0.$

First, we consider $\mathcal {BD} (M, N)$. By following the arguments in \cite[Section 5]{CIS}, we show the following.
\begin{lem}  \label{lem:BD}
Let $\delta_0 > 0$ and let $M, N$ be such that $\max(M, N) \leq Q^{2-\delta_0}$, and let $D_0 \geq 1/2$. Then
\begin{align*}
  \mathcal{BD}(M, N) = \mathcal {MBD}(M, N) + O\left( \frac{Q^{2 + \varepsilon}}{D_0} + D_0Q^{3/2 + \varepsilon}\right),
\end{align*}
where 
\begin{align}
  \label{def:MBDMN}
  \begin{aligned}
\mathcal {MBD}(M, N) := -Q^{1 + \delta(\al, \be)} &\sum_{ \substack{m, n \geq 1 \\ m \neq n} } \frac{\sigma(m; \al) \sigma(n; -\be)}{\sqrt{mn}} V\left( \frac mM\right) V\left( \frac nN\right) \\
&\cdot \sum_{\substack{d \leq D_0 \\ (d, mn) = 1}} \frac{\mu(d)}{d} \frac{\phi(mn)}{mn} \int_0^\infty \Psi(u)  W_{\al, \be} \left( \frac{m}{Q^{3/2}}, \frac{n}{Q^{3/2}}; u \right) \> du.
  \end{aligned}
\end{align}

\end{lem}

\begin{proof}
Let us consider the sum $\mathcal{BD}(M, N)$. By Lemma~\ref{lem:weightW} we can assume that, for any $\varepsilon > 0$, $MN \leq Q^{3+\varepsilon}$.	Now, we express the condition $r|m \pm n$ in~\eqref{eqn:BSinitial} as $\frac{2}{\phi(r)}\sum_{ \substack{\psi \Mod r \\ \psi(-1) = 1} } \psi(m) \overline{\psi(n)}$.  The contribution of the principal characters is 
	$$ \sum_q \left( \sum_{\substack{dr = q \\ d > D_0}}\mu(d)\right) \Psi\left( \frac qQ\right) \sumtwo_{\substack{m, n \geq 1 \\ (mn, q) = 1 \\ m \neq n}} \frac{\sigma(m, \al)\sigma(n, -\be)}{\sqrt{mn}} V\bfrac{m}{M} V\bfrac{n}{N}  W_{\al, \be}\left(m, n ; q\right) .$$
	Since $\sum_{ dr = q } \mu(d) = 0$ when $q > 1,$ the above is
\begin{equation}	\label{eqn:MBD1}
	 - \sumtwo_{\substack{m, n \geq 1  \\ m \neq n}} \frac{\sigma(m, \al)\sigma(n, -\be)}{\sqrt{mn}} V\bfrac{m}{M} V\bfrac{n}{N} \sum_{\substack{d \leq D_0 \\ (d, mn) = 1} }  \mu(d)  \sum_{\substack{r \\ (r, mn) = 1}} \Psi\left( \frac {dr}Q\right)  W_{\al, \be}\left(m, n ; dr\right) .
	 \end{equation}
Now, we use the fact that 
 $$ \sum_{\substack{r \leq x \\ (r, mn)= 1}} 1 = \frac{\phi(mn)}{mn} x + O((mn)^\varepsilon), $$
partial summation, and the formula 
\begin{equation}
\label{eq:WFE}
W_{\al, \be}(m, n; uQ) = Q^{\delta(\al, \be)} W_{\al, \be} \left( \frac{m}{Q^{3/2}}, \frac{n}{Q^{3/2}}; u \right)
\end{equation}
to derive that \eqref{eqn:MBD1} equals 
$$ \mathcal {MBD}(M, N) + O(D_0Q^{\frac 32 + \varepsilon}).$$ 
	
Next let $\mathcal {EB}(M, N)$ be the contribution from the non-principal characters, so that
\begin{align}
  \label{EBterm}
  \begin{aligned}
    \mathcal {EB}(M, N) &:= \sum_{d > D_0} \sum_r \mu(d) \Psi\left( \frac{dr}{Q} \right) \sum_{ \substack{\psi \Mod r \\ \psi(-1) =1 \\ \psi \neq \psi_0}} \sumtwo_{\substack{m, n \geq 1 \\ m \neq n \\ (mn, dr) = 1}} \psi(m) \overline{\psi(n)} \frac{\sigma(m; \al) \sigma(n, -\be)}{\sqrt{mn}} \\
                        & \hskip 1in \cdot V\bfrac{m}{M} V\bfrac{n}{N}  W_{\al, \be}\left(m, n ; dr\right) + O\left( \frac{Q^{2 + \varepsilon}}{D_0}\right),
  \end{aligned}
\end{align}
where $\psi_0$ is the principal character.
We will show that 
	$$\mathcal {EB}(M, N) \ll \frac{Q^{2 + \varepsilon}}{D_0}.$$ 
Note that adding back the terms $m = n$ to~\eqref{EBterm} contributes $O\left( \frac{Q^{2 + \varepsilon}}{D_0}\right)$. By the definition of $W_{\al, \be}$ in \eqref{eqn:Walbe} and Mellin inversion for $V$, the sum over $m, n$ in $\mathcal {EB}(M, N)$ (without the condition $m \neq n$) is
\begin{align} \label{summninEB}
  \begin{aligned}
\frac{1}{(2\pi i )^3}\int_{(1)} \int_{(\varepsilon)} \int_{(\varepsilon)} &G\left( \frac 12 + s; \al, \be \right) H(s; \al, \be) \left( \frac{dr}{\pi}\right)^{3s + \delta(\al, \be)} \widetilde{V}(z) \widetilde{V}(w) M^{z} N^{w} \\
                                                                          & \cdot \sumtwo_{\substack{m, n \geq 1 \\ (mn, dr) = 1}} \psi(m) \overline{\psi(n)} \frac{\sigma(m; \al) \sigma(n, -\be)}{m^{1/2 + s + z} n^{1/2 + s + w}} \> dz \> dw\>\frac{ds}{s},
  \end{aligned}
\end{align}
where $\varepsilon > 0.$ The sum over $m, n$ can be expressed in terms of Dirichlet $L$-functions as 
$$  \prod_{i = 1}^3  \frac{ L(\tfrac 12 + \alpha_i + s + z, \psi)}{ L_{dr}(\tfrac 12 + \alpha_i + s + z, \psi)} \prod_{j=1}^3 \frac{L(\tfrac 12 - \beta_j + s + w, \overline{\psi})  }{L_{dr}(\tfrac 12 - \beta_j + s + w, \overline{\psi})},$$
where $L_{dr}(s, \psi) = \prod_{p | dr} \left( 1- \frac{\psi(p)}{p^s}\right)^{-1}$. Since $\psi$ is not the trivial character, the Dirichlet $L$-functions above are entire.  We thus move the integral over $s$ to $\R(s) = \varepsilon$ without crossing any poles of the integrand. We further note that the gamma factor $G$ is $\ll \exp(- |\textrm{Im}(s)|)$, $L_{dr}(s, \psi) \ll Q^{\varepsilon}$,  $\widetilde{V}(\sigma + it) \ll_{\sigma, A} \frac{1}{1 + |t|^A}$, and $M, N \ll Q^{2 - \delta_0}$. Hence, the triple integral in \eqref{summninEB} is bounded by
\begin{align*} Q^{O(\varepsilon)} \int_{(\varepsilon)} \int_{(\varepsilon)} \int_{(\varepsilon)} &\exp(-|\textrm{Im}(s)|) \frac{1}{1+|z|^{10}} \frac{1}{1+|w|^{10}} \\
                                                &\left( \sum_{i = 1}^3 \left| L\left( \frac 12 + \alpha_j + s + z , \psi \right)\right|^6 + \sum_{j = 1}^3 \left| L\left( \frac 12 - \beta_j + s + w , \psi \right)\right|^6\right) \> dz \> dw \> ds .
\end{align*}
As in~\cite[Proof of Proposition 3]{CIS}, we insert this into~\eqref{EBterm} (recalling we removed the condition $m \neq n$) and use the large sieve inequality (analogously to~\cite[Theorem 7.34]{IK}). Adjusting $\varepsilon$, we obtain that $\mathcal{EB}(M, N) \ll \frac{Q^{2 + \varepsilon}}{D_0}$. 

\end{proof}

We recall that 
\begin{align*}
  \mathcal {BG}(M, N) &= \frac 12\sumtwo_{\substack{m, n \geq 1\\ m\neq n}} \frac{\sigma(m; \al) \sigma(n; -\be) }{\sqrt{mn}}  V\bfrac{m}{M} V\bfrac{n}{N} \\
                      & \hskip 1in \cdot \sumtwo_{\substack{d \leq D_0, r \\ (dr, mn) = 1 \\ r | m \pm n}} \mu(d) \phi(r) \Psi \left( \frac{dr}{Q}\right)   W_{\al, \be}\left(m, n ; dr\right).
\end{align*}

Let $g = \gcd(m, n)$ and write $m = g\m$ and $n = g\n$. Arguing as in~\cite[Equations (28)--(30)]{CIS}, we obtain
$$  \mathcal {BG}(M, N) = \mathcal {BG}^{+}(M, N) + \mathcal {BG}^-(M, N),$$
where
\begin{align}
\label{eq:BGpm}
  \begin{aligned}
  \mathcal{BG}^{\pm}(M, N) &:= \frac 12\sumtwo_{\substack{m, n \geq 1\\ m\neq n}} \frac{\sigma(m; \al) \sigma(n; -\be) }{\sqrt{mn}}  V\bfrac{m}{M} V\bfrac{n}{N} \\
& \qquad \cdot \sum_{ \substack{d \leq D_0 \\ (d, mn) = 1 }} \sum_{\substack{a \geq 1 \\ (a, g) = 1}} \sum_{\substack{b|g}}  \sum_{ \substack{h \geq 1 \\ \m \equiv \mp \n \Mod{abh}}} \mu(d) \mu(a) \mu(b) \\
& \qquad \cdot \frac{|\m \pm \n|}{ah} \Psi \left( \frac{|\m \pm \n|d}{Qh}\right) W_{\al, \be}\left(g\m, g\n ; \frac{d|\m \pm \n|}{h}\right).
\end{aligned}
\end{align}

\begin{rem} \label{rem:gsize}
Since $g = \textrm{gcd}(m, n)$, $b | g$, $m \asymp M$ and $n \asymp N$, we have that $b, g \ll \min( M, N).$  Moreover, the factor $\Psi \left( \frac{|\m \pm \n|d}{Qh}\right)$ forces $h \leq \frac{10 \cdot Q^{2 - \delta_0} D_0}{Q} = 10 \cdot Q^{1 - \delta_0} D_0. $
\end{rem}
We define, for $x, y, u \geq 0$,
\begin{equation}
\label{eq:Walbedef}
\mathcal W^{\pm}_{\al, \be} (x, y; u) := u|x \pm y| \Psi(u |x\pm y|)  W_{\al, \be}\left(x, y ; u |x \pm y|\right).
\end{equation}
\begin{rem} This function $\mathcal W^{\pm}_{\al, \be} (x, y; u)$ here is defined similarly to $\mathcal W^{\pm}_{\al, \be} (x, y; u)$ in \cite{CIS} but we use $W_{\al, \be}\left(x, y ; u |x \pm y|\right)$ instead of $V_{\al, \be}\left(x, y ; u |x \pm y|\right)$, which is defined in (16) of \cite{CIS}.
\end{rem}

We obtain from~\eqref{eq:BGpm},~\eqref{eq:WFE} and ~\eqref{eq:Walbedef} that
\begin{align*}
\mathcal{BG}^{\pm}(M, N) =& \frac{ Q^{1 + \delta(\al, \be)}}{2}  \sumtwo_{\substack{m, n \geq 1\\ m\neq n}} \frac{\sigma(m; \al) \sigma(n; -\be) }{\sqrt{mn}}  V\bfrac{m}{M} V\bfrac{n}{N} \\
&\cdot  \sum_{ \substack{d \leq D_0 \\ (d, mn) = 1 }} \sum_{\substack{a \geq 1 \\ (a, g) = 1}} \sum_{b|g}  \sum_{ \substack{h \geq 1 \\ \m \equiv \mp \n \Mod{abh}}} \frac{\mu(d) \mu(a) \mu(b) }{ad} \mathcal W^{\pm}_{\al, \be} \left( \frac{g\m}{Q^{3/2}}, \frac{g\n}{Q^{3/2}}; \frac{Q^{1/2}d}{gh}\right).
\end{align*}
Next we write the condition $\m \equiv \mp \n \Mod {abh}$ as a sum over characters $\psi \Mod{abh}$. Note that this is possible because $(\m \n, a b h) = 1$ since $(\m , \n) = 1$ and $\m \equiv \pm \n \Mod{abh}$. Then we separate $\mathcal {BG}(M,N) $ into two terms.  One is the contribution of the principal characters, which forms the main term, while the other  is the contribution of the non-principal characters, which contribute to the error term. More precisely we write
$$ \mathcal {BG}^{\pm}(M,N) = \mathcal {MBG}^{\pm}(M,N) + \mathcal {EBG}^{\pm}(M,N),$$
where
\begin{align} \label{def:MBG}
\begin{aligned}  
\mathcal{MBG}^{\pm}(M, N) :=& \frac{ Q^{1 + \delta(\al, \be)}}{2}  \sumtwo_{\substack{m, n \geq 1\\ m\neq n}} \frac{\sigma(m; \al) \sigma(n; -\be) }{\sqrt{mn}}  V\bfrac{m}{M} V\bfrac{n}{N} \\
&\cdot  \sum_{ \substack{d \leq D_0 \\ (d, g\m\n) = 1 }} \sum_{(a, g\m\n) = 1} \sum_{\substack{b|g \\ (b, \m\n) = 1}}  \sum_{ \substack{h \geq 1 \\(h, \m \n)=1 }} \frac{\mu(d) \mu(a) \mu(b) }{ad \phi(abh)} \mathcal W^{\pm}_{\al, \be} \left( \frac{g\m}{Q^{3/2}}, \frac{g\n}{Q^{3/2}}; \frac{Q^{1/2}d}{gh}\right)
\end{aligned}
\end{align}
and 
\begin{align} \label{def:EBG}
  \begin{aligned}
\mathcal{EBG}^{\pm}(M, N) :=& \frac{ Q^{1 + \delta(\al, \be)}}{2}  \sumtwo_{\substack{m, n \geq 1\\ m\neq n}} \frac{\sigma(m; \al) \sigma(n; -\be) }{\sqrt{mn}}  V\bfrac{m}{M} V\bfrac{n}{N} \\
&\cdot  \sum_{ \substack{d \leq D_0 \\ (d, g\m\n) = 1 }} \sum_{\substack{a \geq 1 \\ (a, g) = 1}} \sum_{b|g}  \sum_{ \substack{h \leq 10 \cdot Q^{1 - \delta_0}D_0  \\(abh, \m\n)=1 }} \frac{\mu(d) \mu(a) \mu(b) }{ad \phi(abh)}\\
&\cdot \sum_{ \substack{\psi \Mod{abh} \\ \psi \neq \psi_0} }\psi(\m) \overline{\psi}(\mp \n)\mathcal W^{\pm}_{\al, \be} \left( \frac{g\m}{Q^{3/2}}, \frac{g\n}{Q^{3/2}}; \frac{Q^{1/2}d}{gh}\right).
  \end{aligned}
\end{align}
Moreover we define
\begin{align}
  \label{def:MBEB}
\begin{aligned}
  \mathcal {MBG}(M, N) &:= \mathcal {MBG}^+(M, N) + \mathcal {MBG}^-(M, N)\\
  \mathcal {EBG}(M, N) &:= \mathcal {EBG}^+(M, N) + \mathcal {EBG}^- (M, N).
\end{aligned} 
\end{align}
Thus
\begin{align}\label{eqn:BGrelatedtoMBGandEBG}
 \mathcal {BG}(M, N) = \mathcal {MBG}(M, N) + \mathcal {EBG}(M, N).
\end{align}
To evaluate $\mathcal {MBG}(M, N)$ and $\mathcal {EBG}(M, N)$, we require information about the Mellin transforms of $\mathcal W^{\pm}_{\al, \be} (x, y, u)$.  To be specific, we will collect lemmas about three different types of Mellin transforms. The first type is in the $u$-variable, the second one is in the $x, y$-variables, and the third one is in all three variables. The proofs of lemmas follow closely the proofs in \cite{CLMR, CIS}, 
but using the bound for  $W_{\al, \be}(\xi, \eta; \mu)$ in Lemma \ref{lem:weightW} instead. Thus we will state the results without the proof. Note that identities and bounds in our lemmas differ slightly from lemmas in~\cite{CLMR, CIS}. 

\begin{lemma} \label{lem:Mellin1}
		Given positive real numbers $x, y$, let 
	$$ \widetilde{\mathcal W}^{\pm}_1(x, y; z) = \int_0^\infty \mathcal W^{\pm}_{\al, \be}(x, y; u) u^z \frac{du}{u}. $$
Then the functions $\Wt^\pm_1(x, y; z)$
are analytic for all $z \in \mathbb C$. We have the Mellin inversion formula
\begin{align*}
\W^\pm_{\al, \be}(x, y; u) = \frac{1}{2\pi i} \int_{(c)} \Wt^\pm_1(x, y; z) u^{-z} \> dz,
\end{align*}
where the integral is taken over the line $\tRe(z) = c$ for any real number $c$. The Mellin transforms $\Wt^\pm_1(x, y; z)$ satisfy, for any non-negative integer $\nu$,
$$ |\Wt^\pm_1(x, y;z)| \ll_\nu |x \pm y|^{-\tRe z} \prod_{j = 1}^\nu |z + j|^{-1} \exp \left(-c_0 (xy)^{1/3}\right)$$
for some absolute constant $c_0.$
\end{lemma}
\begin{proof}
This is essentially the same as \cite[Lemma 4]{CIS}.
\end{proof}

\begin{lemma} \label{lem:MellinXY}
Given a positive real number $u,$ we define
$$ \Wt^\pm_2 (s_1, s_2; u) = \int_0^\infty \int_0^\infty \W^\pm_{\al, \be}(x, y ; u) x^{s_1}y^{s_2} \frac{dx}{x} \frac{dy}{y}.$$
Then the functions $\Wt^\pm_2(s_1, s_2 ; u)$ are analytic in the region $\tRe (s_1), \tRe(s_2) > 0$. We have the Mellin inversion formula
$$ \W^\pm_{\al, \be}(x, y ; u) = \frac{1}{(2\pi i)^2} \int_{(c_1)}\int_{(c_2)} \Wt^\pm_2 (s_1, s_2 ; u) x^{-s_1} y^{-s_2} \>d s_1 \> d s_2,$$
where $c_1, c_2 > 0$. The Mellin transforms $\Wt^\pm_2(s_1,s_2 ; u)$ satisfy, for any $k \geq 1$ and $l \geq 0,$ and any $s_1, s_2$ with $0 < \tRe(s_1), \tRe(s_2) \leq 100$ 
$$ |\Wt^\pm_2(s_1, s_2; u)| \ll \frac{1}{\tRe(s_1) \tRe(s_2)} \cdot \frac{(1 + u)^{k-1}}{\max(|s_1|, |s_2|)^k |s_1 + s_2|^l } .$$
\end{lemma}
\begin{proof}
This is essentially the same as~\cite[Proof of Lemma 7.2]{CLMR}.
\end{proof}

The next lemma is similar to Lemma 6 in \cite{CIS} and Lemma 6.2 in \cite{CLMR}. The truncation was not explained in \cite{CIS}, but here we state the needed truncated Mellin inversion formulas. The proof follows closely the proof of \cite[Lemma 6.2]{CLMR}. 

\begin{lemma} \label{lem:MellinXYU} We define
$$ \Wt^\pm_3 (s_1, s_2; z) = \int_0^\infty \int_0^\infty  \int_0^\infty \W^\pm_{\al,\be}(x, y ; u) u^z x^{s_1}y^{s_2} \frac{du}{u} \frac{dx}{x} \frac{dy}{y}$$
and 
$$ \Wt_3 (s_1, s_2; z) = \Wt^+_3 (s_1, s_2; z) + \Wt^-_3 (s_1, s_2; z).$$
Let $\omega = \frac{s_1 + s_2 - z}{2}$ and $\xi = \frac{s_1 - s_2 + z}{2}.$ For $\tRe(s_1), \tRe(s_2) > 0,$ and $|\tRe(s_1 -s_2)| < \tRe(z) < 1$ we have
\begin{align*} 
\Wt_3(s_1, s_2; z) = \frac{\psit(1 + \delta(\al, \be) + 3\omega + z) H(\omega; \al, \be)}{2\omega \pi^{3\omega + \delta(\al, \be)}}  \Hc (\xi, z) G\left(\h + \omega; \al, \be \right), 
\end{align*}
where $\psit$ is the Mellin transform of $\Psi$, and 
$$ \Hc (u, v) = \pi^{1/2} \frac{\Gamma\left(\tfrac{u}{2} \right)\Gamma\left(\tfrac{1-v}{2} \right)\Gamma\left(\tfrac{v-u}{2} \right)}{\Gamma\left(\tfrac{1-u}{2} \right)\Gamma\left(\tfrac{v}{2} \right)\Gamma\left(\tfrac{1-v + u}{2} \right)}.$$

Let $x \neq y$ and $T \ge Q^{\varepsilon}$. For any $c_1, c_2 > 0$ with $ |c_1 - c_2| < c < 1$, one has the truncated Mellin inversion formulas
\begin{align*}
\W(x, y ; u) &=\frac{1}{(2\pi i)^3} \int_{(c)} \int_{c_1 - iT}^{c_1 + iT}\int_{c_2 - iT}^{c_2 + iT} \Wt_3 (s_1, s_2 ; z) u^{-z} x^{-s_1} y^{-s_2} \>d s_2 \> d s_1 \> dz \\
&\qquad \qquad \qquad \qquad +  O\left(\frac{ u^{-c} x^{-c_1} y^{-c_2}}{ T^{1 - c}\left|\log\left( \frac xy\right) \right| }  \right).
\end{align*}
Moreover, let $\Wt_1(x, y ; z) = \Wt_1^+(x, y ; z) + \Wt_1^-(x, y ; z)$. Then for $\tRe z = c$,
\begin{align}
\label{eqn:truncate2}
\Wt_1(x, y ; z) &=\frac{1}{(2\pi i)^2} \int_{c_1 - iT}^{c_1 + iT}\int_{c_2 - iT}^{c_2 + iT} \Wt_3 (s_1, s_2 ; z) x^{-s_1} y^{-s_2} \>d s_2 \> d s_1 \> \\
\nonumber
&\qquad \qquad \qquad \qquad +  O\left(\frac{  x^{-c_1} y^{-c_2}}{ T^{1 - c}\left|\log\left( \frac xy\right) \right| (1 + |z|)^{A} }  \right),
\end{align}for any $A>0$.
Finally, for $\tRe(s_1), \tRe(s_2) > 0,$ and $|\tRe(s_1 -s_2)| < \tRe(z) < 1$, the Mellin transform $\Wt_3(s_1, s_2; z)$ satisfies the bound
\begin{equation} \label{eqn:boundWt3}
|\Wt_3(s_1, s_2;z)| \ll (1 + |z|)^{-A} (1 + |\omega|)^{-A} (1 + |\xi|)^{\tRe(z) - 1},
\end{equation}for any $A>0$.

\end{lemma}
\subsection{Evaluating the main terms} \label{ssec:evalMBG}
In this section we will evaluate $\mathcal {MBG}^{\pm}(M,N)$ defined in \eqref{def:MBG}. First, we define auxiliary functions using the same notation as in \cite[Equation (56) and Lemma 7]{CIS}. Let
$$ F(h,g; \m\n) := \sum_{(a, g\m\n) = 1} \frac{\mu(a)}{a} \sum_{\substack{b|g \\ (b, \m\n) = 1}} \frac{\mu(b)}{\phi(abh)} = \sum_{(\ell, \m\n) = 1} \frac{\mu(\ell)(\ell, g)}{\ell \phi(\ell h)}.$$
If $\R (s) > 0$, then
\begin{align*}
\sum_{\substack{h \geq 1 \\ (h, \m\n) = 1}} \frac{F(h, g; \m\n)}{h^s} = \zeta(s+1) \mathcal K(s; g, \m\n), 
\end{align*}
where 
\begin{align*}
	\mathcal K(s; g, \m\n) := &\phi(\m\n, s+1) \prod_{p \nmid \, g\m \n} \left( 1 - \frac{1}{p(p-1)} + \frac{1}{p^{1 + s}(p - 1)}\right) \\
	&\cdot \prod_{\substack{p | g \\ p \nmid \m\n} } \left(1 - \frac{1}{p^{1 + s}} - \frac{1}{p - 1}\left( 1 - \frac{1}{p^s}\right) \right)
\end{align*}
and $\phi(\ell, s) := \prod_{p | \ell} \left( 1 - \frac{1}{p^s}\right).$  We now prove the following Lemma.
 
\begin{lem} \label{lem:MBGMN} Let $\varepsilon > 0$. Let $\mathcal {MBG}^{\pm}(M, N)$ be as in \eqref{def:MBG} with $D_0 \geq 2$, and let $\mathcal {MBD}(M, N)$ be as in \eqref{def:MBDMN}. Then
	\begin{align*}
	\mathcal {MBG}^{\pm}(M, N) = -\frac 12 \mathcal {MBD}(M, N) + \mathcal {MBG}_1^{\pm}(M, N) +  O \left( \frac{Q^{2 + \varepsilon}}{D_0} + D_0Q^{3/2 + \varepsilon}\right),
	\end{align*}
	where
        \begin{align}
          \label{def:MBG1}
          \begin{aligned}
		\mathcal {MBG}_1^{\pm}(M, N) := &\frac{Q^{1 + \delta(\al, \be)}}{2} \sum_{ \substack{m, n \geq 1 \\ m \neq n } } \frac{\sigma(m; \al) \sigma(n; -\be)}{\sqrt{mn}} V\left( \frac mM\right) V\left( \frac nN\right) \\
		& \cdot \frac{1}{2\pi i } \int_{(\varepsilon)} \widetilde{\mathcal W}_1^\pm \left( \frac{g\m}{Q^{3/2}}, \frac{g\n}{Q^{3/2}} ; z \right) \frac{\zeta(1-z) \mathcal K(-z; g, \m\n) }{\zeta(1 + z) \phi (g\m\n, 1 + z)}\left( \frac{Q^{1/2}}{g}\right)^{-z} \> dz.
          \end{aligned}
        \end{align}
      \end{lem}
\begin{proof}
	We follow the arguments  in \cite[Equations (57)-(62)]{CIS}, using the  Mellin transform from Lemma \ref{lem:Mellin1}. This gives
	\begin{align*}
	\mathcal {MBG}^{\pm}(M, N) = -\frac 12\mathcal {MBD}(M, N) + \mathcal {MBG}_0^{\pm}(M, N) + O(D_0Q^{3/2 + \varepsilon}),
	\end{align*}
	where, for $\varepsilon > 0$,
\begin{align*}
		\mathcal {MBG}_0^{\pm}(M, N) := &\frac{Q^{1 + \delta(\al, \be)}}{2} \sumtwo_{ \substack{m, n \geq 1 \\ m \neq n} } \frac{\sigma(m; \al) \sigma(n; -\be)}{\sqrt{mn}} V\left( \frac mM\right) V\left( \frac nN\right) \\
		& \cdot \frac{1}{2\pi i } \int_{(\varepsilon)} \widetilde{\mathcal W}_1^\pm \left( \frac{g\m}{Q^{3/2}}, \frac{g\n}{Q^{3/2}} ; z \right) \zeta(1-z) \mathcal K(-z; g, \m\n) \left( \frac{Q^{1/2}}{g}\right)^{-z} \sum_{\substack{d \leq D_0 \\ (d, mn) = 1}} \frac{\mu(d)}{d^{1 + z}} \> dz,
\end{align*}
where as usual per our convention $m = g \m$, $n = g \n$ and $(\m, \n) = 1$.
	Next we deal with $\mathcal {MBG}_0^{\pm}(M, N)$ by following the argument in \cite[Equations (62) - (63)]{CIS}. To be more specific, we move the line of integration to $\tRe z = 1- \varepsilon$ and extend the sum over $d$ to all positive integers. Then we move the integration back to $\tRe z = \varepsilon$ 
	at a cost of $O (Q^{2 + \varepsilon} / D_0)$. We then obtain that 
$$ \mathcal {MBG}_0^{\pm}(M, N) = \mathcal {MBG}_1^{\pm}(M, N) + O\left( \frac{Q^{2 + \varepsilon}}{D_0}\right).$$
This concludes the proof of the lemma.

\end{proof}
Let 
\begin{equation} \label{def:MBG1combinepm}
 \mathcal {MBG}_1(M, N) := \mathcal {MBG}_1^{+}(M, N) + \mathcal {MBG}_1^{-}(M, N).
 \end{equation}
By \eqref{def:MBEB} and Lemma~\ref{lem:MBGMN},  we obtain that 
$$\mathcal {MBG}(M, N) = - \mathcal {MBD}(M, N) + \mathcal {MBG}_1(M, N) + O \left( \frac{Q^{2 + \varepsilon}}{D_0} + D_0Q^{3/2 + \varepsilon}\right). $$
The above equation and Lemma \ref{lem:BD} indicate that the possibly large main term $\mathcal {MBD}(M, N)$ of $\mathcal {BD}(M, N)$ is cancelled with one of the main terms from $\mathcal {MBG}(M, N) $. In particular, 

\begin{align}
  \label{eqn:combineBDandMBG}
  \mathcal {BD}(M, N) + \mathcal{MBG}(M, N) = \mathcal {MBG}_1(M, N) +  O \left( \frac{Q^{2 + \varepsilon}}{D_0} + D_0Q^{3/2 + \varepsilon}\right).
\end{align}

Next we consider the main term contribution from $\mathcal {MBG}_1(M, N)$. First, we will show that when summing $\mathcal {MBG}_1(M, N)$ dyadically over $M, N$, the main contribution comes from when both $M, N$ are small ($\ll Q^{2 -\delta_0}$).

\begin{lem} \label{lem:sumMBG1} Let $\varepsilon > 0$ and $\delta_0 > 0$ be fixed. Then
$$\sumtwodee_{M, N \leq Q^{2 - \delta_0} } \mathcal {MBG}_1(M, N) = \sumd_M\sumd_N \mathcal {MBG}_1(M, N)  + O(Q^{2-\frac {1}{4} + \frac{3}{4} \delta_0 + \varepsilon}).$$	
\end{lem}

To prove Lemma \ref{lem:sumMBG1}, it is sufficient to show that $\mathcal {MBG}_1(M, N)$ is small
when $M$ or $N \gg Q^{2 - \delta_0}$. Since we can assume $MN \ll Q^{3 + \varepsilon}$, without loss of generality, we assume that $ M \gg Q^{2 - \delta_0} $ and $ N \ll Q^{1 + \delta_0 + \varepsilon}$. Thus Lemma \ref{lem:sumMBG1} will immediately follow from the following lemma.

\begin{lem}  \label{lem:MBGunbalancedMN}  Let $\varepsilon > 0$. Let $\mathcal {MBG}_1(M, N)$ be as in \eqref{def:MBG1combinepm} with $\mathcal{MBG}_1^\pm(M, N)$ as in~\eqref{def:MBG1}. For any $M \gg Q^{2 - \delta_0}$ and $N \ll Q^{1 + \delta_0 + \varepsilon},$ we have
$$ 
\mathcal {MBG}_1(M, N) \ll Q^{\frac {7}{4} + \frac{3}{4} \delta_0 + \varepsilon}  .
$$
\end{lem}
\begin{proof} 
In the definition of $\mathcal {MBG}_1(M, N)$ in \eqref{def:MBG1combinepm}, 
	 we add up the $\widetilde{\mathcal W}_1^+$ and $\widetilde{\mathcal W}_1^-$ terms from $\mathcal {MBG}_1^\pm(M, N)$ in \eqref{def:MBG1}. Applying the Mellin transform in Equation \eqref{eqn:truncate2} and using in the error term the rapid decay in $z$ from Equation~\eqref{eqn:boundWt3}, we obtain that, for $\varepsilon > 0,$

\begin{align*}
\mathcal {MBG}_1(M, N) = &\frac{Q^{1 + \delta(\al, \be)}}{2} \frac{1}{(2\pi i )^3} \sum_{\substack{m, n \geq 1 \\ m \neq n }} \frac{\sigma(m; \al) \sigma(n; -\be)}{\sqrt{mn}} V\left( \frac{m}{M}\right)V\left( \frac nN\right)\\
& \hskip 0.2 in \cdot \int_{(\varepsilon)} \int_{\frac 12 + \varepsilon - iT}^{\frac 12 + \varepsilon + iT} \int_{\frac 12 + \varepsilon - iT}^{\frac 12 + \varepsilon + iT}  \widetilde{\mathcal W}_3 \left( s_1, s_2 ; z \right)     \frac{\zeta(1-z) \mathcal K(-z; g, \m\n) }{\zeta(1 + z) \phi (g\m\n, 1 + z)} \\
& \hskip 2in  \cdot \left( \frac{Q^{1/2}}{g}\right)^{-z} \left( \frac{Q^{\frac 32}}{m}\right)^{s_1} \left( \frac{Q^{\frac 32}}{n}\right)^{s_2}     \>ds_1 \> ds_2 \> dz  \\
& \hskip 0.2 in + O\left( \frac{Q^{ \frac 52 + 3\varepsilon}}{T^{1 - \epsilon}} \sum_{\substack{m, n \geq 1 \\ m \neq n }} \frac{|\sigma(m; \al) \sigma(n; -\be)|}{m^{1 + \varepsilon } n^{1 + \varepsilon}} \frac{1}{\log \left( \frac mn\right)} \right).
\end{align*}
We choose $T := Q^{6/5}$ so that the error term is bounded by $Q^{\frac{13}{10} + 5\varepsilon}.$

In the main term we remove the condition $m \neq n$ since this is already implied by the ranges of $M$ and $N$.  Using the Mellin transform of $V$, we have 
\begin{align*}
\mathcal {MBG}_1(M, N) = &\frac{Q^{1 + \delta(\al, \be)}}{2} \frac{1}{(2\pi i )^5}  \int_{(\varepsilon)} \int_{\frac 12 + \varepsilon - iT}^{\frac 12 + \varepsilon + iT} \int_{\frac 12 + \varepsilon - iT}^{\frac 12 + \varepsilon + iT} \int_{(\frac{\varepsilon}{4})} \int_{(\frac{\varepsilon}{4})} \widetilde{\mathcal W}_3 \left( s_1, s_2 ; z \right) \widetilde{V}(z_1) \widetilde{V}(z_2)   \\
&   \cdot \frac{\zeta(1-z) }{\zeta(1 + z) }Q^{\frac 32(s_1 + s_2)- \frac z2} M^{z_1} N^{z_2} \mathcal F(s_1, s_2, z_1, z_2; z)  \> dz_1 \> dz_2 \>ds_1 \> ds_2 \> dz + O(Q^{\frac{13}{10} + \varepsilon}),
\end{align*}
where 
$$ \mathcal F(s_1, s_2, z_1, z_2; z) := \sumtwo_{ \substack{m, n \geq 1 } } \frac{\sigma(m; \al) \sigma(n; -\be)}{m^{1/2 + s_1 + z_1 } n^{1/2 + s_2 + z_2}}  \frac{g^z \mathcal K(-z; g, \m\n) }{\phi (g\m\n, 1 + z)}.$$

By the same arguments as in Section 10 in \cite{CIS}, we have
\begin{align*}
\mathcal F(s_1, s_2, z_1, z_2; z) = & \zeta(2-z) \prod_{i = 1}^3 \frac{\zeta\left( \frac 12 + s_1 + z_1 + \alpha_i \right)}{\zeta\left( \frac 12 + s_1 + z_1 + \alpha_i + 1 -z\right)} \prod_{j = 1}^3  \frac{\zeta\left( \frac 12 + s_2 + z_2 - \beta_j \right)}{\zeta\left( \frac 12 + s_2 + z_2 - \beta_j + 1 -z\right)} \\
& \cdot \prod_{i, j = 1}^3 \zeta\left( 1 + s_1 + z_1 + s_2 + z_2 - z + \alpha_i - \beta_j\right) \mathcal R(s_1, s_2, z_1, z_2; z),
\end{align*}
 where $\mathcal R(s_1, s_2, z_1, z_2; z)$ is absolutely convergent in a wider range of $s_1, s_2, z_1, z_2$ and $z$, a subset of which is the region 
 
 $$ \R (z) < \frac 32,  \ \ \ \ \frac 12 + \sum_{i = 1}^{2} \R (s_i +  z_i) > \R (z) + 2\max(|\alpha_i|, |\beta_j|), $$
 $$ \R(s_i + z_i) > \max(|\alpha_i|, |\beta_j|) , \ \ \ \textrm{and}  \ \ \  1 + \R(s_i + z_i ) > \R(z) + \max(|\alpha_i|, |\beta_j|).   $$
 
Now we move the lines of integration in $s_i$ to $\R(s_i) = 2\varepsilon $ for $i = 1, 2. $ We then pick up the residues at nine poles, which are of the form $s_1 = \frac 12 - z_1 - \alpha_\ell$ and $s_2 = \frac 12 - z_2 + \beta_k$ for $\ell, k = 1,2,3.$ We use \eqref{eqn:boundWt3} and the bound in Lemma \ref{le:zeta6th} to bound the remaining integrals by the same arguments as in the proof of \cite[Proposition 7.2]{CLMR}. This gives that

\begin{align*}
\mathcal {MBG}_1(M, N) =& \sumtwo_{\ell, k \in \{1,2,3\}} \frac{Q^{1 + \delta(\al, \be)}}{2} \frac{1}{(2\pi i )^3} \int_{(\varepsilon)} \int_{(\frac{\varepsilon}{4})} \int_{(\frac{\varepsilon}{4})}  \widetilde{\mathcal W}_3 \left( \frac 12 - z_1 - \alpha_\ell , \frac 12 - z_2 + \beta_k ; z \right) \widetilde{V}(z_1) \widetilde{V}(z_2) \\
&  \cdot Q^{\frac 32(1-z_1 - z_2 - \alpha_\ell + \beta_k)- \frac z2} M^{z_1} N^{z_2}   \frac{ \zeta(2-z) \zeta(1-z)}{\zeta(1 + z) }  \prod_{\substack{i = 1,2, 3 \\ i \neq \ell}} \frac{\zeta\left( 1 - \alpha_\ell + \alpha_i \right)}{\zeta(2+\alpha_i-\alpha_\ell-z)}  \\
& \cdot \prod_{\substack{j = 1,2, 3 \\ j \neq k}} \frac{\zeta\left( 1 + \beta_k - \beta_j \right)}{\zeta(2+\beta_k-\beta_j-z)}\prod_{\substack{1 \leq i, j \leq 3 \\ (i, j) \neq (\ell, k)} } \zeta\left( 2 - \alpha_\ell + \beta_k - z + \alpha_i - \beta_j\right) \\
& \cdot \mathcal R\left(\frac 12 - z_1 - \alpha_\ell , \frac 12 - z_2 + \beta_k, z_1, z_2; z\right)  \> dz_1 \> dz_2 \> dz + O\left( Q^{\frac {13}{10} + \varepsilon } (MN)^{\frac{\varepsilon}{4}}\right).
\end{align*}
Next, we move the lines of integration in $z_1$ to $ \R(z_1) = -3/4+\varepsilon/4$ and $z_2$ to $\R(z_2) = 0$. Then we move the line of integration in $z$ to $\R(z) = 3/4$. We picked these lines so that they are in the region of convergence of $\mathcal R$ and satisfy conditions in Lemma \ref{lem:MellinXYU}. Since $\widetilde {V}$ decays rapidly and $\widetilde {W}_3$ satisfies~\eqref{eqn:boundWt3}, we have that 
$$ \mathcal {MBG}_1(M, N) \ll Q^{1 + \frac{\varepsilon}{8}} Q^{\frac 32(1 + \frac {3}{4}-\frac{\varepsilon}{4}) - \frac{3}{8}} M^{-\frac {3}{4}+\frac{\varepsilon}{4}} \ll Q^{\frac{13}{4}-\frac{\varepsilon}{4}} M^{-\frac{3}{4}+\frac{\varepsilon}{4}}$$
and the claim follows since $M \gg Q^{2-\delta_0}$. 
\end{proof}

Lemma~\ref{lem:sumMBG1} implies that we can extract the main contribution of $$ \sumtwodee_{M, N \leq Q^{2 - \delta_0} } \mathcal {MBG}_1(M, N)  $$ from the whole range of dyadic summation $M, N$. 
	From Equation  \eqref{def:MBG1},
\begin{align*}
		\sumd_M\sumd_N \mathcal {MBG}_1(M, N) = &\frac{Q^{1 + \delta(\al, \be)}}{2} \sum_{ \substack{m, n \geq 1 \\ m \neq n } } \frac{\sigma(m; \al) \sigma(n; -\be)}{\sqrt{mn}} \\
		& \cdot \frac{1}{2\pi i } \int_{(\varepsilon)} \widetilde{\mathcal W}_1^\pm \left( \frac{g\m}{Q^{3/2}}, \frac{g\n}{Q^{3/2}} ; z \right) \frac{\zeta(1-z) \mathcal K(-z; g, \m\n) }{\zeta(1 + z) \phi (g\m\n, 1 + z)}\left( \frac{Q^{1/2}}{g}\right)^{-z} \> dz,
\end{align*}
which is the same expression as \cite[Equation (63)]{CIS} (although our definitions of $\widetilde{\mathcal{W}}_1$ differ). Next we use Equation \eqref{eqn:truncate2} to express $\widetilde{\mathcal W}_1^\pm$  as an integration over $s_1$ and $s_2$. Then we take advantage of the work in \cite[Section 10]{CIS}, which extracts from the above expression the 9 main terms in $\tQ(q;\al, \be)$ corresponding to when one $\alpha_i$ is interchanged with one $\beta_j$. We summarize the result in the proposition below. 
	
\begin{prop} \label{prop:maincont9terms}
Let $\mathcal {MBG}_1(M, N)$ be as in \eqref{def:MBG1combinepm} with $\mathcal{MBG}_1^\pm(M, N)$ as in~\eqref{def:MBG1}. Then 
\begin{align*}
  \sumd_M\sumd_N \mathcal {MBG}_1(M, N) = H(0; \al, \be) \sum_q \Psi \left( \frac{q}{Q}\right) \phi^{\flat}(q) \sum_{\substack{\pi \in S_6/S_3 \times S_3 \\ \pi \ \textrm{permutes exactly } \\ \textrm{one $\alpha_i$ and $\beta_j$}} }\mathcal Q(q; \pi(\al), \pi(\be)) + O\left(Q^{\frac {13}{10} + \varepsilon}\right).
\end{align*}
\end{prop}	
We refer readers to~\cite[Section 10]{CIS} for the proof, see in particular~\cite[Equation (66)]{CIS} for the main term. The error term arises from treating the integrals over $s_1$ and $s_2$, analogous to the proof of Lemma \ref{lem:MBGunbalancedMN}.
\subsection{Bounding the error terms} \label{ssec:errorEg}
Now we consider the term $\mathcal {EBG}^\pm (M, N)$ defined in \eqref{def:EBG}. We will show that this contribution is negligible below.

\begin{lem} \label{lem:calcEBG} Let $\varepsilon, \delta_0 \in (0, 1/8)$ be fixed and let $D_0 \geq 1/2$. 
	We have, whenever $\max\{M, N\} \leq Q^{2-\delta_0}$, 
	$$ \mathcal {EBG}^\pm(M, N) \ll Q^{2 - \delta_0 + \varepsilon} D_0    + Q^{5/4 + \varepsilon} D_0^{1/2}. $$
\end{lem}

\begin{proof}
	
The proof is a small modification of the arguments used to derive \cite[Equation (55)]{CIS}, but using the fact that $MN \ll Q^{3 + \varepsilon}$. The first step (see~\cite[Beginning of Section 8]{CIS}) is to truncate the sum over $a$ to be bounded by $2Q$. Utilizing Remark~\ref{rem:gsize} we also make the truncations $b, g \leq Q^{3/2+\varepsilon}$. 
Using then the Mellin transform in Lemma \ref{lem:MellinXY} and the Mellin transform of $V$ in \eqref{def:MellinV}, we have that

\begin{align*}
\mathcal{EBG}^\pm(M, N) =& \frac{ Q^{1 + \delta(\al, \be)}}{2}  \sumthree_{\substack{a \leq 2Q \; b \leq Q^{3/2+\varepsilon} \\ h \leq 10 \cdot Q^{1-\delta_0}D_0 }} \sum_{ \substack{\psi \Mod{abh} \\ \psi \neq \psi_0} } \sum_{\substack{g \leq Q^{3/2+\varepsilon} \\ b| g, (a, g) = 1}}\sum_{ \substack{d \leq D_0 \\ (d, g) = 1 }}   \frac{\mu(d) \mu(a) \mu(b) }{adg \phi(abh)} \\
&\cdot\frac{1}{(2\pi i)^4} \int_{(\varepsilon)} \int_{(\varepsilon)}\int_{\left( \frac12 + \varepsilon \right)} \int_{\left( \frac12 + \varepsilon \right)} \widetilde{\mathcal W}_2^\pm \left( s_1, s_2; \frac{Q^{1/2}d}{gh} \right) \widetilde{V}(z_1) \widetilde{V}(z_2)\left( \frac{Q^{3/2}}{g}\right)^{s_1 + s_2} \frac{M^{z_1}N^{z_2}}{g^{z_1 + z_2}}\\
&\cdot \sumtwo_{\substack{\m, \n\\ \m\neq \n , (\m, \n) = 1 \\ (\m \n, d) = 1}} \frac{\sigma(g\m; \al) \sigma(g\n; -\be) }{\m^{\frac 12 + s_1 + z_1} \n^{\frac 12 + s_2 + z_2}}   \psi(\m) \overline{\psi}(\mp \n) \>ds_1  \>ds_2 \>dz_1\>dz_2 +  O \left( Q^{\frac 32 + 3\varepsilon} \right).
\end{align*}

Next we express the sums over $\m, \n$ in terms of product of $L$-functions. Since $\psi$ is not a trivial character, $L$-functions have no poles. As in~\cite[(53)--(54)]{CIS} we can move the line of integration over $s_i$ to $\R(s_i) = \varepsilon$. Then we change variables, letting $w_i = s_i + z_i$, and obtain that the contribution to $\mathcal {EBG}^\pm(M, N)$ of the main term above is bounded by 
\begin{align}\label{eqn:errorbeforedyadic}
  \begin{aligned}
&\ll Q^{1 + 7\varepsilon} \sumthree_{\substack{a \leq 2Q, \; b \leq Q^{3/2+\varepsilon} \\ h \leq 10 \cdot Q^{1-\delta_0}D_0 }} \sum_{ \substack{\psi \Mod{abh} \\ \psi \neq \psi_0} } \sum_{\substack{g \leq Q^{3/2+\varepsilon} \\ b| g, (a, g) = 1}}\sum_{ \substack{d \leq D_0 \\ (d, g) = 1 }}   \frac{1}{adg \phi(abh)} \int_{(\varepsilon)}\int_{(\varepsilon)} \int_{(2\varepsilon)}\int_{(2\varepsilon)}  \left| \widetilde{\mathcal W}_2^\pm \left( s_1, s_2; \frac{Q^{1/2}d}{gh} \right) \right|\\
& \cdot  \left|\widetilde{V}(w_1 - s_1) \right| \left|\widetilde{V}(w_2 - s_2) \right| \left( 1 + \prod_{i = 1}^3 \left| L \left( \frac 12 +  w_1 + \alpha_i, \psi\right)L \left( \frac 12 + w_2 - \beta_i, \overline{\psi}\right)\right| \right) \>dw_1 \>dw_2 \> ds_1 \> ds_2 .
  \end{aligned}
\end{align}

We consider the sums over $g$ and $d$ and apply the bound for  $\widetilde{\mathcal W}_2^{\pm}$ in Lemma \ref{lem:MellinXY} to derive that for any $h \geq 1,$ and $s_1, s_2$ with $\R(s_i) = \varepsilon$, and any fixed natural number $k$,
\begin{align}
\label{eq:W2etcbound}
\begin{aligned}
\sum_{\substack{g \leq Q^{3/2+\varepsilon} \\ b|g} } \frac 1g\sum_{d \leq D_0} \frac 1d \left| \widetilde{\mathcal W}_2^\pm \left(s_1, s_2; \frac{Q^{1/2} d}{gh} \right)\right| &\ll Q^\varepsilon \sum_{\substack{g \leq Q^{3/2+\varepsilon} \\ b|g} } \frac 1g \sum_{d \leq D_0} \frac 1d  \frac{1}{\max \{|s_1|, |s_2| \}^{k} |s_1 + s_2|^3} \left( 1 + \frac{Q^{1/2}D_0}{gh} \right)^{k-1} \\
&\ll  \frac{Q^{2\varepsilon}}{b\max \{|s_1| + 1, |s_2| + 1\}^{k} (|s_1 + s_2| + 1)^3} \left( 1 + \frac{Q^{1/2}D_0}{bh} \right)^{k-1}.
\end{aligned}
\end{align}
Notice first that, for any $\ell \geq 0$, $\widetilde{V}(z) \ll \frac{1}{(1 + |z|)^\ell}$, so the contribution of $|w_i - s_i| \gg Q^{\varepsilon}$ to \eqref{eqn:errorbeforedyadic} is $\ll_A Q^{-A}$.

Let us now return to~\eqref{eqn:errorbeforedyadic}. We divide the variables $a, b, h$ into dyadic blocks $a \sim A, b \sim B, h \sim H$ (with $A \ll Q, \, B \ll Q^{3/2+\varepsilon}$, and $H \ll Q^{1-\delta_0} D_0$), let $\ell = abh$, and also divide $w_1$ and $w_2$ into blocks such that $\max\{|w_1| + 1, |w_2| + 1\} \sim T$.  Then we apply~\eqref{eq:W2etcbound} and H\"older's inequality to  \eqref{eqn:errorbeforedyadic}. We derive that for any $k$, the contribution of each block is 
\begin{align}
\label{eq:dyadicblockcontr}
\begin{aligned}
&\ll \frac{Q^{1+ 9\varepsilon}}{A^{2 - \varepsilon} B^{2-\varepsilon} H^{1 - \varepsilon} } \left( 1 + \frac{Q^{1/2}D_0}{BH}\right)^{k - 1} \frac{1}{T^k}   \\ 
& \hskip 0.2in \cdot \max_{\substack{\alpha \in \{\alpha_1, \alpha_2, \alpha_3, \\ \beta_1, \beta_2, \beta_3\}}} \sum_{ABH \leq \ell < 8ABH} \sum_{ \substack{\psi \Mod{\ell} \\ \psi \neq \psi_0} }  \int_{-4T}^{4T} \left(  1 +   \left| L \left( \frac 12 + 2\varepsilon + it + \alpha, \psi \right)\right|^6 \right) \>dt.
\end{aligned}
\end{align}
As in \cite[Section 8]{CIS}, we apply the large sieve inequality. The precise bound we need is the sixth moment variant of~\cite[Proposition 3.2]{CLMR} which follows completely similarly (morally bounding the sixth moment corresponds to, by the approximate functional equation, bounding the mean square of a Dirichlet polynomial of length $\ll (ABHT)^{3/2}$ over a set of size $\ll (ABH)^2 T$)).  Consequently~\eqref{eq:dyadicblockcontr} is 
\begin{align}
\label{eq:dyadicblockcontr2}
\begin{aligned}
\ll  \frac{Q^{1+ 10\varepsilon}}{A^{2 - \varepsilon} B^{2-\varepsilon} H^{1 - \varepsilon} } \left( 1 + \frac{Q^{1/2}D_0}{BH}\right)^{k - 1} \frac{1}{T^k} \left( T(ABH)^2 +  (TABH)^{3/2}\right).
\end{aligned}
\end{align}

When $T \leq  1 + \frac{Q^{1/2}D_0}{BH}$, we choose $k = 1$, and otherwise, we choose $k = 4,$ so in any case~\eqref{eq:dyadicblockcontr2} is
\[
\ll Q^{1+15\varepsilon} \left(H +\left( 1 + \frac{Q^{1/2}D_0}{BH}\right)^{1/2}H^{1/2} \right).
\]

Recall that $H \ll {Q^{1 - \delta_0} D_0}$. Thus after dyadic summation $A, B, H, T$, we derive that the contribution to \eqref{eqn:errorbeforedyadic} from this case  is bounded by
\begin{align*}
 &\ll Q^{2 - \delta_0 + 16\varepsilon} D_0    + Q^{5/4 + 16\varepsilon} D_0^{1/2},
\end{align*}
so the claim follows by adjusting $\varepsilon$.
\end{proof}

\subsection{Proof of Proposition \ref{prop:offdiagonal}} \label{ssec:conclusionBS}

From \eqref{def:MBEB}, \eqref{eqn:BGrelatedtoMBGandEBG}, \eqref{eqn:combineBDandMBG}, and Lemma \ref{lem:calcEBG}, we derive that
\begin{align*} \mathcal {BD}(M, N) + \mathcal {BG}(M, N)  
  &=  \mathcal{MBG}_1(M, N) + O \left( \frac{Q^{2 + \varepsilon}}{D_0} + D_0Q^{3/2 + \varepsilon}  + Q^{2 - \delta_0 + \varepsilon} D_0 +  Q^{5/4 + \varepsilon} D_0^{1/2} \right).
\end{align*}
Then from \eqref{eqn:BSinitial}, Lemma \ref{lem:sumMBG1} and Proposition \ref{prop:maincont9terms}, we obtain that
\begin{align*} \mathcal B \mathcal S(\Psi, Q; \al, \be) &= H(0; \al, \be) \sum_q \Psi \left( \frac{q}{Q}\right) \phi^{\flat}(q) \sum_{\substack{\pi \in S_6/S_3 \times S_3 \\ \pi \ \textrm{permutes exactly } \\ \textrm{one $\alpha_i$ and $\beta_j$}} }\mathcal Q(q; \pi(\al), \pi(\be)) \\
                                                     &\hskip 0.5in + O\left( \frac{Q^{2 + \varepsilon}}{D_0} + D_0Q^{\frac 32 + \varepsilon}  + Q^{2 - \delta_0 + \varepsilon} D_0 + Q^{2 - \frac{1}{4} + \frac{3\delta_0}{4} + \varepsilon} \right).
\end{align*}

To balance the error terms $\frac{Q^{2 + \varepsilon}}{D_0}$ and $Q^{2 - \delta_0 + \varepsilon} D_0$ we choose $D_0 = Q^{\delta_0/2}$. Then the error terms is 
\begin{align*}
  O\left( Q^{2 - \frac{\delta_0}2 + \varepsilon} + Q^{2 - \frac{1}{4} + \frac{3\delta_0}{4} + \varepsilon} \right),
\end{align*}
so the claim follows.

\section{Unbalanced sums}\label{sec:unbalancedprelim}
We now prepare to prove Proposition \ref{prop:unbalanced}. Recall that we are interested in bounding 
$$\sum_{q} \Psi\bfrac{q}{Q} \sumbq  S(M, N),
$$when  $Q^{2-\delta_0} \le M \le Q^{3 + \varepsilon}$.  

\subsection{Notational simplification}
To simplify notation, we set $\al = \be = (0, 0, 0)$.  The case where the shifts are nonzero may be proven similarly with no conceptual change. To be precise, we start by writing
\begin{align*}
S(M, N) =  \sumtwo_{m, n} \frac{\tau_3(m) \tau_3(n) \chi(m) \cb(n)}{\sqrt{mn}} W_{0, 0}\left(m, n; q\right) V\bfrac{m}{M} V\bfrac{n}{N},
\end{align*}where $\tau_3(l)$ denotes the number of ways of writing $l$ as a product of three natural numbers.  

Recalling from \eqref{eqn:Walbe} that
$$
W_{0, 0}(m, n; q) = \frac{1}{2\pi i} \int_{\left(\frac{1}{\log Q}\right)} G \left(\frac 12 + s; 0, 0\right) H(s; 0, 0) \left( \frac{mn\pi^3}{q^3} \right)^{-s} \frac{ds}{s},
$$and the rapid decay of $G$ (which follows from the definition~\eqref{eq:Gdef} and Stirling's formula), we see that it suffices to bound
$$\sumtwo_{m, n} \frac{\tau_3(m) \tau_3(n) \chi(m) \cb(n)}{(mn)^{1/2+s}} V\bfrac{m}{M} V\bfrac{n}{N}
$$for $|s| \ll q^\varepsilon$ and $\tRe s = \frac{1}{\log Q}$.  We will further allow ourselves to rewrite the above as 
$$\sumtwo_{m, n} \frac{\tau_3(m) \tau_3(n) \chi(m) \cb(n)}{(mn)^{1/2}} V\bfrac{m}{M} V\bfrac{n}{N}
$$for slightly different functions $V$, where now
\begin{equation}\label{eqn:Vnew}
V^{(k)}{(x)} \ll q^\varepsilon,
\end{equation}for all integer $k\ge 0$. We shall assume this throughout the rest of the paper.

We now write $m = efg$, and apply a smooth partition of unity to $e, f$ and $g$, to see that our sum is now at most $\log^3 Q$ sums of the form
\begin{align*}
\sumfour_{e, f, g, n} \frac{\tau_3(n) \chi(efg) \cb(n)}{(efgn)^{1/2}} V\bfrac{e}{E} V\bfrac{f}{F} V\bfrac{g}{G} V\bfrac{efg}{M} V\bfrac{n}{N},
\end{align*}where $EFG \asymp M$. Without loss of generality we may assume that $E \geq F \geq G$. We may again neglect the factor $V\bfrac{efg}{M}$ in the same manner in which we removed $ W_{0,0}\left(m, n; q\right)$, and absorb a factor of $\frac{\sqrt{EFGN}}{\sqrt{efgn}}$ into the smooth functions $V$. Thus we will examine
\begin{align*}
S(E, F, G, N) := \frac{1}{\sqrt{EFGN}}\sumfour_{e, f, g, n} \tau_3(n) \chi(efg) \cb(n) V\bfrac{e}{E} V\bfrac{f}{F} V\bfrac{g}{G} V\bfrac{n}{N}
\end{align*}
for $EFG \asymp M$. Again, the weight functions $V$ have changed slightly, but still satisfy \eqref{eqn:Vnew}.

\subsection{Initial manipulations}
By Lemma \ref{lem:orthogonal}, we have that
\begin{align*}
&\sum_q \Psi\bfrac{q}{Q}\sumbq  S(E, F, G, N) \\
&= \frac 12 \sum_{d} \sum_r \Psi\bfrac{dr}{Q} \mu(d) \phi(r) \frac{1}{\sqrt{EFGN}} \sumfour_{\substack{e, f, g, n \\ e \equiv \pm \overline{fg} n \Mod r\\ (efgn, dr) = 1}} \tau_3(n) V\bfrac{e}{E} V\bfrac{f}{F} V\bfrac{g}{G} V\bfrac{n}{N}.
\end{align*}

The conditions $e \equiv \overline{fg} n \Mod r$ and $e \equiv -\overline{fg} n \Mod r$ are dealt with by similar methods, so we examine the case $e \equiv \overline{fg} n \Mod r$ only.  Thus, we will focus our attention on
\begin{align*}
\calS := \sum_{d} \sum_r \Psi\bfrac{dr}{Q} \mu(d) \phi(r) \frac{1}{\sqrt{EFGN}} \sumfour_{\substack{e, f, g, n\\ e \equiv \overline{fg} n \Mod r\\ (efgn, dr) = 1}} \tau_3(n) V\bfrac{e}{E} V\bfrac{f}{F} V\bfrac{g}{G} V\bfrac{n}{N}.
\end{align*}

We shall apply Poisson summation to two or three of the variables. As usual, we write
\[
\widehat{V}(\xi) = \int_{-\infty}^\infty V(x) \ex(-\xi x) dx
\]
for the Fourier transform of $V$. For clarity, we record the following lemma, which is essentially an application of Poisson summation.

\begin{lem}\label{lem:firstpoisson}
Let $r, f, g, n, \lambda \in \mathbb{N}$ with $(fgn, r) = 1$. Then
\begin{align*}
\sum_{\substack{e\equiv \overline{fg} n \Mod r \\ (e, \lambda) = 1}} V\bfrac{e}{E} = \frac{E}{r} \sum_{\substack{\nu_1 | \lambda \\ (\nu_1, r) = 1}} \frac{\mu(\nu_1)}{\nu_1} \sum_e \mathrm{e}\bfrac{ne\overline{\nu_1 fg}}{r} \widehat V\bfrac{Ee}{\nu_1 r}.
\end{align*}
\end{lem}

\begin{proof}
	Detecting the condition $(e, \lambda) = 1$ by Möbius inversion (introducing $\mu(\nu_1)$), we have
\begin{align*}
\sum_{\substack{e\equiv \overline{fg} n \Mod r\\ (e, \lambda) = 1}} V\bfrac{e}{E} 
= \sum_{\substack{\nu_1 | \lambda}} \mu(\nu_1) \sum_{\substack{e \equiv \overline{f g} n \Mod{r} \\ \nu_1 | e}} V \bfrac{e}{E}.
\end{align*}
Note that $(e, r) = 1$ since $e \equiv \overline{f g} n \Mod{r}$ and $(f g n, r) = 1$. Thus $(\nu_1, r) = 1$. Making a change of variable and applying Poisson summation, we see that the above is equal to
\begin{align*}
\sum_{\substack{\nu_1 | \lambda \\ (\nu_1, r) = 1}} \mu(\nu_1) \sum_{e\equiv \overline{\nu_1 fg} n \Mod r} V\bfrac{\nu_1 e}{E} &= \sum_{\substack{\nu_1 | \lambda \\ (\nu_1, r) = 1}} \mu(\nu_1)\sum_m V\bfrac{\nu_1(mr + \overline{\nu_1 fg}n)}{E}\\
&= \sum_{\substack{\nu_1 | \lambda \\ (\nu_1, r) = 1}} \mu(\nu_1) \sum_e \int V\bfrac{\nu_1(tr + \overline{\nu_1 fg}n)}{E} \ex(-et) dt,
\end{align*}which gives the desired result upon a change of variables $y = \frac{\nu_1(tr + \overline{\nu_1 fg}n)}{E}$. 
\end{proof}

By Lemma \ref{lem:firstpoisson}, we see that 
\begin{align*}
\calS &= \frac{\sqrt{E}}{\sqrt{FGN}} \sum_{d} \sum_r \sum_{\substack{\nu_1 | d \\ (\nu_1, r) = 1}}  \Psi\bfrac{dr}{Q} \mu(d) \frac{\phi(r)}{r} \frac{\mu(\nu_1)}{\nu_1}  \\
&\cdot \sumfour_{\substack{e, f, g, n \\ (fgn, dr) = 1}} \ex\bfrac{ne\overline{\nu_1 fg}}{r} \widehat V\bfrac{Ee}{\nu_1 r} \tau_3(n)  V\bfrac{f}{F} V\bfrac{g}{G} V\bfrac{n}{N}.
\end{align*}
Here, it is convenient to isolate the contribution of the $e = 0$ term, which is
\begin{align}
\begin{aligned}
\label{eqn:e0contribution}
&\widehat V(0)  \frac{\sqrt{E}}{\sqrt{FGN}} \sum_{d} \sum_r \sum_{\substack{\nu_1 | d \\ (\nu_1, r) = 1}} \Psi\bfrac{dr}{Q} \mu(d) \frac{\phi(r)}{r}  \frac{\mu(\nu_1)}{\nu_1} \\
&\cdot \sumthree_{\substack{f, g, n \\ (fgn, dr) = 1}} \tau_3(n) V\bfrac{f}{F} V\bfrac{g}{G} V\bfrac{n}{N}.
\end{aligned}
\end{align}
For given $q \in \mathbb{N}$, consider the contribution of $d$ and $r$ such that $dr = q$ to~\eqref{eqn:e0contribution}. We have that
\begin{align*}
\sum_{d|q} \mu(d) \frac{\phi(q/d)}{q/d} \sum_{\substack{\nu_1| d \\ (\nu_1, q/d) = 1}} \frac{\mu(\nu_1)}{\nu_1} =
 \sum_{d|q} \mu(d) \prod_{p|q/d} \left(1-\frac 1p\right) \prod_{\substack{p|d\\ p\nmid q/d}} \left(1-\frac 1p\right) = \frac{\phi(q)}{q} \sum_{d \mid q} \mu(d)
 =0
\end{align*}
for $q>1$, as is the case for us. Thus, the quantity in \eqref{eqn:e0contribution} vanishes and so
\begin{align*}
\calS &= \frac{\sqrt{E}}{\sqrt{FGN}} \sum_{d} \sum_r \sum_{\substack{\nu_1 | d \\ (\nu_1, r) = 1}}  \Psi\bfrac{dr}{Q} \mu(d) \frac{\phi(r)}{r} \frac{\mu(\nu_1)}{\nu_1}  \\
&\cdot \sum_{e \neq 0} \sumthree_{\substack{f, g, n \\ (fgn, dr) = 1}} \ex\bfrac{ne\overline{\nu_1 fg}}{r} \widehat V\bfrac{Ee}{\nu_1 r} \tau_3(n)  V\bfrac{f}{F} V\bfrac{g}{G} V\bfrac{n}{N}.
\end{align*}

Next we remove the condition $(n, r) = 1$ by M\"obius inversion (introducing $\mu(\gamma)$), getting that
\begin{align*}
\calS &= \frac{\sqrt{E}}{\sqrt{FGN}} \sum_{\substack{d, \gamma \\ (d, \gamma) = 1}} \mu(d) \mu(\gamma) \sum_r \sum_{\substack{\nu_1 | d \\ (\nu_1, r) = 1}}  \Psi\bfrac{d\gamma r}{Q}  \frac{\phi(\gamma r)}{\gamma r} \frac{\mu(\nu_1)}{\nu_1}  \\
&\cdot \sum_{e \neq 0} \sumthree_{\substack{f, g, n \\ (fg, d\gamma r) = 1 \\ (n, d) = 1}} \ex\bfrac{ne\overline{\nu_1 fg}}{r} \widehat V\bfrac{Ee}{\nu_1 \gamma r} \tau_3(\gamma n)  V\bfrac{f}{F} V\bfrac{g}{G} V\bfrac{\gamma n}{N}.
\end{align*}
We also divide out common factors of $e$ and $r$ at this point. In order to do this, we write $r\kappa$ for $r$ and $e \kappa$ for $e$ where for the new variables $(e, r) = 1$. Hence
\begin{align*}
\calS &= \frac{\sqrt{E}}{\sqrt{FGN}} \sum_{\substack{d, \gamma \\ (d, \gamma) = 1}} \mu(d) \mu(\gamma) \sum_{\kappa} \sum_r \sum_{\substack{\nu_1 | d \\ (\nu_1, r\kappa) = 1}}  \Psi\bfrac{d\gamma \kappa r}{Q}  \frac{\phi(\gamma \kappa r)}{\gamma \kappa r} \frac{\mu(\nu_1)}{\nu_1}  \\
&\cdot\sum_{\substack{e \neq 0 \\ (e, r) = 1}} \sumthree_{\substack{f, g, n \\ (fg, d\gamma\kappa r) = 1 \\ (n, d) = 1}} \ex\bfrac{ne\overline{\nu_1 fg}}{r} \widehat V\bfrac{Ee}{\nu_1 \gamma r} \tau_3(\gamma n)  V\bfrac{f}{F} V\bfrac{g}{G} V\bfrac{\gamma n}{N}.
\end{align*}

Our next step is to apply Poisson summation to the variable $f$, and for this we record the following lemma. There and later we write, for $a, b, c \in \mathbb{N}$,
\[
S(a, b; c) := \sumstar_{x \Mod{c}} \ex\left(\frac{ax+b\overline{x}}{c}\right)
\]
for the classical Kloosterman sum.
\begin{lem}\label{lem:2timespoisson}
Let $n, e, \nu_1, g, r, \alpha \in \mathbb{N}$ with $(g\nu_1, r) = 1$. Then
$$
\sum_{\substack{f\\(f, \alpha r)=1}} \ex\bfrac{ ne\overline{\nu_1 fg}}{r} V\bfrac{f}{F} = \frac{F}{r} \sum_{\substack{\nu_2|\alpha \\ (\nu_2, r) = 1}} \frac{\mu(\nu_2)}{\nu_2} \sum_{f} S(\overline{\nu_2} f, ne\overline{\nu_1 g}; r)\widehat{V}\bfrac{fF}{\nu_2 r}.
$$
\end{lem}
\begin{proof}
Let $c = ne\overline{\nu_1 g}$. We have by M\"obius inversion, applied to the condition $(f,\alpha) = 1$, 
\begin{align*}
\sum_{\substack{f\\(f, \alpha r)=1}} \ex\bfrac{c \overline{f}}{r} V\bfrac{f}{F}
= \sum_{\nu_2 | \alpha} \mu(\nu_2) \sum_{\substack{(f,r) = 1 \\ \nu_2 | f}} \ex\bfrac{c \overline{f}}{r} V\bfrac{f}{F}.
\end{align*}
Since $(f,r) = 1$ and $\nu_2 | r$ we also have $(\nu_2, r) = 1$. Furthermore we can replace the condition $(f,r) = 1$ by opening into a sum over arithmetic progressions $f \equiv a \Mod{r}$ with $(a,r) = 1$. Making a change of variable $f \mapsto f \nu_2$ to remove the condition $\nu_2 | f$ we get that the above is equal to 
\begin{align*}
	& \sum_{\substack{\nu_2|\alpha \\ (\nu_2, r) = 1}} \mu(\nu_2) \sumstar_{a \Mod r} \ex\bfrac{c\overline a}{r} \sum_{f \equiv \overline{\nu_2} a \Mod r} V\bfrac{f\nu_2}{F}\\ & = \sum_{\substack{\nu_2|\alpha \\ (\nu_2, r) = 1}} \mu(\nu_2) \sumstar_{a \Mod r} \ex\bfrac{c\overline a}{r} \frac{F}{\nu_2 r} \sum_{f} \ex\bfrac{\overline{\nu_2}af}{r} \widehat{V}\bfrac{fF}{\nu_2 r}\\
&= \sum_{\substack{\nu_2|\alpha \\ (\nu_2, r) = 1}} \mu(\nu_2) \frac{F}{\nu_2 r} \sum_{f} S(\overline{\nu_2} f, c; r)\widehat{V}\bfrac{fF}{\nu_2 r}.
\end{align*}
\end{proof}

Applying Lemma \ref{lem:2timespoisson} with $\alpha = d\gamma \kappa$, we obtain
\begin{align*}
\mathcal{S} &= \frac{\sqrt{EF}}{\sqrt{GN}} \sum_{\substack{d, \gamma \\ (d, \gamma) = 1}}\mu(d) \mu (\gamma) \sum_{\kappa} \sum_{r} \frac{\phi(\gamma \kappa r)}{\gamma \kappa r^2} \Psi\bfrac{d\gamma \kappa r}{Q}  \sum_{\substack{\nu_1 \mid d \\ (\nu_1, r\kappa) = 1}} \frac{\mu(\nu_1)}{\nu_1} \sum_{\substack{\nu_2 \mid d\gamma \kappa \\ (\nu_2, r) = 1}}  \frac{\mu(\nu_2)}{\nu_2} \notag \\
&\cdot \sum_{\substack{e \neq 0 \\ (e, r) = 1}} \sumthree_{\substack{f, g, n \\(g, d\gamma \kappa r) = 1\\ (n, d)=1}}  S(\overline{\nu_2} f, ne\overline{\nu_1 g}; r) \tau_3(\gamma n)  \widehat{V}\bfrac{fF}{\nu_2 r} \widehat V\bfrac{Ee}{\nu_1\gamma r} V\bfrac{g}{G} V\bfrac{\gamma n}{N}.
\end{align*}
It is convenient to isolate the contribution arising from $f = 0$, which is
\begin{align*}
\mathcal{S}(f=0) &:=  \widehat{V}(0) \frac{\sqrt{EF}}{\sqrt{GN}} \sum_{\substack{d, \gamma \\ (d, \gamma) = 1}}\mu(d) \mu (\gamma) \sum_{\kappa}  \sum_{r} \frac{\phi(\gamma \kappa r)}{\gamma \kappa r^2} \Psi\bfrac{d\gamma \kappa r}{Q}  \sum_{\substack{\nu_1 \mid d \\ (\nu_1, r\kappa) = 1}} \frac{\mu(\nu_1)}{\nu_1} \sum_{\substack{\nu_2 \mid d\gamma \kappa \\ (\nu_2, r) = 1}}  \frac{\mu(\nu_2)}{\nu_2} \notag \\
&\cdot \sum_{\substack{e \neq 0 \\ (e, r) = 1}} \sumtwo_{\substack{g, n  \\ (g, d\gamma \kappa r) = 1\\ (n, d)=1}}   \tau_3(\gamma n) \fr_r(ne\overline{\nu_1 g}) \widehat V\bfrac{Ee}{\nu_1\gamma r} V\bfrac{g}{G} V\bfrac{\gamma n}{N},
\end{align*}
where we have the usual Ramanujan sum (see e.g.~\cite[formula (3.5)]{IK})
$$\fr_r(ne\overline{\nu_1 g}) := \sumstar_{a\Mod r} \ex\bfrac{ane\overline{\nu_1 g}}{r} = \sumstar_{a\Mod r} \ex\bfrac{a n}{r} = \mu\bfrac{r}{(n, r)} \frac{\phi(r)}{\phi\bfrac{r}{(n, r)}} \ll (n, r).
$$

\begin{lem}\label{lem:d>Df=0} With the above notation, assuming $EFG \gg Q^{2-\delta_0}$, we have
\begin{align*}
\mathcal{S}(f=0) \ll Q^{11/6+\delta_0/3 + \varepsilon}.
\end{align*}
\end{lem}

\begin{proof}
We write $l = (n, r)$, and get
\begin{align*}
\mathcal{S}(f=0) &\ll  Q^{\varepsilon} \frac{\sqrt{EF}}{\sqrt{GN}} \sum_{\substack{d, \gamma \\ (d, \gamma) = 1}} \sum_{\kappa} \sum _{l \leq 2 Q} l \sum_{\substack{r\\ l|r}} \Psi\bfrac{d\gamma\kappa r}{Q}  \frac{1}{r} \sum_{\substack{\nu_1 \mid d \\ (\nu_1, r\kappa) = 1}} \frac{1}{\nu_1} \sum_{\substack{\nu_2 \mid d\gamma\kappa \\ (\nu_2, r) = 1}} \frac{1}{\nu_2} \notag \\
&\cdot \sum_{e \neq 0} \sum_g \sum_{\substack{n \\ l|n}}   \left|\widehat V\bfrac{Ee}{\nu_1 \gamma r} V\bfrac{g}{G} V\bfrac{\gamma n}{N}\right|.
\end{align*}
Since $\widehat V(x) \ll Q^\varepsilon x^{-A}$ for any $A$, here
$$\sum_{e \neq 0} \sum_{\substack{n\\ l|n}}   \left|V\left(\frac{\gamma n}{N}\right) \widehat V\bfrac{Ee}{\nu_1 \gamma r}\right| \ll Q^\varepsilon \frac{\nu_1 r\gamma}{E} \frac{N}{\gamma l}.
$$
Thus
\[
\mathcal{S}(f=0) \ll Q^{3\varepsilon} \frac{\sqrt{FGN}}{\sqrt{E}} \sum_{\substack{d, \gamma \\ (d, \gamma) = 1}} \sum_{\kappa} \sum_{l \leq 2Q} \sum_{\substack{r\\ l|r}} \Psi\bfrac{d \gamma \kappa r}{Q}  \frac{l}{r} \frac{r\gamma}{l\gamma}  \ll Q^{1+4\varepsilon} \sqrt{\frac{FGN}{E}}.
\]
Using $E \ge F\ge G$, $EFGN \ll Q^{3+\varepsilon}$ and $EFG \gg Q^{2-\delta_0}$, we see that
\[
\mathcal{S}(f=0) \ll Q^{1+4\varepsilon} \sqrt{\frac{Q^{3+\varepsilon}}{E^2}} \ll \frac{ Q^{5/2+5\varepsilon}}{Q^{(2-\delta_0)/3}},
\]
and the claim follows by adjusting $\varepsilon$.
\end{proof}

It now remains to examine
\begin{align} \label{masterEq}
  \begin{aligned}
\mathcal{S}(f\neq 0) &:= \frac{\sqrt{EF}}{\sqrt{GN}} \sum_{\substack{d, \gamma \\ (d, \gamma) = 1}}\mu(d) \mu (\gamma) \sum_{\kappa} \sum_{r} \frac{\phi(\gamma \kappa r)}{\gamma \kappa r^2} \Psi\bfrac{d\gamma \kappa r}{Q}  \sum_{\substack{\nu_1 \mid d \\ (\nu_1, r\kappa) = 1}} \frac{\mu(\nu_1)}{\nu_1} \sum_{\substack{\nu_2 \mid d\gamma \kappa \\ (\nu_2, r) = 1}}  \frac{\mu(\nu_2)}{\nu_2} \\
& \cdot \sumfour_{\substack{e , f, g, n \\ ef \neq 0 \\ (e, r) = (n, d) = 1 \\ (g, d\gamma \kappa r) = 1}}  S(\overline{\nu_2} f, ne\overline{\nu_1 g}; r) \tau_3(\gamma n)  \widehat{V}\bfrac{fF}{\nu_2 r} \widehat V\bfrac{Ee}{\nu_1 \gamma r} V\bfrac{g}{G} V\bfrac{\gamma n}{N}.
  \end{aligned}
  \end{align}
We split $\calS(f \neq 0)$ into two parts according to whether $d\gamma\kappa \leq D$ or not, writing, for a parameter $D \geq 1/2$,
\[
\mathcal{S}(f \neq 0) = \mathcal{S}(d\gamma\kappa \leq D, f \neq 0) + \mathcal{S}(d\gamma\kappa > D, f \neq 0).
\]
We shall prove the following two propositions 
\begin{prop}\label{prop:d>D}
Let $D \geq 1/2$ and $\delta_0 \in (0, 1/8)$. Assume that $EFG \geq Q^{2-\delta_0}$. We have
$$\mathcal{S}(d\gamma\kappa >D, f \neq 0) \ll   Q^{2+\varepsilon} \left(\frac{1}{D} + Q^{-1/6+\delta_0/3 }\right).
$$
Moreover, when $EFG \geq Q^{5/2 + \delta'}$ for some $\delta' \in (0, 1/2)$, we have, for any $\varepsilon > 0$,
$$\mathcal{S}(d\gamma\kappa >1/2, f \neq 0) \ll  Q^{2+\varepsilon} \left(Q^{-\delta'/2} + Q^{-1/6+\delta_0/3 } \right).
$$
\end{prop}

\begin{prop}\label{prop:d<D}
Let $D \geq 1/2$. Assume that $Q^{2-\delta_0} \leq EFG \leq Q^{5/2 + \delta'}$ for some $\delta' \in (0, 1/2)$ and  $\delta_0 \in (0, 1/8)$. Then we have, for any $\varepsilon > 0$,
\[
\mathcal{S}(d\gamma \kappa \leq D, f \neq 0) \ll Q^{2+\varepsilon} \left(D Q^{-11/384+  \delta_0 / 2 +11 \delta'/192} + DQ^{-11/192+ 139 \delta_0 / 192} \right).
\]
\end{prop}
Proposition~\ref{prop:d>D} will be proven in Section~\ref{sec:dgammabig} and Proposition~\ref{prop:d<D} will be proven in Section~\ref{se:dgamsmall} (after deriving the necessary bound for averages of Kloosterman sums in Section~\ref{sec:AverageKloo}).

\begin{proof}[Proof of Proposition~\ref{prop:unbalanced} assuming Propositions~\ref{prop:d>D} and~\ref{prop:d<D}]
Combining Propositions~\ref{prop:d>D} and~\ref{prop:d<D} with Lemma~\ref{lem:d>Df=0} we obtain
\[
\mathcal{S} \ll Q^{2+\varepsilon} \left(\frac{1}{D} + Q^{-1/6+\delta_0/3} + Q^{-\delta'/2}+ D Q^{-11/384+\delta_0 / 2  +11 \delta'/192} + DQ^{-11/192+ 139 \delta_0 / 192} \right).
\]
Now the second term is always smaller than the last two terms. Furthermore the fifth term is always smaller than the fourth term for $\delta_0 < \frac 18 < \frac{11}{86}$. Hence Proposition~\ref{prop:unbalanced} follows.
\end{proof}

\section{The case $d\gamma \kappa$ is large} \label{sec:dgammabig}
In case $d \gamma \kappa > D$, we proceed to apply Poisson summation one more time, to the sum over $g$. The following lemma takes care of this step.

\begin{lem}\label{lem:3timespoisson}
Let $\nu_1, \nu_2, f, n, e, r \in \mathbb{N}$ with $(\nu_1\nu_2, r) = 1$. Then
\begin{align*}
\sum_{\substack{g\\(g, \alpha r) = 1}}S(\overline{\nu_2} f, ne\overline{\nu_1 g}; r) V\bfrac{g}{G}= \sum_{\substack{\nu_3|\alpha \\ (\nu_3, r) = 1}}\frac{\mu(\nu_3)}{\nu_3} \frac{G}{r} \sum_{g} \mathcal {KS}(\overline{\nu_2}f, \overline{\nu_3} g, \overline{\nu_1 } ne;r) \widehat{V}\bfrac{gG}{\nu_3 r},
\end{align*}
where the hyper-Kloosterman sum $\mathcal {KS}(\overline{\nu_2}f, \overline{\nu_3}g, \overline{\nu_1} ne; r)$ is defined by
$$\mathcal {KS}(\overline{\nu_2} f, \overline{\nu_3} g, \overline{\nu_1} ne; r) := \sumtwostar_{a, b \Mod r} \ex\bfrac{a\overline{\nu_2} f + b \overline{\nu_3} g + \overline{ab \nu_1}ne}{r}.
$$
\end{lem}
\begin{proof}
We have
\begin{align*}
\sum_{\substack{g\\(g, \alpha r) = 1}}S(\overline{\nu_2}f, ne\overline{\nu_1 g}; r) V\bfrac{g}{G}
&= \sumstar_{a\Mod r} \ex\bfrac{a\overline{\nu_2}f}{r} \sum_{\substack{g\\(g, \alpha r) = 1}} \ex\bfrac{\overline{a}ne\overline{\nu_1 g}}{r} V\bfrac{g}{G}.
\end{align*}
Lemma~\ref{lem:2timespoisson} yields
\begin{align*}
\sum_{\substack{g\\ (g, \alpha r) = 1}} \ex\bfrac{\overline{a} ne \overline{\nu_1 g}}{r}  V\bfrac{g}{G} = \frac{G}{r} \sum_{\substack{\nu_3 \mid \alpha \\ (\nu_3, r) = 1}} \frac{\mu(\nu_3)}{\nu_3} \sum_g S(\overline{\nu_3}g, \overline{a \nu_1}ne; r) \widehat V\bfrac{Gg}{\nu_3 r},
\end{align*}
and the claim follows.
\end{proof}

Applying Lemma \ref{lem:3timespoisson} with $\alpha = d\gamma \kappa$ gives
\begin{align*}
\mathcal{S}(d\gamma\kappa >D, f\neq 0) &:= \frac{\sqrt{EFG}}{\sqrt{N}} \sum_{\substack{d, \gamma, \kappa \\ (d, \gamma) = 1 \\ d\gamma\kappa >D}}\sum_{r} \frac{\phi(\gamma \kappa r)}{\gamma \kappa r^3} \Psi\bfrac{d\gamma\kappa r}{Q} \mu(\gamma) \mu(d) \sum_{\substack{\nu_1 \mid d \\ (\nu_1, \kappa r) = 1}} \frac{\mu(\nu_1)}{\nu_1} \sumtwo_{\substack{\nu_2, \nu_3 \mid d\gamma\kappa \\ (\nu_2 \nu_3, r) = 1}} \frac{\mu(\nu_2)}{\nu_2}\frac{\mu(\nu_3)}{\nu_3} \\
&\cdot \sumfour_{\substack{e , f, g, n \\ ef \neq 0 \\ (e, r) = (n, d)=1}} \mathcal {KS}(\overline{\nu_2}f, \overline{\nu_3} g, \overline{\nu_1 } ne;r) \tau_3(\gamma n)   \widehat V\bfrac{Ee}{\nu_1\gamma r} \widehat{V}\bfrac{fF}{\nu_2 r} \widehat V\bfrac{Gg}{\nu_3 r} V\bfrac{\gamma n}{N}.
\end{align*}
We write $\mathcal{S}(d\gamma\kappa >D, f\neq 0, g=0)$ for the contribution of $g = 0$ terms to $\mathcal{S}(d\gamma\kappa >D, f\neq 0)$, and bound it using the following lemma.

\begin{lem}\label{lem:d>Dg=0}
We have, for any $D \geq 1/2$,
$$
\mathcal{S}(d\gamma \kappa >D, f \neq 0, g=0) \ll Q^{7/6 + 2\delta_0/3 + \varepsilon}.
$$
\end{lem}
\begin{proof}
Note that, for $(\nu_1 \nu_2, r) = 1$,
\begin{align*}
\mathcal {KS}(\overline{\nu_2}f,0, \overline{\nu_1} ne; r)
=\sumtwostar_{a, b \Mod r} \ex\bfrac{a\overline{\nu_2}f + \overline{ab\nu_1} ne }{r}
= \sumtwostar_{a, b \Mod r} \ex\bfrac{a f + b n }{r}
\end{align*}
by a change of variables, and the above is then $\fr_r(f) \fr_r(n)$ (recall that $\fr$ denotes the Ramanujan sum).  Using the bound 
$$\fr_r(f) \fr_r(n) \ll (r, f) (r, n),
$$
and writing $\ell = (r, f)$ and $k = (r, n)$, we see that
\begin{align*}
\mathcal{S}(d\gamma \kappa >D, f \neq 0, g=0) &\ll Q^\varepsilon \frac{\sqrt{EFG}}{\sqrt{N}} \sum_{\substack{d, \gamma, \kappa \\ D \leq d\gamma\kappa \leq 2Q}}\sum_{k, \ell \leq 2Q} k \ell \sum_{\substack{r \sim Q/(d \gamma \kappa) \\ k | r, \; \ell \mid r}} \frac{1}{r^2} \sum_{\nu_1 \mid d} \frac{1}{\nu_1} \sumtwo_{\substack{\nu_2, \nu_3 \mid d\gamma\kappa}} \frac{1}{\nu_2 \nu_3} \\
&\cdot \sumthree_{\substack{e, f, n \\ ef \neq 0 \\ k \mid n, \; \ell \mid f}} \left|\widehat V\bfrac{Ee}{\nu_1\gamma r} \widehat{V}\bfrac{fF}{\nu_2 r} V\bfrac{\gamma n}{N}\right|.
\end{align*}
Using that $e, f \neq 0$ and that, for any $A > 0$, $\widehat{V}(x) \ll Q^{\varepsilon } x^{-A}$, we obtain that, for any $k, \ell$,
\begin{align*}
&\sumthree_{\substack{e, f, n \\ ef \neq 0 \\ k \mid n, \; \ell \mid f}} \left|\widehat V\bfrac{Ee}{\nu_1\gamma r} \widehat{V}\bfrac{fF}{\nu_2 r} V\bfrac{\gamma n}{N}\right|  \ll Q^{3\varepsilon} \frac{N}{\gamma k} \frac{\nu_1 r \gamma }{E} \frac{\nu_2 r}{F\ell} \ll Q^{3\varepsilon} \frac{N \nu_1 \nu_2}{k\ell EF} r^2.
\end{align*}

Hence
\begin{align*}
\mathcal{S}(d \gamma \kappa >D, f\neq 0, g=0) &\ll Q^{4\varepsilon} \frac{\sqrt{GN}}{\sqrt{EF}} \sum_{\substack{d, \gamma, \kappa \\ D \leq d\gamma\kappa \leq 2Q}} \sumtwo_{k, \ell \leq 2Q} \sum_{\substack{r \sim Q/(d \gamma \kappa) \\k|r, \; \ell|r}} 1\\
&\ll  Q^{4\varepsilon} \frac{\sqrt{GN}}{\sqrt{EF}} \sum_{\substack{d, \gamma, \kappa \\ D \leq d \gamma \kappa \leq 2Q}} \sumtwo_{k, \ell \leq 2Q} \frac{Q}{d\gamma \kappa [k, \ell]} \ll  Q^{1+5\varepsilon} \frac{\sqrt{GN}}{\sqrt{EF}}.
\end{align*}
Since $E \geq F \geq G$, $NEFG \ll Q^{3+\varepsilon}$, and $EFG \gg Q^{2-\delta_0}$, we obtain 
\begin{align*}
\mathcal{S}(d \gamma \kappa >D, f\neq 0, g=0) \ll Q^{1+5\varepsilon}\sqrt{\frac{GN}{EF}} \ll Q^{1+ 5\varepsilon} \frac{Q^{3/2+\varepsilon/2}}{EF} \ll \frac{Q^{5/2+6\varepsilon}}{Q^{\frac{2}{3}(2-\delta_0)}} \ll Q^{7/6 + 2\delta_0/3 + 6\varepsilon},
\end{align*}
and the claim follows by adjusting $\varepsilon$.
\end{proof}

We now proceed to bound 
\begin{align}\label{eqn:d>Dfnotzerognotzero}
\begin{aligned}
&\mathcal{S}(d \gamma \kappa>D, f\neq 0, g\neq 0)\\
&:=  \frac{\sqrt{EFG}}{\sqrt{N}} \sum_{\substack{d, \gamma, \kappa \\ (d, \gamma) = 1 \\ d\gamma\kappa >D}}\sum_{r} \frac{\phi(\gamma \kappa r)}{\gamma \kappa r^3} \Psi\bfrac{d\gamma\kappa r}{Q} \mu(\gamma) \mu(d) \sum_{\substack{\nu_1 \mid d \\ (\nu_1, \kappa r) = 1}} \frac{\mu(\nu_1)}{\nu_1} \sumtwo_{\substack{\nu_2, \nu_3 \mid d\gamma\kappa \\ (\nu_2 \nu_3, r) = 1}} \frac{\mu(\nu_2)}{\nu_2}\frac{\mu(\nu_3)}{\nu_3} \\
&\quad \cdot \sumfour_{\substack{e , f, g, n \\ efg \neq 0 \\ (e, r) = (n, d)=1}} \mathcal {KS}(\overline{\nu_2}f, \overline{\nu_3} g, \overline{\nu_1 } ne;r) \tau_3(\gamma n)   \widehat V\bfrac{Ee}{\nu_1\gamma r} \widehat{V}\bfrac{fF}{\nu_2 r} \widehat V\bfrac{Gg}{\nu_3 r} V\bfrac{\gamma n}{N}.
\end{aligned}
\end{align}

Let us first see what happens when we simply bound the above sum using a point-wise bound for the hyper-Kloosterman sum.  Here, we use the bound 
$$ \mathcal {KS}(\overline{\nu_2}f, \overline{\nu_3}g, \overline{\nu_1} ne; r) \ll_{\varepsilon} r^{1+\varepsilon} (f, r),$$
valid for any $\varepsilon > 0$. This was proven by Deligne for prime $r$, and extended to general $r$ by R.A. Smith (see~\cite[Theorem 6]{Smith}). The result of Smith is far more detailed; for instance, one may replace $(f, r)$ with $(ne, r)$ or $(g, r)$.  We split the sum on the right hand side of~\eqref{eqn:d>Dfnotzerognotzero} according to $\ell = (r, f)$. Notice that, for given $\ell \in \mathbb{N}$,
\begin{align*}
&\sumtwo_{\substack{e, f, g, n \\ efg \neq 0, \; \ell \mid f \\ (e, r) = (n, d) = 1}}  \mathcal {KS}(\overline{\nu_2}f, \overline{\nu_3}g, \overline{\nu_1} ne; r) \tau_3(\gamma n)   \widehat V\bfrac{Ee}{\nu_1 \gamma r} \widehat{V}\bfrac{fF}{\nu_2 r} \widehat V\bfrac{Gg}{\nu_3 r}  V\bfrac{\gamma n}{N}\\
&\ll Q^{4\varepsilon} r \ell \frac{\nu_1 \gamma r}{E}\frac{\nu_2 r}{\ell F}\frac{\nu_3 r}{G} \frac{N}{\gamma}  \ll Q^{4\varepsilon} \frac{\nu_1 \nu_2 \nu_3 r^4 N }{EFG\gamma}.
\end{align*}
From this, we see that
\begin{align}
\label{eq:unbaltriv}
\mathcal{S}(d\gamma\kappa >D, f\neq 0, g\neq 0) &\ll  Q^{5\varepsilon} \frac{\sqrt{N}}{\sqrt{EFG}} \sum_{\substack{d, \gamma, \kappa \\ (d, \gamma) = 1 \\ d\gamma\kappa >D}}\sum_{\ell} \sum_{\substack{r \\ \ell \mid r}} \Psi\bfrac{d\gamma\kappa r}{Q} r^2 \ll  Q^{6\varepsilon} \frac{\sqrt{N}}{\sqrt{EFG}} \cdot \frac{Q^3}{D^2}.
\end{align}

For $\delta' > 0$, this bound is $\ll Q^{2-\delta'/2+8\varepsilon}$ when  
\begin{align*}
D^2 \ge \frac{\sqrt{N}}{\sqrt{EFG}} Q^{1 +\delta'/2-2\varepsilon}.
\end{align*}
The right hand side is $< 1$ in the very unbalanced case when $M \asymp EFG \geq Q^{5/2 + \delta'}$, so the above computation already suffices and, adjusting $\varepsilon$, we obtain the following:

\begin{lem}\label{lem:EFGverylarge}
If $EFG \geq Q^{5/2 + \delta'}$, then 
\begin{align*}
\mathcal{S}(d\gamma\kappa >1/2, f\neq 0, g\neq 0) &\ll  Q^{2+\varepsilon - \delta'/2}.
\end{align*}
\end{lem}
We note that this lemma along with Lemmas \ref{lem:d>Df=0} and \ref{lem:d>Dg=0} proves the second part of Proposition \ref{prop:d>D}.

Towards the other extreme in the unbalanced sum case, when $EFG = Q^{2-\delta_0}$ and $N \ll Q^{1+\delta_0 +\varepsilon}$, we see that we need
\begin{align*}
D \ge Q^{\frac 14 + \frac{\delta_0}{2} + \frac{\delta'}{4}},
\end{align*}
which is far too large for our purposes.  In this range, we need to take better advantage of the average over $q$.  This is the content of the rest of this section. 
\begin{rem}
\label{rem:triviImpr} 
It might be possible to slightly improve on the error term in Theorem~\ref{thm:main} through using the bound~\eqref{eq:unbaltriv} in a slightly larger region. However, this would complicate the calculations and we have decided not to pursue this.
\end{rem}

The first part of Proposition \ref{prop:d>D} follows from Lemma \ref{lem:d>Df=0}, Lemma \ref{lem:d>Dg=0} and the following lemma.
\begin{lem} \label{lem:efgnotzerod>D} Let notations be as above and assume that $EFG \geq Q^{2-\delta_0}$. We have
$$ \mathcal{S}(d \gamma\kappa>D, f\neq 0, g\neq 0) \ll \frac{Q^{2+\varepsilon}}{D} + \frac{Q^{5/3 +2\delta_0/3 + \varepsilon}}{D^{1/2}}. $$
\end{lem}

\subsection{Setup to prove Lemma \ref{lem:efgnotzerod>D}}

By \eqref{eqn:d>Dfnotzerognotzero}
\begin{align}\label{eqn:d>Dfnotzerognotzero2}
\mathcal{S}(d\gamma\kappa > D, f\neq 0, g\neq 0) &\le \frac{\sqrt{EFG}}{\sqrt{N}} Q^{\varepsilon}\sum_{D<d \gamma \kappa \leq 2Q} \left|\mathcal{S}(d,\gamma, \kappa, f\neq 0, g\neq 0)\right|
\end{align}
where
\begin{align*}
&\mathcal{S}(d,\gamma, \kappa, f\neq 0, g\neq 0) := \bfrac{d \gamma \kappa}{Q}^2 \sum_{r\sim Q/(d \gamma \kappa)} \; \sum_{\substack{\nu_1 \mid d \\ (\nu_1, r) = 1}} \frac{1}{\nu_1} \sumtwo_{\substack{\nu_2, \nu_3 \mid d\gamma\kappa \\ (\nu_2 \nu_3, r) = 1}} \frac{1}{\nu_2}\frac{1}{\nu_3} \\
&\cdot \left|\sumfour_{\substack{e, f, g, n \\ efg \neq 0 \\ (e, r) = (n, d) = 1}}  \mathcal {KS}(\overline{\nu_2} f, \overline{\nu_3}g, \overline{\nu_1}n e; r) \tau_3(\gamma n)   \widehat V\bfrac{Ee}{\nu_1 \gamma r} \widehat{V}\bfrac{fF}{\nu_2 r} \widehat V\bfrac{Gg}{\nu_3 r}  V\bfrac{\gamma n}{N}\right|.
\end{align*}
By a change of variables, we have that
$$\mathcal {KS}(\overline{\nu_2}f, \overline{\nu_3}g, \overline{\nu_1} ne; r) 
= \sumtwostar_{a, b \Mod r} \ex\bfrac{\overline{ab \nu_2} f + b\overline{\nu_3} g + a\overline{\nu_1} ne}{r}
$$
so writing
\begin{align*}
S_1 := \sum_{\substack{e \neq 0 \\ (e, r) = 1}} \sum_{\substack{n \\ (n, d) = 1}} \ex\bfrac{a\overline{ \nu_1} ne}{r} \tau_3(\gamma n)   \widehat{V}\bfrac{eE}{\nu_1 \gamma r} V\bfrac{\gamma n}{N}
\end{align*}
and
\begin{align*}
S_2 := \sum_{f \neq 0} \sum_{g \neq 0} \; \sumstar_{b\Mod r} \ex\bfrac{b\overline{\nu_3}g + \overline{ab\nu_2} f}{r} \widehat V\bfrac{Ff}{\nu_2 r}  \widehat V\bfrac{Gg}{\nu_3 r}.
\end{align*}
We have
\begin{equation}
\label{eq:Sfgneq0bound}
\mathcal{S}(d,\gamma,\kappa, f\neq 0, g\neq 0) \ll \bfrac{d\gamma\kappa}{Q}^2 \sum_{\nu_1 \mid d} \frac{1}{\nu_1} \sumtwo_{\substack{\nu_2, \nu_3| d\gamma\kappa}} \frac{1}{\nu_2}\frac{1}{\nu_3} \sum_{\substack{r\sim Q/(d \gamma \kappa) \\ (r, \nu_1 \nu_2 \nu_3) = 1}}\; \; \sumstar_{a \Mod r} |S_1 S_2|.
\end{equation}
By the rapid decay of $\widehat{V}$, we see that $S_1 \ll Q^{-A}$ for any $A$ if $E \gg  Q^{1+\varepsilon} \frac{\nu_1}{d\kappa}$. Using also that $EFG \gg Q^{2-\delta_0}$ and $E\ge F\ge G$, we may assume that 
\begin{align*}
Q^{2/3 - \delta_0/3} \ll E \ll Q^{1+\varepsilon} \frac{\nu_1}{d \kappa} \ll Q^{1+\varepsilon},
\end{align*} 
since $\nu_1\le d$.  
We also see that
\begin{align*}
F \gg \sqrt{\frac{Q^{2-\delta_0}}{E}} \gg Q^{1/2 - \delta_0/2 - \varepsilon}.
\end{align*}
Furthermore, again apart from errors of size $\ll Q^{-A}$, we can assume that in $S_1$ and $S_2$, we have the restrictions $|e| \leq E', |f| \leq F'$, and $|g| \leq G'$ , where
\begin{equation}
\label{eqn:E'G'}
E' :=\frac{\nu_1 Q}{d\kappa E}Q^\varepsilon, \quad F' := \frac{\nu_2 Q}{d\gamma\kappa F}Q^\varepsilon, \quad G' := \frac{\nu_3 Q}{d\gamma\kappa G}Q^\varepsilon.
\end{equation} 

To complete the proof of Lemma~\ref{lem:efgnotzerod>D} we will apply the Cauchy-Schwarz inequality and the large sieve. In order to rigorously do this, we need to remove some extra dependencies on $r$ in $S_1$ and $S_2$. 
We now write $\widetilde{\widehat{V}}(s)$ as the Mellin transform of $\widehat{V}$.  To be precise, for $\tRe s >0$,
\begin{align*}
\widetilde{\widehat{V}}(s) = \int_0^\infty \widehat{V}(x) x^{s-1} \> dx.
\end{align*}
Due to the decay properties of $\widehat{V}$ and its derivatives in \eqref{eqn:Vnew}, we have by repeated integration by parts that
\begin{align*}
\widetilde{\widehat{V}}(s) \ll_k \frac{Q^\varepsilon}{|s|^k}.
\end{align*}
By Mellin inversion, we have, for $x>0$,
\begin{align*}
\widehat{V}(x) = \frac{1}{2\pi i} \int_{(c)} \widetilde{\widehat{V}}(s) x^{-s} ds,
\end{align*}for any $c>0$.  We will set 
\begin{equation}\label{eqn:c}
c = \frac{1}{\log Q}.
\end{equation}

We further remove the condition $(e, r) = 1$ using M\"obius inversion (introducing $\mu(\omega)$). Thus we have that
\begin{equation}\label{eqn:S_1}
S_1 = \frac{1}{2\pi i} \int_{(c)} \bfrac{\nu_1 \gamma r}{E}^s \widetilde{\widehat{V}}(s) \sum_{\omega | r} \mu(\omega)  S_1(s; \omega)  ds + O(Q^{-A}),
\end{equation}
for any $A > 0$, where
\begin{align*}
  S_1(s, \omega) =  \sum_{\substack{0 < |e| \leq \frac{E'}{\omega}}} \sum_{\substack{n\\ (n, d) = 1}} \ex\bfrac{a\overline{\nu_1} \omega en}{r} \tau_3(\gamma n)(e\omega)^{-s}V\bfrac{\gamma n}{N}.
\end{align*}

Similarly,
\begin{align}\label{eqn:S_2}
S_2 = \frac{1}{(2\pi i)^2} \int_{(c)} \int_{(c)}   S_2(s_1, s_2)  \bfrac{\nu_2 r}{F}^{s_1}\bfrac{\nu_3 r}{G}^{s_2} \widetilde{\widehat{V}}(s_1)\widetilde{\widehat{V}}(s_2) ds_1 ds_2 + O(Q^{-A}),
\end{align}
for any $A>0$, where
\begin{align*}
S_2(s_1, s_2) = \sumtwo_{\substack{0 < |f| \leq F' \\ 0 < |g| \leq G'}} \; \; \sumstar_{b\Mod r} \ex\bfrac{b\overline{\nu_3}g + \overline{ab\nu_2} f}{r}  f^{-s_1} g^{-s_2}.
\end{align*}
Most of the rest of this section is devoted to proving the following proposition.
\begin{prop}\label{prop:Vdss1s2bdd}
Let $s, s_1, s_2$ be complex numbers with real parts equalling $c = 1/\log Q$, and let
\begin{align*}
\Vcal(d, s, s_1, s_2) := \bfrac{d\gamma \kappa }{Q}^2 \sum_{\nu_1| d} \frac{1}{\nu_1} \sumtwo_{\nu_2, \nu_3|d \gamma \kappa} \frac{1}{\nu_2}\frac{1}{\nu_3} \sum_{\substack{r\sim Q/(d \gamma \kappa) \\ (r, \nu_1 \nu_2 \nu_3) = 1}} \;  \sumstar_{a \Mod r} \sum_{\omega | r } |S_1(s; \omega) S_2(s_1, s_2)|.
\end{align*}
We have that
\begin{align*}
\Vcal(d, s, s_1, s_2) \ll Q^{\varepsilon} \bfrac{Q}{d\gamma \kappa}^2 \sqrt{\frac{N}{EFG}} + Q^{\varepsilon} \bfrac{Q}{d\gamma \kappa}^{3/2} \frac{N}{E \sqrt{FG}}.
\end{align*}
\end{prop}

Before proceeding, let us prove Lemma \ref{lem:efgnotzerod>D} assuming Proposition \ref{prop:Vdss1s2bdd}.  Note that by the decay properties of $\widetilde{\widehat{V}}(s)$, we may truncate the integrals appearing in \eqref{eqn:S_1} and \eqref{eqn:S_2} to $|\tIm s|,|\tIm s_1|,|\tIm s_2|\leq Q^\varepsilon$ with error $O(Q^{-A})$ for any $A>0$.  Moreover, recalling $c=\frac{1}{\log Q}$, 
 $$\bfrac{\nu_1 \gamma r}{E}^{s},\bfrac{\nu_2 r}{F}^{s_1}, \bfrac{\nu_3 r}{G}^{s_2} \ll 1.$$  Thus by \eqref{eq:Sfgneq0bound}, \eqref{eqn:S_1}, and~\eqref{eqn:S_2}, and Proposition~\ref{prop:Vdss1s2bdd}
\begin{align*}
\mathcal{S}(d, \gamma, \kappa, f\neq 0, g\neq 0) \ll  Q^{2\varepsilon} \bfrac{Q}{d\gamma \kappa}^2 \sqrt{\frac{N}{EFG}} + Q^{2\varepsilon} \bfrac{Q}{d\gamma \kappa}^{3/2} \frac{N}{E \sqrt{FG}},
\end{align*}
and putting this into \eqref{eqn:d>Dfnotzerognotzero2}, we obtain
\begin{align*}
\mathcal{S}(d\gamma \kappa > D, f\neq 0, g\neq 0) 
&\ll  Q^{3\varepsilon} \frac{\sqrt{EFG}}{\sqrt{N}} \sum_{d\gamma \kappa > D} \left[  \bfrac{Q}{d\gamma \kappa}^2 \sqrt{\frac{N}{EFG}} +  \bfrac{Q}{d\gamma \kappa}^{3/2} \frac{N}{E \sqrt{FG}} \right]\\
&\ll \frac{Q^{2+3\varepsilon}}{D} + \frac{Q^{3/2+3\varepsilon}}{D^{1/2}} \sqrt{\frac{N}{E}}.
\end{align*}
Using the bounds $N\ll Q^{1+\delta_0+\varepsilon}$ and $E \gg Q^{2/3 - \delta_0/3}$, we get
\begin{align*}
\mathcal{S}(d\gamma\kappa > D, f\neq 0, g\neq 0) &\ll \frac{Q^{2+3\varepsilon}}{D} + \frac{Q^{3/2+3\varepsilon}}{D^{1/2}} Q^{1/6+2\delta_0/3 +\varepsilon/2},
\end{align*} 
which proves Lemma \ref{lem:efgnotzerod>D} after adjusting $\varepsilon$.

\subsection{Proof of Proposition \ref{prop:Vdss1s2bdd}}

By the Cauchy-Schwarz inequality,
\begin{align}\label{eqn:Vcalbdd1}
\Vcal(d, s, s_1, s_2) \ll  Q^\varepsilon \bfrac{d\gamma\kappa}{Q}^2\sum_{\nu_1 \mid d} \frac{1}{\nu_1} \sumtwo_{\substack{\nu_2, \nu_3| d\gamma\kappa}} \frac{1}{\nu_2}\frac{1}{\nu_3} \;  \sqrt{\mathcal{S}_1 \mathcal{S}_2},
\end{align}
for
\begin{align*}
\mathcal{S}_1 &:= \sum_{\substack{r\sim Q/(d\gamma\kappa) \\ (r, \nu_1) = 1}} \;  \sumstar_{a \Mod r} \sum_{\omega|r} |S_1(s; \omega)|^2 
\end{align*}and
\begin{align*}
\mathcal{S}_2 := \sum_{\substack{r\sim Q/(d\gamma\kappa) \\ (r, \nu_2 \nu_3)= 1}} \;  \sumstar_{a \Mod r} |S_2(s_1, s_2)|^2.
\end{align*}  
We now proceed to bound $\mathcal{S}_1$ and $\mathcal{S}_2$.

\begin{lem}\label{lem:T1} With the above notation,
\begin{equation*} 
\mathcal{S}_1 \ll Q^\varepsilon \frac{\nu_1 N}{E} \bfrac{Q}{d\gamma\kappa}^2 \left(\frac{Q}{d\gamma\kappa} + \frac{\nu_1 N}{ E}\right).
\end{equation*}
\end{lem}
\begin{proof}
Writing $r' = r/\omega$ and $b = \overline{\nu_1} a$, we have,
\begin{align*}
\mathcal{S}_1 &\ll Q^\varepsilon  \sum_{\omega \leq 2Q/(d\gamma\kappa)} \omega \sum_{\substack{r'\sim \frac{Q}{d\gamma \kappa\omega}}} \;  \sumstar_{b \Mod{r'}}  \left|\sum_{0 < |e| \leq E'/\omega} \sum_{\substack{n \\ (n, d) = 1}} \ex\bfrac{b en}{r'} \tau_3(\gamma n)(e\omega)^{-s}V\bfrac{\gamma n}{N}\right|^2.
\end{align*}
Applying the classical additive large sieve (see e.g.~\cite[Theorem 7.11]{IK} with a trigonometric polynomial of length $\ll E'N/(\gamma \omega)$ and coefficients $a_n \ll Q^\varepsilon$), we obtain
\begin{align*}
\mathcal{S}_1 &\ll  Q^{5\varepsilon} \sum_{\omega \leq 2Q/(d\gamma \kappa)} \omega \left(\bfrac{Q}{d\gamma \kappa \omega}^2 + \frac{\nu_1 Q N}{dE \gamma \kappa \omega}\right) \frac{\nu_1 Q N}{d\gamma\kappa \omega E} \ll Q^{5\varepsilon}  \left(\bfrac{Q}{d\gamma \kappa}^2 + \frac{\nu_1 Q N}{dE \gamma \kappa}\right) \frac{\nu_1 Q N}{d\gamma\kappa E}.
\end{align*}
Adjusting $\varepsilon$, this completes the proof. 
\end{proof}

\begin{lem}\label{lem:T2}
\begin{align*}
\mathcal{S}_2 \ll  Q^{\varepsilon}  \frac{\nu_2 \nu_3}{FG} \bfrac{Q}{d\gamma\kappa}^5 \left(1 + \frac{\nu_2 \nu_3}{FG}\right).
\end{align*}
\end{lem}
\begin{proof}
We need to do some preparations before being able to apply the large sieve. We first change the variable $a$ to $\overline{a}$, and write
\begin{align*}
\mathcal{S}_2 &= \sum_{\substack{r\sim Q/(d\gamma\kappa) \\ (r, \nu_2 \nu_3) = 1}} \;  \sumstar_{a \Mod r} \left|\sumtwo_{\substack{0 < |f| \leq F' \\ 0 < |g| \leq G'}} \;\; \sumstar_{b\Mod r} \ex\bfrac{b\overline{\nu_3}g + a\overline{b\nu_2} f}{r}  f^{-s_1} g^{-s_2} \right|^2\\
&\le  \sum_{\substack{r\sim Q/(d\gamma\kappa) \\ (r, \nu_2 \nu_3) = 1}} \;  \sum_{a \Mod r}  \left|\sumtwo_{\substack{0 < |f| \leq F' \\ 0 < |g| \leq G'}}\;\; \sumstar_{b\Mod r} \ex\bfrac{b\overline{\nu_3}g + a\overline{b\nu_2} f}{r}  f^{-s_1} g^{-s_2} \right|^2.
\end{align*}

Opening up the square, and using orthogonality in the complete sum over $a \Mod r$, we have 
\begin{align}\label{eqn:T21}
\mathcal{S}_2 
&\le \Biggl|\sum_{\substack{r\sim Q/(d\gamma\kappa) \\ (r, \nu_2 \nu_3) = 1}} r \sumfour_{\substack{f_1,f_2, g_1, g_2 \\ 0 < |f_j| \leq F'\\ 0 < |g_j| \leq G'}}\sumstar_{\substack{b_1, b_2\Mod r\\ b_2 f_1 \equiv b_1 f_2  \Mod r}} \ex\bfrac{b_1\overline{\nu_3}g_1 - b_2\overline{\nu_3}g_2}{r}  f_1^{-s_1} f_2^{-\overline{s_2}} g_1^{-s_2} g_2^{-\overline{s_2}}\Biggr|.
\end{align}
The congruence 
\begin{equation}\label{eqn:becongruence}
b_2 f_1 \equiv b_1 f_2  \Mod r
\end{equation}
implies that $(f_1, r) = (f_2, r) = \nu$, say.  

Write $r = \hat{r} u$, where $(u, r/\nu) = 1$ and $p \mid \hat{r} \implies p \mid r/\nu$. In particular $r/\nu \mid \hat{r}$, and we write
$$\hat{r} = \frac{r}{\nu} \Delta
$$for some $\Delta \in \mathbb{N}$.  Note that
\begin{align*}
&\nu = \Delta u \quad \textup{ and } \quad \textup{rad}(\hat{r}) \mid \frac{r}{\nu},
\end{align*}
where as usual, we write $\textup{rad}(\hat{r})$ to denote the largest square free divisor of $\hat{r}$.

Then, writing $f_1' = f_1/\nu$ and $f_2' = f_2/\nu$, we see that \eqref{eqn:becongruence} is equivalent to 
\begin{align*}
b_2 f_1' \equiv b_1 f_2' \Mod{r/\nu},
\end{align*}
as well as to 
\begin{align}
  \label{eq:b2b1cong}
b_2 \equiv b_1  \overline{f_1'} f_2' + l\frac{r}{\nu} \Mod{ \hat{r}},
\end{align}
holding for some $0\le l < \Delta = \frac{\hat{r}}{r/\nu}$.  Note moreover that each value of $l$ in the range $0\le l < \Delta$ gives a distinct reduced residue $b_2 \Mod{\hat{r}}$ since $\textup{rad}(\hat{r}) = \textup{rad}(r/\nu)$.

By the Chinese Remainder Theorem, we may write $b_i = x_i u \overline{u} + y_i \hat{r} \overline{\hat{r}}$, where $\overline{u}$ is inverse of $u$ $\Mod{\hat{r}}$ and $\overline{\hat{r}}$ is the inverse of $\hat{r}$ $\Mod{u}$.  Then~\eqref{eq:b2b1cong}, and thus also \eqref{eqn:becongruence}, is equivalent to
\begin{align*}
x_2 \equiv x_1 \overline{f_1'} f_2' + l\frac{r}{\nu} \Mod{ \hat{r}},
\end{align*}
for $0\le l < \Delta$.

Thus,
\begin{align}
  \label{eqn:g1divisor}
  \begin{aligned}
&\sumstar_{\substack{b_1, b_2\Mod r\\ b_2 f_1 \equiv b_1 f_2  \Mod r}}\ex\bfrac{b_1\overline{\nu_3}g_1 - b_2\overline{\nu_3}g_2}{r}\\
&= \sumstar_{x\Mod{\hat{r}}} \sum_{l \Mod \Delta} \ex\bfrac{\overline{u\nu_3}(xg_1 - x \overline{f_1'} f_2'  g_2 - l\frac{r}{\nu} g_2)}{\hat{r}} \sumstar_{y_1, y_2 \Mod u} \ex\bfrac{y_1 \overline{\nu_3\hat{r}} g_1 - y_2\overline{\nu_3\hat{r}} g_2}{u}\\
&= \Delta \mathbf{1}_{\Delta|g_2} \sumstar_{x\Mod{\hat{r}}} \ex\bfrac{xf_1'g_1 - xf_2' g_2}{\hat{r}} \sumstar_{y_1, y_2 \Mod u} \ex\bfrac{y_1 g_1 - y_2 g_2}{u}
  \end{aligned}
\end{align}
by an appropriate change of variables and where 
\begin{align*}
\mathbf{1}_{\Delta|g_2} = 
\begin{cases}
1 &\textup{ if } \Delta|g_2\\
0 &\textup{ otherwise.}
\end{cases}
\end{align*}  

The expression above may look unnaturally asymmetric with respect to the $g_i$, and now we rectify that situation. Recalling that $\textup{rad}(\hat{r}) = \textup{rad}(r/\nu)$, we may write reduced residues $x \Mod{\hat{r}}$ as $x = t + l\frac{r}{\nu}$ where $t$ runs through the reduced residues modulo $\frac{r}{\nu}$, and $l$ runs through all integers in the range $0\le l<\Delta$. Hence we have, assuming $\Delta \mid g_2$,
\begin{align}
  \label{eqn:g2divisor}
\begin{aligned}
\sumstar_{x\Mod{\hat{r}}} \ex\bfrac{xf_1'g_1 - xf_2' g_2}{\hat{r}} 
&=\sumstar_{x\Mod{\hat{r}}} \ex\bfrac{-xf_2' g_2/\Delta}{r/\nu} \ex\bfrac{xf_1'g_1}{\hat{r}}\\ 
&=\sumstar_{t\Mod{r/\nu}} \ex\bfrac{-tf_2' g_2/\Delta}{r/\nu} \ex\bfrac{tf_1'g_1}{\hat{r}} \sum_{l \Mod \Delta} \ex\bfrac{lf_1'g_1}{\Delta} \\
&=\Delta \mathbf{1}_{\Delta|g_1} \sumstar_{t\Mod{r/\nu}} \ex\bfrac{tf_1'g_1/\Delta-tf_2' g_2/\Delta}{r/\nu}.
\end{aligned}
\end{align}

Plugging \eqref{eqn:g2divisor} into \eqref{eqn:g1divisor} and then into \eqref{eqn:T21}, and for simplicity writing $f_i$ for $f_i' = f_i/\nu$, $g_i$ for $g_i/\Delta$ we see that
\begin{align}
  \label{eqn:T22}
  \begin{aligned}
\mathcal{S}_2 
&\leq \sum_{r\sim Q/(d\gamma\kappa) } r \sum_{\nu|r}\sum_{\substack{\nu = u\Delta \\ (u, r/\nu) = 1 \\ p \mid \Delta \implies p \mid r/\nu}} \Delta^{2} \sumstar_{t\Mod{r/\nu}}\\
&\Biggl|\sumfour_{\substack{f_1,f_2, g_1, g_2 \\ 0 < |f_j| \leq F'/\nu\\ 0 < |g_j| \leq G'/\Delta\\ (f_j, r/\nu) = 1}}\ex\bfrac{tf_1g_1-tf_2 g_2}{r/\nu} \sumstar_{y_1, y_2 \Mod u} \ex\bfrac{y_1 g_1 - y_2 g_2}{u}  f_1^{-s_1} f_2^{- \overline{s_1}} g_1^{-s_2} g_2^{-\overline{s_2}}\Biggr|.
  \end{aligned}
\end{align}

Unfortunately, we now have the conditions $(f_j, r/\nu) = 1$ for $j = 1, 2$ which impedes a rigorous application of the large sieve.  We pick out these condition by M\"obius inversion (introducing $\mu(\lambda_i)$).
For $\lambda, t, r' \in \mathbb{N}$, let
\begin{align*}
\calS(\lambda, t, r') := \sumtwo_{\substack{f, g\\ 0 < |f| \leq F'/(\lambda \nu)\\ 0 < |g| \leq G'/\Delta}} \ex\bfrac{t  f g}{r'} \sumstar_{y\Mod u} \ex\bfrac{yg}{u} (\lambda f)^{-s_1} g^{-s_2}.
\end{align*}

Then by \eqref{eqn:T22}, we have
\[
\mathcal{S}_2 \le \sum_{r\sim Q/(d\gamma\kappa)} r \sum_{\nu|r} \sum_{\substack{\nu = u\Delta \\ (u, r/\nu) = 1 \\ p \mid \Delta \implies p \mid r/\nu}} \Delta^{2} \sumstar_{t\Mod{r/\nu}}  \sumtwo_{\lambda_1, \lambda_2|\frac{r}{\nu}} \left|S\left(\lambda_1, t, \frac{r}{\nu \lambda_1}\right) \overline{S\left(\lambda_2, t, \frac{r}{\nu \lambda_2}\right)}\right|.
\]
Note that that $S(\lambda, t, r')$ depends on $t$ modulo $r'$. Hence we can simplify the notation by writing $r$ for $\frac{r}{\lambda \nu}$. Applying the inequality $|xy| \leq |x|^2+|y|^2$ and the bound $\sum_{\lambda_i|r} 1 \ll Q^\varepsilon$, we obtain
\begin{equation}
\label{eqn:T23}
\mathcal{S}_2 \ll Q^\varepsilon \frac{Q}{d\gamma \kappa} \sum_{\nu \leq 2Q} \sum_{\substack{\nu = u\Delta \\ (u, r/\nu) = 1 \\ p \mid \Delta \implies p \mid r/\nu}} \Delta^2 \sum_{\lambda \leq 2Q} \lambda \sum_{\substack{r \sim \frac{Q}{d \gamma \kappa \lambda \nu}}}  \; \sumstar_{t \Mod r} |S(\lambda, t, r)|^2.
\end{equation}

We now apply the large sieve (see e.g.~\cite[Theorem 7.11]{IK}) to get 
\begin{align}\label{eqn:T2subsum1}
\sum_{\lambda \leq 2Q} \lambda \sum_{\substack{r \sim \frac{Q}{d \gamma \kappa \lambda \nu}}}  \; \sumstar_{t \Mod r} |S(\lambda, t, r)|^2
&\ll \sum_{\lambda\leq 2Q} \lambda \left(\bfrac{Q}{d \gamma \kappa \lambda \nu}^2 + \frac{F'G'}{\lambda \nu \Delta}\right) \sum_{j} |a_\lambda(j)|^2,
\end{align}
where
\begin{align*}
a_\lambda(j) := \sum_{\substack{j = fg\\ 0 < |f| \leq F'/(\lambda \nu)\\ 0 < |g| \leq G'/\Delta}} \sumstar_{y\Mod u} \ex\bfrac{yg}{u} (\lambda f)^{-s_1} g^{-s_2}.
\end{align*}
Here we have a Ramanujan sum 
$$\sumstar_{y\Mod u} \ex\bfrac{yg}{u}\ll (g, u),
$$
and furthermore
$$|(\lambda f)^{-s_1} g^{-s_2}| \ll 1,
$$
upon recalling \eqref{eqn:c}.  Thus, writing $\delta_j = (g_j, u)$,

\begin{align}\label{eqn:aL2}
\sum_{j} |a_\lambda(j)|^2 \ll \sum_{\delta_1, \delta_2 \mid u} \delta_1 \delta_2 \sumtwo_{\substack{f_1, g_1 \\ 0 < |f_1| \leq F'/(\lambda \nu)\\ 0 < |g_1| \leq G'/\Delta\\ \delta_1 \mid g}} \sum_{\substack{f_2, g_2\\f_2g_2 = f_1g_1 \\ \delta_2 \mid g_2}} 1
\ll  Q^\varepsilon \left(\sum_{\substack{[\delta_1, \delta_2] \mid u}}  \frac{\delta_1 \delta_2}{[\delta_1, \delta_2]}\right) \frac{F'G'}{\lambda \nu\Delta}
\ll  Q^{2\varepsilon} \frac{F'G' u}{\lambda \nu\Delta}.
\end{align}
By \eqref{eqn:T2subsum1} and \eqref{eqn:aL2}, we see that
\begin{align*}
\sum_{\lambda \leq 2Q} \lambda \sum_{\substack{r \sim \frac{Q}{d \gamma \kappa \lambda \nu}}}  \; \sumstar_{t \Mod r} |S(\lambda, t, r)|^2
\ll Q^{3\varepsilon} \frac{F'G' u}{ \nu\Delta} \left(\bfrac{Q}{d \gamma \kappa \nu}^2 + \frac{F'G'}{\nu \Delta}\right).
\end{align*}
Putting this into \eqref{eqn:T23}, we have
\begin{align*}
\mathcal{S}_2 &\ll Q^{4\varepsilon} \frac{Q}{d\gamma \kappa} \sum_{\nu \leq 2Q} \sum_{\substack{\nu = u\Delta \\ (u, r/\nu) = 1 \\ p \mid \Delta \implies p \mid r/\nu}} \Delta^2 \frac{F'G' u}{ \nu\Delta} \left(\bfrac{Q}{d \gamma \kappa \nu}^2 + \frac{F'G'}{\nu \Delta}\right)\\
&\ll  \frac{Q^{1+5\varepsilon}}{d\gamma \kappa} F'G' \left(\bfrac{Q}{d\gamma \kappa}^2 + F'G'\right).
\end{align*}
By~\eqref{eqn:E'G'}, we obtain
\[
\mathcal{S}_2 \ll  \frac{Q^{3+9\varepsilon}}{(d\gamma \kappa)^3} \frac{\nu_2 \nu_3}{FG} \left(\bfrac{Q}{d\gamma \kappa}^2 + \frac{\nu_2 \nu_3 Q^2}{(d\gamma \kappa)^2 FG}\right),
\]
which implies the claim after adjusting $\varepsilon$.
\end{proof}

It is now straightforward to complete the proof of Proposition \ref{prop:Vdss1s2bdd} using Lemmas \ref{lem:T1} and \ref{lem:T2}.  Indeed, by Lemmas \ref{lem:T1} and \ref{lem:T2} and \eqref{eqn:Vcalbdd1}
\begin{align*}
& \Vcal(d, s, s_1, s_2) \\
&\ll  Q^{2\varepsilon} \bfrac{d\gamma \kappa}{Q}^2\sum_{\nu_1 | d} \frac{1}{\nu_1} \sumtwo_{\nu_2, \nu_3|d\gamma \kappa} \frac{1}{\nu_2}\frac{1}{\nu_3} \; \sqrt{\frac{\nu_1 \nu_2 \nu_3 N}{EFG}} \bfrac{Q}{d\gamma \kappa}^{7/2} \left( \bfrac{Q}{d\gamma \kappa}^{1/2} + \sqrt{\frac{\nu_1 N}{ E}}\right)\left(1+ \sqrt{\frac{\nu_2\nu_3}{FG}} \right) \\
&\ll  Q^{3\varepsilon} \bfrac{Q}{d\gamma \kappa}^{3/2} \sqrt{\frac{N}{EFG}} \left( \bfrac{Q}{d\gamma \kappa}^{1/2} + \sqrt{\frac{N}{E}} \right),
\end{align*}
and the claim follows after adjusting $\varepsilon$.

\section{Averages of Kloosterman sums}
\label{sec:AverageKloo}
Write, for sequences $\mathbf{a} = (a_m)_{m \geq 1}$ and $\mathbf{b} = (b_{n, r, s})_{n,r,s \geq 1}$,
\[
\Vert \mathbf{a} \Vert_2  = \sqrt{\sum_m |a_m|^2} \quad \text{and} \quad \Vert \mathbf{b} \Vert_2 = \sqrt{\sum_{n, r, s} |b_{n, r, s}|^2}.
\]
To deal with the averages of Kloosterman sums appearing in the case when $d\lambda\kappa$ is small we shall use the following refinement of~\cite[Theorem 10]{DI}. 

\begin{lemma}
\label{le:Klo1}
Let $C, M, N, R, S \geq 1/2$ and let $g \colon \mathbb{R}^5 \to \mathbb{R}$ be a smooth function with compact support on $[C, 2C] \times (0, \infty)^4$ such that, for any $\varepsilon > 0$
\begin{align*}
\left|\frac{\partial^{\nu_1 + \nu_2 + \nu_3 + \nu_4 + \nu_5}}{\partial c^{\nu_1} \partial m^{\nu_2} \partial n^{\nu_3} \partial r^{\nu_4} \partial s^{\nu_5}} g(c, m, n, r, s)\right| \ll_{\nu_j} (CMNRS)^\varepsilon c^{-\nu_1} m^{-\nu_2} n^{-\nu_3} r^{-\nu_4} s^{-\nu_5},
\end{align*}
for every $\nu_j \geq 0$,  $1 \leq j \leq 5$.
Assume that
\begin{equation}
\label{def:X}
X := \frac{CS\sqrt{R}}{4\pi \sqrt{MN}} \gg 1.
\end{equation}
Let $\mathbf{a} = (a_m)_{m \geq 1}$ and $\mathbf{b} = (b_{n,r,s})_{n,r,s \geq 1}$ denote two sequences. 
Let
\[
\mathcal{L}^{\pm}(C, M, N, R, S) = \sumfour_{\substack{r \sim R, s \sim S, m \sim M, n \sim N \\ (r, s) = 1}}  a_{m} b_{n, r, s} \sum_{\substack{c \\ (c, r) = 1}} g(c, m, n, r, s) S(\pm n, m\overline{r}, sc).
\]
Then, for any $\varepsilon > 0$,
\[
\mathcal{L}^{\pm}(C, M, N, R, S) \ll (CMNRS)^\varepsilon L(C, M, N, R, S) \Vert \mathbf{a} \Vert_2 \Vert \mathbf{b} \Vert_2,
\]
where
\[
L(C, M, N, R, S) = CS\sqrt{R}\cdot \sqrt{RS} \left(1+\sqrt{\frac{M}{RS}}\right) \left(1+\sqrt{\frac{N}{RS}}\right)\left(1+\frac{X^2}{\left(1+\frac{RS}{M}\right)^2 \left(1+\frac{RS}{N}\right)}\right)^{7/64}.
\]
\end{lemma}

The proof is very similar to that of~\cite[Theorem 10]{DI}, with the main new input being the more recent progress in \cite{KS} towards bounds on exceptional eigenvalues.  We incorporate these bounds into our proof in much the same manner as (8.18) in \cite{DI} where an older $L^\infty$ bound on exceptional eigenvalues is incorporated.  A similar computation has appeared in other recent works, such as the work \cite{DPR}.

Here, we only point out differences, and freely borrow notation and definitions from~\cite{DI}. After using Kuznetsov formula, Deshouillers and Iwaniec use two large sieve type results for cusp form coefficients.

We will denote by $\rho^{(q)}_{j, \mathfrak a}(n)$ the $n$th coefficient of a Maa{\ss} form of level $q$ and eigenvalue $\lambda_j = \frac 14 + \kappa_j^2$, expanded around the cusp $\mathfrak{a}$ of $\Gamma_0(q)$ and with the coefficients normalized so that the Maa{\ss} form has $L^2$ norm equal to one.

The first large sieve type result is a special case of \cite[Theorem 5]{DI}, specialized to the case when the level is $r s$ with $(r,s) = 1$, and the Fourier coefficients of the Maa{\ss} forms are expanded around the cusp $\mathfrak a = 1/s$, so that in the notation of \cite[Section 1.1]{DI} we then have $\mu(\mathfrak{a}) = 1 / (rs)$. 

\begin{lem} \label{LSavgoverexceigen}
Let $r, s$ be non-negative integers, where $(r, s) = 1.$ Let $\mathbf{a} = (a_n)_{n \geq 1}$ denote a sequence of complex numbers. For any $X \geq 1$ and $\varepsilon > 0,$ we have
\begin{align*}
	\sum_{\lambda_j - \textrm{except}}^{(rs)} X^{2i\kappa_j} \left| \sum_{n \sim N} a_n \rho_{j, 1/s}^{(rs)}(n)\right|^2 \ll \left( 1 + \sqrt{\frac{NX}{rs}}\right) \left(1 + \sqrt{\frac{N^{1 + \varepsilon}}{rs}} \right) \| \mathbf{a} \|^2_2,
\end{align*}
where the implied constant depends only on $\varepsilon$. Here $\sum^{(rs)}$ denotes a sum over exceptional eigenvalues $\lambda_j$'s of the Hecke group $\Gamma_0(rs)$. 
\end{lem}

The second large sieve type result is \cite[Theorem 6]{DI}.
\begin{lem} \label{LAavgoverexceigenandlevel}
Let $X, Q, N, \varepsilon$ be positive numbers and $\mathbf a = (a_{n})_{n \geq 1}$ denote a sequence of complex numbers. Then we have
\begin{align*}
\sum_{q \leq Q} \sum_{\lambda_j - \textrm{except}}^{(q)} X^{4i\kappa_j} \left| \sum_{n \sim N} a_n \rho_{j, \infty}^{(q)}(n) \right|^2 \ll (QN)^{\varepsilon} \left( Q + N + NX\right)\| \mathbf{a}\|^2_2,
\end{align*}
where the implied constant depends only on $\varepsilon$. Here $\sum^{(q)}$ denotes a sum over exceptional eigenvalues $\lambda_j$'s of the Hecke group $\Gamma_0(q)$.
\end{lem}

We can refine the above two lemmas using the bound for exceptional eigenvalues $\kappa_j$ from Kim and Sarnak's work \cite{KimS}, which states that when $\lambda_j = 1/4 + \kappa_j^2$ is exceptional, then $0 < |i\kappa_j| \leq \frac{7}{64}.$  The choice of sign of $i \kappa_j$ is irrelevant, and for convenience, we will assume 
\begin{equation} \label{bound:Kims}
0 < i\kappa_j \leq \frac{7}{64}.
\end{equation}

\begin{lem} \label{lem:bddYside}
Let $Q, Y \geq 1$ and $\mathbf{a} = (a_m)_{m \geq 1}$ denote a sequence of complex numbers. Then, for every $\varepsilon > 0$, 
\begin{align*}
\sum_{q \sim Q} \sum_{\lambda_j -\textrm{except}}^{(q)} Y^{2i\kappa_j}\left|   \sum_{m \sim M} a_m \overline{\rho_{j, \infty}^{(q)} (m)}  \right|^2 \ll_{\varepsilon} ( Q M)^{\varepsilon} \cdot (Y_1^{7/32}  + 1) (Q + M) \|\mathbf{a} \|_2^2,
\end{align*}
where 
$$Y_1 = \frac{Y}{1+ \bfrac{Q}{M}^2}.$$
\end{lem}
\begin{proof}
By~\eqref{bound:Kims} and Lemma \ref{LAavgoverexceigenandlevel}
\begin{align*}
&\sum_{q \sim Q} \sum_{\lambda_j -\textrm{except}}^{(q)} Y^{2i\kappa_j}\left|   \sum_m a_m \overline{\rho_{j, \infty}^{(q)} (n)}  \right|^2 \\
& \ll (Y_1^{7/32}+1) \sum_{q \sim Q} \sum_{\lambda_j -\textrm{except}}^{(q)} (Y/Y_1)^{2i\kappa_j}\left|   \sum_m a_m \overline{\rho_{j, \infty}^{(q)} (n)}  \right|^2  \\ 
&\ll_{\varepsilon} (QM)^{\varepsilon} \cdot \Big ( Y_1^{7/32} + 1 \Big ) \left( Q + M + M \sqrt{Y/Y_1} \right) \|\mathbf{a} \|_2^2 \\ & \ll (Q M)^{\varepsilon} \cdot \Big ( Y_1^{7/32} + 1 \Big) (Q + M) \|\mathbf{a} \|_2^2,
\end{align*}
as needed. 
\end{proof}

\begin{lem}
  \label{lem:nsumnoFE}
Let $R, S, Z \geq 1$ and let $\mathbf{b} = (b_{n,r,s})_{n,r,s \geq 1}$ denote a sequence of complex numbers. Then, for every $\varepsilon > 0$, 
\begin{align*}
&\sumtwo_{\substack{r \sim R, s \sim S \\ (r,s) =1}} \sum_{\lambda_j -\textrm{except}}^{(rs)} Z^{2i\kappa_j} \left| \sum_{n \sim N} \ b_{n,r,s} \ \rho_{j,  1/s}^{(rs)} (n) \right|^2 \ll_{\varepsilon} (Z_1^{7/32}+1) \left( 1 + \frac{N^{1 + \varepsilon}}{RS} \right)\| \mathbf{b} \|_2^2,
\end{align*}
where
\[
Z_1 = \frac{Z}{1+\frac{RS}{N}}.
\]
\end{lem}

\begin{proof} Applying Lemma \ref{LSavgoverexceigen},
\begin{align*}
&\sumtwo_{\substack{r \sim R, s \sim S \\ (r,s) =1}} \sum_{\lambda_j -\textrm{except}}^{(rs)} Z^{2i\kappa_j} \left| \sum_{n \sim N} \ b_{n,r,s} \ \rho_{j,  1/s}^{(rs)} (n) \right|^2 \\
&\ll (Z_1^{7/32}+1) \sumtwo_{\substack{r \sim R, s \sim S \\ (r,s) =1}} \sum_{\lambda_j -\textrm{except}}^{(rs)} (Z/Z_1)^{2i\kappa_j} \left| \sum_{n \sim N} \ b_{n,r,s} \ \rho_{j,  1/s}^{(rs)} (n) \right|^2 \\
&\ll (Z_1^{7/32} + 1) \cdot \sum_{\substack{r \sim R, s \sim S \\ (r,s) = 1}} \left( 1 + \sqrt{\frac{N}{RS} \cdot \frac{Z}{Z_1}} \right) \left( 1 + \sqrt{\frac{N^{1 + \varepsilon}}{RS}} \right) \cdot \Big ( \sum_{n} |b_{n,r,s}|^2 \Big ) \\ & \ll (Z_1^{7/32} + 1) \cdot \left( 1 + \frac{N^{1 + \varepsilon}}{RS} \right)\| \mathbf{b} \|_2^2.
\end{align*}
\end{proof}

\begin{proof}[Proof of Lemma~\ref{le:Klo1}]
Similarly to~\cite[Proof of Theorem 10]{DI} we restrict to the case where
\[
g(c,m,n,r,s) = \frac{CS \sqrt{R}}{cs\sqrt{r}} f\left(\frac{4\pi \sqrt{mn}}{cs\sqrt{r}}\right)
\]
with $f$ a smooth function supported on $[X^{-1}, 2X^{-1}]$, where $X$ is defined as in \eqref{def:X}, and with $|f^{(\nu)}(x)| \ll x^{-\nu}$ for every $\nu \geq 0$. 
By \cite[formula (9.2)]{DI} we have
\[
\mathcal{L}^{\pm}(C, M, N, R, S) = CS\sqrt{R} \sumfour_{\substack{r \sim R, s \sim S, m \sim M, n \sim N \\ (r, s) = 1}}  a_{m} b_{n, r, s} \ex\left( -n \frac{\overline{s}}{r} \right) \sum^\Gamma \frac{1}{\gamma} f\left(\frac{4\pi \sqrt{mn}}{\gamma}\right) S_{\infty \frac{1}{s}}(m, \pm n, \gamma).
\]
Now we apply Kuznetsov's formula from Theorem \cite[Theorem 1]{DI}. All the terms except the contribution from the exceptional spectrum are treated in the same way as in \cite[Section 9]{DI}. We obtain from~\cite[formula (9.4)]{DI}, using $X \gg 1$, that the contribution of holomorphic, continuous and regular spectrum to $\mathcal{L}^{\pm}(C, M, N, R, S)$ is bounded by
\begin{align*}
\ll (CMNRS)^{\varepsilon} \sqrt{RS} \frac{(CS\sqrt{R} + C\sqrt{SM})(CS\sqrt{R}+C\sqrt{SN})}{CS\sqrt{R}} \|\mathbf{a} \|_2 \| \mathbf{b} \|_2.
\end{align*}
Rearranging, we see that this is acceptable.
Furthermore, writing $b'_{n,r,s} = b_{n,r,s} e(-n\frac{\overline{s}}{r})$, the contribution $\mathcal{L}_{exc}(C, M, N, R, S)$ of the exceptional spectrum is 
\begin{align*}
CS\sqrt{R} \sumtwo_{\substack{r \sim R, s \sim S \\ (r, s) = 1}} \sum_{\lambda_j -\textrm{except}}^{(rs)} \frac{\widehat f (\kappa_j)}{\cosh (\pi \kappa_j)} \left( \sum_m a_m \overline{\rho_{j, \infty}^{(rs)} (m)} \right) \left( \sum_{n}  b'_{n, r, s} \rho_{j,  1/s}^{(rs)}(n) \right),
\end{align*}
where (and for the rest of the proof of this Lemma) $\widehat{f}$ denotes the Bessel-Kuznetsov transform defined in \cite[Equation (1.22)]{DI} and not the usual Fourier transform. By Lemma 1 of \cite{BHM} (see also the proof of \cite[Equation (7.1)]{DI}), for $-\frac14 < r < \frac 14$,  

\begin{equation} \label{ourbound} 
\widehat{f}(i r) \ll X^{2 |r|}.\end{equation}

In applying this bound, note that our $f$ is supported in $[X^{-1}, 2 X^{-1}]$ while the $f$ in Lemma 1 of \cite{BHM} is supported in $x \asymp X$. Recalling that we have picked $i \kappa_j > 0$, so $\hat{f}(\kappa_j) \ll X^{2 i \kappa_j}$, and thus

\begin{align} \label{eqn:CSforU4excbigR2}
  \begin{aligned}
&\mathcal{L}_{exc}(C, M, N, R, S) \\
&\ll CS\sqrt{R} \left(\sumtwo_{\substack{r \sim R, s \sim S \\ (r,s) = 1}} \sum_{\lambda_j -\textrm{except}}^{(rs)} Y^{2i\kappa_j}\left|   \sum_m a_m \overline{\rho_{j, \infty}^{(rs)} (n)}  \right|^2 \right)^{1/2} \left( \sumtwo_{\substack{r \sim R, s \sim S \\ (r,s) = 1}}\sum_{\lambda_j -\textrm{except}}^{(rs)} Z^{2i\kappa_j}\left|\sum_{n} b'_{n,r,s} \ \rho_{j, 1/s}^{(rs)} (n)\right|^2\right)^{1/2}
  \end{aligned}
\end{align}
for any $Y, Z$ such that $YZ = X^2$, where $Y, Z \ge 1$ are parameters to be determined.  

Now Lemma~\ref{le:Klo1} follows by combining Lemmas~\ref{lem:bddYside} and~\ref{lem:nsumnoFE} with~\eqref{eqn:CSforU4excbigR2}, choosing e.g. $Z = 1+RS/N$.
\end{proof}

\section{The case $d\gamma \kappa$ is small}
\label{se:dgamsmall}
\subsection{Initial reductions}
We aim to prove Proposition~\ref{prop:d<D}. Looking back to \eqref{masterEq}, replacing $S(\overline{\nu_2} f, n e\overline{\nu_1 g}; r)$ by $S(e f, n\overline{\nu_1 \nu_2 g}; r)$ (which can be done since $(e\nu_2, r) = 1$), we find that
\begin{align*}
\calS(d\gamma \kappa \leq D) &= \frac{\sqrt{EF}}{\sqrt{GN}} \sum_{\substack{d, \gamma, \kappa \\ (d, \gamma) = 1 \\ d\gamma \kappa \leq D}} \mu(d) \mu (\gamma) \sum_{r} \frac{\phi(\gamma \kappa r)}{\gamma \kappa r^2} \Psi\bfrac{d\gamma \kappa r}{Q}  \sum_{\substack{\nu_1 \mid d \\ (\nu_1, r\kappa) = 1}} \frac{\mu(\nu_1)}{\nu_1} \sum_{\substack{\nu_2 \mid d\gamma \kappa \\ (\nu_2, r) = 1}}  \frac{\mu(\nu_2)}{\nu_2} \\
&\cdot \sumfour_{\substack{e \neq 0, \; f \neq 0, \; g, \; n \\ (e, r) = (n, d) =  1 \\ (g, d\gamma \kappa r) = 1}}  S(ef, n\overline{\nu_1\nu_2 g}; r) \tau_3(\gamma n)  \widehat{V}\bfrac{fF}{\nu_2 r} \widehat V\bfrac{Ee}{\nu_1 \gamma r} V\bfrac{g}{G} V\bfrac{\gamma n}{N}.
\end{align*}
We write
\begin{align*}
\frac{\phi(\gamma \kappa r)}{\gamma \kappa r} = \sum_{a \mid \gamma \kappa r} \frac{\mu(a)}{a} = \sumthree_{\substack{\gamma_1 \mid \gamma, \; \kappa_1 \mid \kappa, \; r_1 \mid r \\ (\kappa_1 r_1, \gamma/\gamma_1) = (r_1, \kappa/\kappa_1) = 1}} \frac{\mu(\gamma_1 \kappa_1 r_1)}{\gamma_1 \kappa_1 r_1} = \sumthree_{\substack{\gamma = \gamma_1 \gamma_2, \; \kappa = \kappa_1 \kappa_2, \; r = r_1 r_2 \\ (\kappa_1 r_1, \gamma_2) = (r_1, \kappa_2) = 1}} \frac{\mu(\gamma_1 \kappa_1 r_1)}{\gamma_1 \kappa_1 r_1}.
\end{align*}
Writing also $\mu(\gamma_1 \gamma_2) = \mathbf{1}_{(\gamma_1, \gamma_2) = 1} \mu(\gamma_1) \mu(\gamma_2)$, we see that
\begin{align*}
  \mathcal{S}(d\gamma\kappa \leq D) &= \frac{\sqrt{EF}}{\sqrt{GN}} \sumthree_{\substack{d, \gamma_1, \gamma_2, \kappa_1, \kappa_2 \\ (d, \gamma_1 \gamma_2) = (\gamma_1 \kappa_1, \gamma_2) = 1 \\ d\gamma_1 \gamma_2 \kappa_1 \kappa_2 \leq D}} \mu(d) \mu(\gamma_1) \mu(\gamma_2) \sum_{\substack{r_1 \\ (r_1, \gamma_2 \kappa_2) = 1}} \frac{\mu(\kappa_1 \gamma_1 r_1)}{\kappa_1 \gamma_1 r_1^2}  \sum_{\substack{\nu_1 \mid d \\ (\nu_1, \kappa_1 \kappa_2 r_1) = 1}} \frac{\mu(\nu_1)}{\nu_1}\\
  & \cdot \sum_{\substack{\nu_2 \mid d\gamma_1 \gamma_2 \kappa_1 \kappa_2 \\ (\nu_2, r_1) = 1}} \frac{\mu(\nu_2)}{\nu_2} \sumfour_{\substack{e \neq 0, \; f \neq 0, \; g, \; n \\ (e, r_1) = (n, d) = 1 \\ (g, d\gamma_1 \gamma_2 \kappa_1 \kappa_2 r_1) = 1}}  \sum_{\substack{r_2 \\ (r_2, e \nu_1 \nu_2 g) = 1}} \frac{1}{r_2}    S(e f, n\overline{\nu_1 \nu_2 g}; r_1 r_2) \\
&\cdot \widehat V\bfrac{Ee}{\nu_1 \gamma_1 \gamma_2 r_1 r_2} \Psi\bfrac{d\gamma_1 \gamma_2 \kappa_1 \kappa_2 r_1 r_2}{Q} \tau_3(\gamma_1 \gamma_2 n)  \widehat{V}\bfrac{fF}{\nu_2 r_1 r_2} V\bfrac{g}{G} V\bfrac{\gamma_1 \gamma_2 n}{N}.
\end{align*}

Next we remove the condition $(r_2, e) = 1$ using M\"obius inversion (inroducing $\mu(\omega)$). We get
\begin{align*}
\mathcal{S}(d\gamma\kappa \leq D) &= \frac{\sqrt{EF}}{\sqrt{GN}} \sumthree_{\substack{d, \gamma_1, \gamma_2, \kappa_1, \kappa_2 \\ (d, \gamma_1 \gamma_2) = (\gamma_1 \kappa_1, \gamma_2) = 1 \\ d\gamma_1 \gamma_2 \kappa_1 \kappa_2 \leq D}} \mu(d) \mu(\gamma_1) \mu(\gamma_2) \sum_{(r_1, \gamma_2 \kappa_2) = 1} \frac{\mu(\kappa_1 \gamma_1 r_1)}{\kappa_1 \gamma_1 r_1^2} \sum_{\substack{\nu_1 \mid d \\ (\nu_1, \kappa_1 \kappa_2 r_1) = 1}} \frac{\mu(\nu_1)}{\nu_1}   \\
&\cdot \sum_{\substack{\nu_2 \mid d\gamma_1 \gamma_2 \kappa_1 \kappa_2 \\ (\nu_2, r_1) = 1}} \frac{\mu(\nu_2)}{\nu_2} \sum_{\substack{\omega \\ (\omega, r_1g \nu_1 \nu_2) = 1}} \mu(\omega) \sumfour_{\substack{e \neq 0, \; f \neq 0, \; g, \; n \\ (e, r_1) = (n, d) = 1 \\ (g, d\gamma_1 \gamma_2 \kappa_1 \kappa_2 r_1) = 1}} \sum_{\substack{r_2 \\ (r_2, \nu_1 \nu_2 g) = 1}} \frac{1}{\omega r_2}    S(\omega e f, n\overline{\nu_1 \nu_2 g}; r_1 \omega r_2) \\
&\cdot \widehat V\bfrac{Ee}{\nu_1 \gamma_1 \gamma_2 r_1 r_2} \Psi\bfrac{d\gamma_1 \gamma_2 \kappa_1 \kappa_2 \omega r_1 r_2}{Q} \tau_3(\gamma_1 \gamma_2 n)  \widehat{V}\bfrac{fF}{\nu_2 r_1 r_2 \omega} V\bfrac{g}{G} V\bfrac{\gamma_1 \gamma_2 n}{N}.
\end{align*}
Notice that with an error $O(Q^{-10})$ we can, using decay properties of $\widehat{V}$ and support of $\Psi$, restrict to
\begin{align*}
\omega ef &\ll \omega \cdot \frac{\nu_1 \gamma_1 \gamma_2 r_1 r_2}{E} \cdot \frac{\nu_2 r_1 r_2 \omega}{F} Q^\varepsilon = (d \gamma_1 \gamma_2 \kappa_1 \kappa_2 \omega r_1 r_2)^2 \cdot \frac{\nu_1 \nu_2}{EF d^2 \gamma_1 \gamma_2 (\kappa_1 \kappa_2)^2} Q^\varepsilon \\
&\ll Q^2 \frac{\nu_1 \nu_2}{EF d^2 \gamma_1 \gamma_2 (\kappa_1 \kappa_2)^2} Q^\varepsilon
\end{align*}
and
\begin{equation}
\label{eq:omegabound}
\omega \leq \frac{2Q}{d \gamma_1 \gamma_2 \kappa_1 \kappa_2 r_1 r_2} \ll \frac{Q^{1+\varepsilon} \nu_1}{|e| E \kappa_1 \kappa_2 d} \ll \frac{Q^{1+\varepsilon}}{E}.
\end{equation}

We split variables dyadically, so that $\kappa_j \sim \mathcal{K}_j$, $\omega \sim \Omega$, $\nu_j \sim \mathcal{V}_j$, $d \sim \mathcal{D}$, $\gamma_j \sim \mathcal{G}_j$, $r_1 \sim R_1$. Then
\begin{equation}
\begin{split}
\label{eq:tildedefs}
r_2 &\asymp \widetilde{C} := \frac{Q}{\mathcal{D} \mathcal{G}_1 \mathcal{G}_2 \mathcal{K}_1 \mathcal{K}_2 \Omega R_1}, \quad n \asymp \widetilde{M} := \frac{N}{\mathcal{G}_1 \mathcal{G}_2}, \quad |\omega e f| \asymp \widetilde{N} \in \left[1, \frac{\mathcal{V}_1 \mathcal{V}_2 Q^{2+\varepsilon}}{\mathcal{D}^2\mathcal{G}_1 \mathcal{G}_2 \mathcal{K}_1^2\mathcal{K}_2^2 EF}
\right], \\
\nu_1 \nu_2 g &\asymp \widetilde{R} := \mathcal{V}_1 \mathcal{V}_2 G, \quad r_1 \omega \asymp \widetilde{S} := R_1 \Omega.
\end{split}
\end{equation}

Ignoring the need to separate the variables in some of the smooth factors (which can be done by standard applications of integral transformations), the contribution of a given dyadic part to $\mathcal{S}(d\gamma\kappa \leq D)$ is essentially, for some $\kappa_j \sim \mathcal{K}_j, d \sim \mathcal{D}, \gamma_j \sim \mathcal{G}_j$, $1 \leq j \leq 2$ and with $g$ a smooth function satisfying the conditions of Lemma \ref{le:Klo1}, 
\begin{align}
\label{eq:abnamed}
\ll \mathcal{T} := Q^\varepsilon \frac{\sqrt{EF}}{\sqrt{GN}} \mathcal{K}_2 \mathcal{D} \mathcal{G}_2 \frac{1}{\widetilde{C}} \left| \sum_{\substack{\widetilde{r} \sim \widetilde{R} \\ \widetilde{s} \sim \widetilde{S} \\ (\widetilde{r}, \widetilde{s}) = 1}} \sum_{\substack{\widetilde{m} \sim \widetilde{M} \\ \widetilde{n} \sim \widetilde{N}}} a_{\widetilde{m}} b_{\widetilde{n}, \widetilde{r}, \widetilde{s}} \sum_{\substack{\widetilde{c} \sim \widetilde{C} \\ (\widetilde{c}, \widetilde{r}) = 1}} g(\widetilde{c}, \widetilde{m}, \widetilde{n}, \widetilde{r}, \widetilde{s}) S(\pm\widetilde{n}, \widetilde{m}\overline{\widetilde{r}}, \widetilde{s}\widetilde{c})\right|,
\end{align}
where 
\[
a_{\widetilde{m}} = \tau_3(\gamma_1 \gamma_2 \widetilde{m}) \mathbf{1}_{(\widetilde{m}, d) = 1}, \quad b_{\widetilde{n}, \widetilde{r}, \widetilde{s}} = \sumthree_{\substack{\widetilde{n} = \omega e f, \; \widetilde{s} = r_1 \omega, \; \widetilde{r} = \nu_1 \nu_2 g \\ (\omega, r_1 \widetilde{r}) = (e \gamma \kappa_2, r_1) = 1 \\ \nu_1 \mid d, \; \; \nu_2 \mid d\gamma_1 \gamma_2 \kappa_1 \kappa_2 \\ (\nu_1, \kappa_1 \kappa_2 r_1) = (\nu_2, r_1) = 1 \\ (g, d\gamma \kappa_1 \kappa_2 r_1) = 1 \\ w \sim \Omega, \; r_1 \sim R_1 \\ \nu_1 \sim \mathcal{V}_1, \; \nu_2 \sim \mathcal{V}_2}} \frac{\mu(\omega)}{\omega} \frac{\mu(\kappa_1 \gamma_1 r_1)}{r_1^2} \frac{\mu(\nu_1)}{\nu_1} \frac{\mu(\nu_2)}{\nu_2} V \Big ( \frac{g}{G} \Big ).
\]

\subsection{Applying the Kloosterman sum bounds}
Now we are ready to apply Lemma~\ref{le:Klo1} to~\eqref{eq:abnamed}. Notice first that

$$\Vert a_{\widetilde{m}} \Vert_2^2 \ll Q^\varepsilon \widetilde{M} \asymp Q^\varepsilon \frac{N}{\mathcal{G}_1 \mathcal{G}_2}, $$
\begin{align*}
\Vert b_{\widetilde{n}, \widetilde{r}, \widetilde{s}} \Vert_2^2 &\leq \frac{1}{(\mathcal{V}_1 \mathcal{V}_2)^2 \Omega^2 R_1^4} \sum_{\substack{\widetilde{n} \sim \widetilde{N}}} \sum_{\widetilde{r} \sim \widetilde{R}} \sum_{\widetilde{s} \sim \widetilde{S}} \Biggl|\sumthree_{\substack{\widetilde{n} = \omega e f, \, \widetilde{s} = r_1 \omega, \, \widetilde{r} = \nu_1 \nu_2 g \\ \nu_1 \mid d, \nu_2 \mid d\gamma_1 \gamma_2 \kappa_1 \kappa_2 \\ \omega \sim \Omega, r_1 \sim R_1 \\ \nu_1 \sim \mathcal{V}_1, \nu_2 \sim \mathcal{V}_2}} 1 \Biggr|^2 \ll \frac{Q^\varepsilon}{(\mathcal{V}_1 \mathcal{V}_2)^2 \Omega^2 R_1^4} \cdot \frac{\widetilde{N}\widetilde{S}}{\Omega} \cdot \frac{\widetilde{R}}{\mathcal{V}_1 \mathcal{V}_2} \\
&\ll Q^{\varepsilon} \frac{\widetilde{N} \widetilde{R} \widetilde{S}}{\Omega^3 R_1^4 (\mathcal{V}_1 \mathcal{V}_2)^3} \ll Q^{2\varepsilon} \frac{GQ^2}{EF} \cdot \frac{1}{\Omega^2 R_1^3 \mathcal{D}^2\mathcal{G}_1 \mathcal{G}_2 \mathcal{K}_1^2\mathcal{K}_2^2 \mathcal{V}_1 \mathcal{V}_2},
\end{align*}
and
$$  \widetilde{C}\widetilde{S}\sqrt{\widetilde{R}} \asymp \widetilde{C} R_1 \Omega \sqrt{\mathcal{V}_1 \mathcal{V}_2 G}. $$
Hence
\begin{align*}
  &\frac{\sqrt{EF}}{\sqrt{GN}} \mathcal{K}_2 \mathcal{D} \mathcal{G}_2 \frac{1}{\widetilde{C}} \Vert a_{\widetilde{m}} \Vert_2 \Vert b_{\widetilde{n}, \widetilde{r}, \widetilde{s}} \Vert_2 \widetilde{C}\widetilde{S}\sqrt{\widetilde{R}} \\
  &\ll Q^{2\varepsilon} \frac{\sqrt{EF}}{\sqrt{GN}} \mathcal{K}_2 \mathcal{D} \mathcal{G}_2 \cdot \sqrt{\frac{N}{\mathcal{G}_1 \mathcal{G}_2}} \cdot \sqrt{\frac{GQ^2}{EF} \cdot \frac{1}{\Omega^2 R_1^3 \mathcal{D}^2\mathcal{G}_1 \mathcal{G}_2 \mathcal{K}_1^2\mathcal{K}_2^2 \mathcal{V}_1 \mathcal{V}_2}} \cdot R_1 \Omega \cdot \sqrt{\mathcal{V}_1 \mathcal{V}_2 G} \\
  &\ll Q^{1+2\varepsilon} \sqrt{G} \cdot \frac{1}{\mathcal{G}_1 R_1^{1/2} \mathcal{K}_1}.
\end{align*}
Furthermore 
\begin{equation}
\label{eq:Xdef}
\widetilde{X} := \frac{\widetilde{C}\widetilde{S}\sqrt{\widetilde{R}}}{4\pi\sqrt{\widetilde{M}\widetilde{N}}} \gg_{\varepsilon} \frac{\frac{Q}{D\mathcal{G}_1 \mathcal{G}_2 \mathcal{K}_1 \mathcal{K}_2 \Omega R_1} \cdot R_1 \Omega \cdot \sqrt{\mathcal{V}_1 \mathcal{V}_2 G}}{\sqrt{\frac{N}{\mathcal{G}_1 \mathcal{G}_2} \cdot \frac{\mathcal{V}_1 \mathcal{V}_2 Q^{2+\varepsilon}}{D^2\mathcal{G}_1 \mathcal{G}_2 \mathcal{K}_1^2\mathcal{K}_2^2 EF}}} \gg_{\varepsilon} Q^{-\varepsilon} \sqrt{\frac{EFG}{N}} \geq 1
\end{equation}
for all $Q$ large enough, provided that $\varepsilon$ is choosen sufficiently small.

By Lemma~\ref{le:Klo1} we obtain that
\begin{equation*}
\mathcal{T} \ll  Q^{1+4\varepsilon} \sqrt{G} \cdot \frac{1}{\mathcal{G}_1 R_1^{1/2} \mathcal{K}_1} \cdot \sqrt{\widetilde{R}\widetilde{S}} \left(1+\sqrt{\frac{\widetilde{M}}{\widetilde{R}\widetilde{S}}}\right) \left(1+\sqrt{\frac{\widetilde{N}}{\widetilde{R}\widetilde{S}}}\right)\left(1+\frac{\widetilde{X}^2}{\left(1+\frac{\widetilde{R}\widetilde{S}}{\widetilde{M}}\right)^2 \left(1+\frac{\widetilde{R}\widetilde{S}}{\widetilde{N}}\right)}\right)^{7/64}.
\end{equation*}
Let us first note that by~\eqref{eq:tildedefs} and Assumptions of Proposition~\ref{prop:d<D}
\[
\frac{\widetilde{N}}{\widetilde{R}\widetilde{S}} \leq \frac{Q^{2+\varepsilon}}{EFG} \leq Q^{\delta_0+\varepsilon} \quad \text{and} \quad \frac{\widetilde{M}}{\widetilde{R}\widetilde{S}} \leq \frac{N}{\Omega G}.
\]
Furthermore, using also~\eqref{eq:Xdef},
\begin{align*}
\frac{\widetilde{X}^2}{\left(1+\frac{\widetilde{R}\widetilde{S}}{\widetilde{M}}\right)^2 \left(1+\frac{\widetilde{R}\widetilde{S}}{\widetilde{N}}\right)}  \ll \frac{1}{1+\left(\frac{\Omega G}{N}\right)^2}\cdot \frac{(\widetilde{C}\widetilde{S})^2 \widetilde{R}}{\widetilde{M} \widetilde{N}} \cdot  \frac{\widetilde{N}}{\widetilde{R}\widetilde{S}}  \ll \frac{N^2}{N^2+(\Omega G)^2} \cdot \frac{Q^2}{N} \ll  \frac{Q^2 N}{N^2+(\Omega G)^2}.
\end{align*}
Hence
\begin{align}
\label{eq:Tbound}
\begin{aligned}
\mathcal{T} &\ll Q^{1+ \delta_0 / 2 + 6\varepsilon} \sqrt{G} (\sqrt{\mathcal{V}_1 \mathcal{V}_2 \Omega G} + \sqrt{\mathcal{V}_1 \mathcal{V}_2 N}) \left(1+ \frac{Q^2 N}{N^2+(\Omega G)^2}\right)^{7/64} \\
&\ll Q^{1+ \delta_0 /2 +6\varepsilon} \sqrt{G} D (\sqrt{\Omega G} + \sqrt{N}) \left(1+ \frac{Q^2 N}{N^2+(\Omega G)^2}\right)^{7/64}.
\end{aligned}
\end{align}

We split into two cases.

\textbf{Case 1: $\Omega G \geq N$.} In this case we obtain the bound
\begin{align*}
\mathcal{T} & \ll Q^{1+ \delta_0 / 2 + 6\varepsilon} \sqrt{G} D \sqrt{\Omega G} \left(1+ \frac{Q^2 N}{(\Omega G)^2}\right)^{7/64} \\ & \ll D Q^{1+ \delta_0 / 2 +6\varepsilon} \left(G \sqrt{\Omega} + Q^{7/32}N^{7/64} G^{1-7/32} \Omega^{1/2-7/32}\right).
\end{align*}
By~\eqref{eq:omegabound}, $\Omega \ll Q^{1+\varepsilon}/E \leq Q^{1+\varepsilon}/G$, and furthermore $N \leq Q^{3+\varepsilon}/(EFG) \leq Q^\varepsilon (Q/G)^3$. Hence
\begin{align*}
\mathcal{T} &\ll D Q^{1+ \delta_0 /2 +7\varepsilon}\left(Q^{1/2}G^{1/2} + Q^{7/32+21/64+1/2-7/32} G^{-21/64+1-7/32-(1/2-7/32)}\right) \\
&= D Q^{1+  \delta_0 / 2 +7\varepsilon}\left(Q^{1/2}G^{1/2} + Q^{53/64} G^{11/64}\right).
\end{align*}
Note that the second term dominates as long as $G \leq Q$. Moreover $G \leq (EFG)^{1/3} \leq Q^{5/6+\delta'/3}$ and $\delta' < 1/2$, thus the second term always dominates, and thus 
\[
\mathcal{T} \ll D Q^{1+373/384+  \delta_0 / 2 +11 \delta'/192+7\varepsilon}.
\]

\textbf{Case 2: $\Omega G < N$.} In this case~\eqref{eq:Tbound} implies
\[
\mathcal{T} \ll Q^{1+  \delta_0 / 2 +6\varepsilon} \sqrt{G} D \sqrt{N} \left(1+ \frac{Q^2}{N}\right)^{7/64}. 
\]
Now $N \leq Q^2$, so 
\[
\mathcal{T} \ll D Q^{1+  \delta_0 / 2 +6\varepsilon} \cdot G^{1/2} Q^{7/32} N^{1/2-7/64}.
\]
Using that $GN \leq Q^{3+\varepsilon}/(EF)$, we see that
\[
\mathcal{T} \ll D Q^{1+  \delta_0 / 2 +7\varepsilon} G^{7/64} Q^{3/2-7/64}/(EF)^{1/2-7/64}.
\]
This is largest when $E$ and $F$ are as small as possible so that $E=F=G=Q^{(2-\delta_0)/3}$, and so
\[
\mathcal{T} \ll D Q^{5/2-7/64+  \delta_0 / 2 + 7\varepsilon} Q^{\frac{2-\delta_0}{3}(21/64-1)} \ll D Q^{1+181/192+ \delta_0/2 + 43 \delta_0 / 192 +7\varepsilon}.
\]
Combining the two cases, we obtain
\[
\mathcal{S}(d\gamma \kappa \leq D) \ll Q^{2+7\varepsilon} \cdot Q^{ \delta_0 / 2 } \cdot \left(D Q^{-11/384+11 \delta'/192} + DQ^{-11/192+43\delta_0/192} \right),
\]
and Proposition~\ref{prop:d<D} follows by adjusting $\varepsilon$.


\section{Acknowledgements}
V.C. acknowledges support from a Simons Travel grant 709707 and NSF grant DMS-2101806. X.L. acknowledges support from Simons Travel Grants 524790 and 962494 and NSF grant DMS-2302672. K.M. was partially supported by Academy of Finland grant no. 285894. M.R. is supported by NSF grant DMS-2401106.


\end{document}